\documentclass[aop,preprint]{imsart}

\RequirePackage[OT1]{fontenc}
\RequirePackage{amsthm,amsmath,natbib}
\RequirePackage[colorlinks=false]{hyperref}
\RequirePackage{hypernat}

\numberwithin{equation}{section}

\usepackage{enumerate}
\usepackage{amssymb}
\usepackage{comment}
\usepackage{appendix}
\usepackage{graphicx}
\usepackage{xr}
\usepackage{subfigure}

\input xy
\xyoption{all}

\startlocaldefs

\newtheorem{theorem}{Theorem}[section]
\newtheorem{lemma}[theorem]{Lemma}
\newtheorem{proposition}[theorem]{Proposition}

\newtheorem{remark}[theorem]{Remark}

\newcommand{\B}{\mathbf{B}}
\newcommand{\C}{\mathbf{C}}

\newcommand{\E}{\mathbf{E}}

\newcommand{\h}{\mathbf{H}}

\newcommand{\Z}{\mathbf{Z}}
\newcommand{\p}{\mathbf{P}}
\newcommand{\Q}{\mathbf{Q}}
\newcommand{\R}{\mathbf{R}}

\newcommand{\CA}{\mathcal {A}}
\newcommand{\CB}{\mathcal {B}}
\newcommand{\CC}{\mathcal {C}}
\newcommand{\CD}{\mathcal {D}}
\newcommand{\CE}{\mathcal {E}}
\newcommand{\CF}{\mathcal {F}}

\newcommand{\CK}{\mathcal {K}}
\newcommand{\CL}{\mathcal {L}}

\newcommand{\CQ}{\mathcal {Q}}

\newcommand{\CU}{\mathcal {U}}
\newcommand{\CV}{\mathcal {V}}

\newcommand{\CZ}{\mathcal {Z}}

\newcommand{\CH}{\mathcal {H}}

\newcommand{\ext}{{\rm ext}}

\newcommand{\dist}{{\rm dist}}

\newcommand{\diam}{{\rm diam}}
\newcommand{\var}{{\rm Var}}
\newcommand{\im}{{\rm Im}}

\newcommand{\cov}{{\rm Cov}}

\newcommand{\SLE}{{\rm SLE}}

\newcommand{\internal}{{\rm int}}

\newcommand{\one}{\mathbf{1}}

\newcommand{\ol}{\overline}

\newcommand{\wh}{\widehat}
\newcommand{\wt}{\widetilde}

\endlocaldefs

\begin{document}

\title{Universality for $\SLE(4)$}
\runtitle{Universality for $\SLE(4)$}

\begin{aug}


\author{\fnms{Jason} \snm{Miller}\thanksref{t2}\ead[label=e2]{jmiller@math.stanford.edu}}

\thankstext{t2}{Research supported in part by NSF grants DMS-0406042 and DMS-0806211.}

\runauthor{Jason Miller}

\affiliation{Stanford University}
\address{Stanford University\\
Department of Mathematics\\
Stanford, CA 94305\\
\printead{e2}\\
\phantom{E-mail:\ jmiller@math.stanford.edu}}

\end{aug}
\date{\today}

\begin{abstract}
We resolve a conjecture of Sheffield that $\SLE(4)$, a conformally invariant random curve, is the universal limit of the chordal zero-height contours of random surfaces with isotropic, uniformly convex potentials.  Specifically, we study the \emph{Ginzburg-Landau $\nabla \phi$ interface model} or \emph{anharmonic crystal} on $D_n = D \cap \tfrac{1}{n} \Z^2$ for $D \subseteq \C$ a bounded, simply connected Jordan domain with smooth boundary.  This is the massless field with Hamiltonian $\CH(h) = \sum_{x \sim y} \CV(h(x) - h(y))$ with $\CV$ symmetric and uniformly convex and $h(x) = \phi(x)$ for $x \in \partial D_n$, $\phi \colon \partial D_n \to \R$ a given function.   We show that the macroscopic chordal contours of $h$ are asymptotically described by $\SLE(4)$ for appropriately chosen $\phi$.
\end{abstract} 

\maketitle

\section{Introduction}

The idea of statistical mechanics is to model physical systems by describing them probabilistically at microscopic scales and then studying their macroscopic behavior.  Many lattice-based planar models at criticality are believed to have scaling limits which are invariant under conformal symmetries, a reflection of the heuristic that the asymptotic behavior at criticality should be independent of the choice of the underlying lattice.  The realizations of these models tend to organize themselves into large clusters separated from each other by thin interfaces which, in turn, have proven to be interesting objects to study in the scaling limit.  The last decade has brought a number of rather exciting developments in this direction, primarily due to the introduction of SLE \cite{S01}, a one-parameter family of conformally invariant random curves which are conjectured to describe the limiting interfaces in many models.  This has now been proved rigorously in several special cases: loop-erased random walk and the uniform spanning tree \cite{LSW04}, chordal level lines of the discrete Gaussian free field \cite{SS09}, the harmonic explorer \cite{SS05}, the Ising model on the square lattice \cite{S07} and on isoradial graphs \cite{CS10U}, and percolation on the triangular lattice \cite{S01, CN06}.

One of the core principles of statistical mechanics is that of \emph{universality}: the exact microscopic specification of a model should not affect its macroscopic behavior.  There are two ways in which universality can arise in this context: stability of the limit with respect to \emph{changes to the lattice} and, the stronger notion, with respect to \emph{changes to the Hamiltonian}.  The results of \cite{LSW04, SS09, SS05, CS10U} fall into the first category.  Roughly, this follows in \cite{LSW04, SS09, SS05} since underlying the conformal invariance of the models described in these works is the convergence of simple random walk to Brownian motion, a classical result which is lattice independent and is in fact true in much greater generality.  Extending Smirnov's results on the Ising model \cite{S07} beyond the square lattice is much more challenging \cite{CS10DCI, CS10U} and it is a well-known open problem to extend the results of \cite{S01, CN06} to other lattices.

The purpose of this work is to prove the conformal invariance of limiting interfaces for a large class of random surface models that is \emph{stable with respect to non-perturbative changes to the Hamiltonian}, of which there is no prior example.

\subsection{Main Results}

Specifically, we study the massless field on $D_n = D \cap \tfrac{1}{n} \Z^2$ with Hamiltonian $\CH(h^n) = \sum_{b \in D_n^*} \CV(\nabla h^n(b))$.  Here, $D \subseteq \C$ is a bounded, simply connected Jordan domain with smooth boundary.  The sum is over the set $D_n^*$ of edges in the induced subgraph of $\tfrac{1}{n} \Z^2$ with vertices in $D$ and $\nabla h^n(b) = h^n(y) - h^n(x)$ denotes the discrete gradient of $h^n$ across the oriented bond $b=(x,y)$.  We assume that $h^n(x) = \phi^n(x)$ when $x \in \partial D_n$ and $\phi^n \colon \partial D_n \to \R$ is a given bounded function.  We consider a general interaction $\CV \in C^2(\R)$ which is assumed only to satisfy:
\begin{enumerate}
\item $\CV(x) = \CV(-x)$ (symmetry),
\item $0 < a_\CV \leq \CV''(x) \leq A_\CV < \infty$ (uniform convexity), and
\item $\CV''$ is $L$-Lipschitz.
\end{enumerate}
This is the so-called \emph{Ginzburg-Landau $\nabla \phi$ effective interface (GL) model}, also known as the \emph{anharmonic crystal}.  The first condition is a reflection of the hypothesis that our bonds are undirected. The role played by the second and third conditions is technical.  Note that we can assume without loss of generality that $\CV(0) = 0$.   The variables $h^n(x)$ represent the heights of a random surface which serves as a model of an interface separating two pure phases.  The simplest case is $\CV(x) = \tfrac{1}{2} x^2$, which corresponds to the so-called discrete Gaussian free field, but our hypotheses allow for much more exotic choices such as $\CV(x) = 4x^2 + \cos(x) + e^{-x^2}$.
 
The purpose of this work is to determine the limiting law of the \emph{chordal zero-height contours} of $h^n$.  To keep the article from being unnecessarily complicated, we will select our boundary conditions in such a way that there is only a single such curve.  That is, we fix $x,y \in \partial D$ distinct and let $x_n, y_n$ be points in $\partial D_n$ with minimal distance to $x,y$, respectively.  Denote by $\partial_+^n$ the part of $\partial D_n$ connecting $x_n$ to $y_n$ in the clockwise direction and $\partial_-^n = \partial D_n \setminus \partial_+^n$.  Let $x_n^*,y_n^*$ be the edges containing $x_n,y_n$, respectively, which connect $\partial_+^n$ to $\partial_-^n$.  Suppose that $h^n$ has the law of the GL model on $D_n$ with boundary conditions $h^n|_{\partial D_n} \equiv \phi^n$ where $\phi^n|_{\partial_+^n} \in (0,\infty)$ and $\phi^n|_{\partial_-^n} \in (-\infty,0)$.  Let $\gamma^n$ be the unique path in $D_n^*$ connecting $x_n^*$ to $y_n^*$ which has the property that for each $t$, $\gamma^n(t)$ is the first edge $\{u,v\}$ on the square adjacent to $\gamma^n(t-1)$ in the clockwise direction such that $h^n(u) > 0$ and $h^n(v) < 0$ if $\gamma^n(t-1)$ is oriented horizontally and in the counterclockwise direction if $\gamma^n(t-1)$ is oriented vertically (see Figure \ref{fig::turning_rule}).

 \begin{theorem}
 \label{intro::thm::sle_convergence}
There exists $\lambda \in (0,\infty)$ depending only on $\CV$ such that the following is true.  If $\phi^n|_{\partial_+^n} = \lambda$ and $\phi^n|_{\partial_-^n} = -\lambda$, then up to reparameterization, the piecewise linear interpolation of $\gamma^n$ converges in distribution with respect to the uniform topology to an $\SLE(4)$ curve connecting $x$ to $y$ in $D$.
 \end{theorem}

   \begin{figure}
     \centering
         \includegraphics[width=1\textwidth]{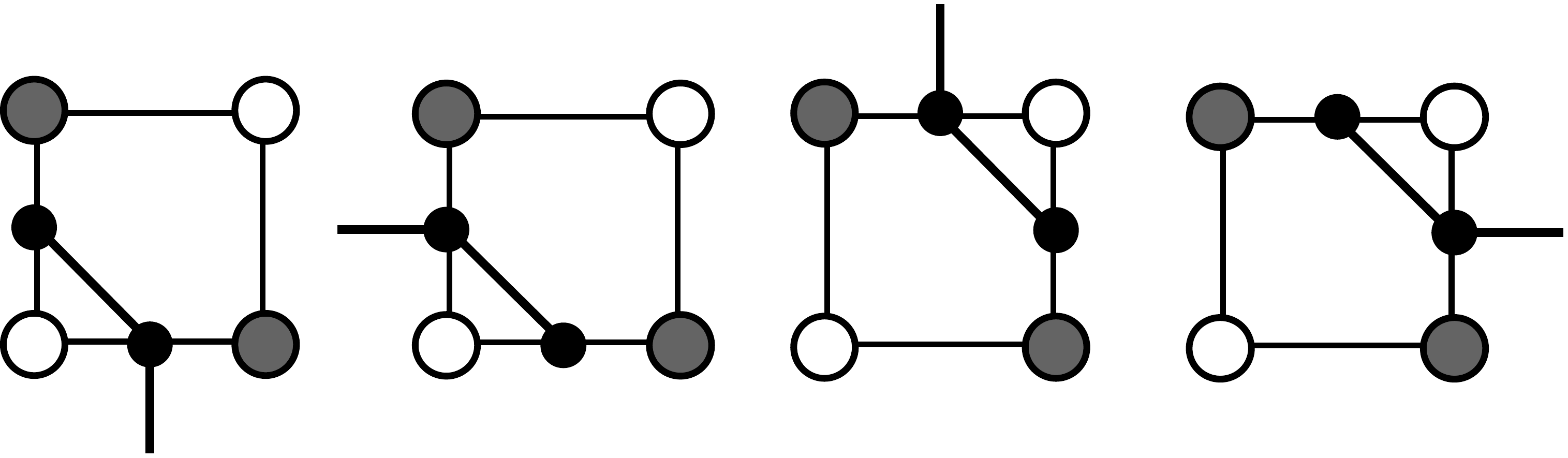}
          \caption{We must a fix a convention which dictates the direction in which $\gamma^n$ turns, as ambiguities may arise.  The white (resp. gray) disks at the boundary of a square indicate sites at which the field is positive (resp. negative).  In each of the situations depicted above, there are two possible directions in which $\gamma^n$ can turn and still preserve the constraint that the field is positive (resp. negative) on the left (resp. right) side of $\gamma^n$.  Our convention is that on horizontal (resp. vertical) dual edges, $\gamma^n$ goes to the first dual edge in the clockwise direction (resp. counterclockwise) where there is a sign change.  \label{fig::turning_rule}}
 \end{figure}

This is a resolution of the following conjecture due to Sheffield for the GL model, which constitutes a large and important special case:\bigskip

\noindent{\bf Problem 10.1.3 \cite{SHE_RS06}:}
\emph{``If a height function $\phi$ on $\Z^2$ is interpolated to a function $\ol{\phi}$ which is continuous and piecewise linear on simplices, then the level sets $C_a$, given by $\ol{\phi}^{-1}(a)$, for $a \in \R$ are unions of disjoint cycles.  What do the typical ``large'' cycles look like when $\Phi$ is simply attractive and sampled from a rough gradient phase?  The answer is given in \cite{SS09} in the simplest case of quadratic nearest neighbor potentials - in this case, ``the scaling limit'' of the loops as the mesh size gets finer is well defined, and the limiting loops look locally like a variant of the Schramm-Loewner evolution with parameter $\kappa = 4$.  We conjecture that this limit is universal - i.e., that the level sets have the same limiting law for all simply attractive potentials in a rough phase.''}\bigskip

In Sheffield's terminology, $\Phi$ is said to be a \emph{simply attractive potential}, i.e. a convex, nearest-neighbor, difference potential, if $\Phi$ takes the form $\sum_{b \in D_n^*} \CV_b(\nabla h(b))$ where for each $b \in D_n^*$, $\CV_b$ is a convex function.  $\Phi$ is said to be isotropic if $\CV_b = \CV$, i.e. does not depend on $b$.  Thus the Hamiltonian for the GL model is an \emph{isotropic simply attractive potential} which is \emph{uniformly convex}.

The GL has been the subject of much recent work.  Gibbs states were classified by Funaki and Spohn in \cite{FS97}, where they also study macroscopic dynamics.  A large deviations principle for the surface shape with zero boundary conditions but in the presence of a chemical potential was established by Deuschel, Giacomin, and Ioffe in \cite{DGI00} and Funaki and Sakagawa in \cite{FS04} extend this result to the case of non-zero boundary conditions using the contraction principle.  The behavior of the maximum is studied by Deuschel and Giacomin in \cite{DG00} and by Deuschel and Nishikawa in \cite{DN07} in the case of Langevin dynamics.  Central limit theorems for Gibbs states were proved by Naddaf and Spencer in \cite{NS97} for zero tilt and later by Giacomin, Olla, and Spohn for general tilt and dynamics in \cite{GOS01}.  The CLT on finite domains as well as an explicit representation for the limiting covariance was obtained in \cite{M10}.

We remark that it is possible to weaken significantly the restrictions on the boundary conditions.  As shown in a forthcoming work, the limit is $\SLE(4;\rho)$ in the piecewise constant case and a ``continuum version'' of $\SLE(4;\rho)$ for $C^1$ boundary conditions.  We also remark that the reason for the convention dictating the direction in which $\gamma^n$ turns in Theorem \ref{intro::thm::sle_convergence} is that with this choice the law of $\gamma^n$ is invariant with respect to the transformation given by exchanging the signs of the boundary conditions.  Moreover, this scheme yields a path which is equivalent to that which arises from the triangulation method described in \cite[Section 1.5]{SS09}, in particular by adding to $\Z^2$ the edges of the form $\{(x,y),(x+1,y-1)\}$.  There are a number of other natural local rules, for example for the curve to move to the first edge in the clockwise direction where the height field has a sign change.  This to a law leads law which is not invariant with respect to this transformation and the limit is no longer $\SLE(4)$, but rather some $\SLE(4;\rho)$.

\subsection{Overview of $\SLE$}

The Schramm-Loewner evolution ($\SLE$) is a one-parameter family of conformally invariant random curves, introduced by Oded Schramm in \cite{S0} as a candidate for, and later proved to be, the scaling limit of loop erased random walk \cite{LSW04} and the interfaces in critical percolation \cite{S01, CN06}.  $\SLE$ comes in two different flavors: radial and chordal.  The former describes a curve connecting a point on the boundary of a domain to its interior and the latter a curve connecting two points on the boundary.  We will restrict our discussion to the latter case since it is the one relevant for this article.  We remark that there are many excellent surveys on $\SLE$, for example \cite{LAW05, W03}, to which we direct the reader interested in a detailed introduction to the subject.

Chordal $\SLE(\kappa)$ on the upper half-plane $\h=\{ z \in \C : \im(z) > 0\}$ connecting $0$ to $\infty$ is easiest to describe first in terms of a family of \emph{random conformal maps} which are given as the solution to the Loewner ODE
\[ \partial_t g_t(z) = \frac{2}{g_t(z) - W(t)},\ \ g_0(z) = z.\]
Here, $W = \sqrt{\kappa} B$ where $B$ is a standard Brownian motion.  The domain of $g_t$ is $\h_t = \{z \in \h : \tau(z) > t\}$ where $\tau(z) = \inf\{ t \geq 0 : \im(g_t(z)) = 0\}$.  By work of Rohde and Schramm \cite{RS05}, $\h_t$ arises as the unbounded connected component of a random curve $\gamma$ on $[0,t]$, the $\SLE$ \emph{trace}.  This is what allows us to refer to $\SLE$ as a curve.  $\SLE(\kappa)$ connecting boundary points $x$ and $y$ of a simply connected Jordan domain is defined by applying a conformal transformation $\varphi \colon \h \to D$ to $\SLE(\kappa)$ on $\h$ sending $0$ to $x$ and $\infty$ to $y$.  Of course, this leaves one degree of freedom in the choice of $\varphi$, so this only defines $\SLE(\kappa)$ on $D$ up to reparameterization.

The following two properties characterize chordal $\SLE(\kappa)$:
\begin{enumerate}
  \item {\it conformal invariance}: If $D,D'$ are simply connected Jordan domains with marked boundary points $x,y \in \partial D$ and $x',y' \in \partial D'$ and $\varphi \colon D \to D'$ is a conformal map taking $x,y$ to $x',y'$, respectively, then the image of a chordal $\SLE(\kappa)$ connecting $x$ to $y$ in $D$ under $\varphi$ is a chordal $\SLE(\kappa)$ connecting $x'$ to $y'$ in $D'$
  \item {\it domain Markov property}: if $\gamma$ is the trace of a chordal $\SLE(\kappa)$ from $x$ to $y$ in $D$, then conditional on $\gamma[0,s]$, $\gamma$ has the law of a chordal $\SLE(\kappa)$ from $\gamma(s)$ to $y$ in the connected component of $D \setminus \gamma[0,s]$ containing $y$.
\end{enumerate}

Many families of random curves arising from interfaces of two-dimensional lattice models are believed to satisfy these two properties in the scaling limit, hence converge to some $\SLE(\kappa)$.  While there are many conjectures, proving such convergence is extremely challenging and, as we mentioned earlier, rigorous proofs are available only in a few isolated cases.  Establishing a strong form of universality has proved to be particularly difficult since the arguments in these works are rather delicate and depend critically on the microscopic specification of the model.  For example, while the interfaces of percolation on the \emph{triangular lattice} have been shown to converge to $\SLE(6)$ \cite{CN06, S01} the combinatorial argument of \cite{S01} is not applicable for any other lattice.  Even the seemingly simple extension of the results of \cite{S01} to percolation on the square lattice has been open for much of the past decade. 

\subsection{Strategy of Proof}

The strategy for proving convergence to $\SLE$ involves several steps, the most difficult and important of which is to find an observable of the underlying model and prove it has a conformally invariant limit which is also a martingale.  We now describe the observable used in this article.  Suppose that $\gamma$ is the trace of an $\SLE(4)$ curve in $\h$ from $0$ to $\infty$ and $(g_t)$ is the corresponding family of conformal maps.  Let $f_t \colon \h_t \to \R$ be the function harmonic on $\h_t$ with boundary values $0$ on the right side of $\gamma$ and $(0,\infty)$ and $1$ on the left side of $\gamma$ and $(-\infty,0)$.  We can express $f_t$ explicitly in terms of $g_t$ as follows:
\[ f_t(z) = \frac{1}{\pi} \im(\log(g_t(z))).\]
A calculation of the Ito derivative of the right side shows that $f_t(z)$ evolves as a martingale in time for $z$ fixed precisely because $\kappa = 4$.  This property characterizes $\SLE(4)$ among random simple curves \cite{SS05, SS09}. 

The key of the proof of Theorem \ref{intro::thm::sle_convergence} is to show that this approximately holds for the corresponding interface of the GL model, provided $\lambda > 0$ is chosen appropriately.  Specifically, suppose that $D, D_n, x_n, y_n, \gamma^n$ are as in the statement of Theorem \ref{intro::thm::sle_convergence} and that $\CF_t^n = \sigma( \gamma_s^n : s \leq t)$.  Let $D_n(\gamma,t,\epsilon) = \{ x \in D_n : \dist(x, \partial D_n \cup \gamma^n([0,t])) \geq \epsilon\}$.

\begin{theorem}
\label{intro::thm::approximate_martingale}
There exists $\lambda \in (0,\infty)$ depending only on $\CV$ such that the following is true.  Let $f_t^n$ be the function on $D_n \setminus \gamma^n([0,t])$ which is discrete harmonic in the interior and has boundary values $\lambda$ on $\partial_+^n$ and the left side of $\gamma^n([0,t])$ and $-\lambda$ on $\partial_-^n$ and the right side of $\gamma^n([0,t])$.  Also let $M_t^n(x) = \E[ h^n(x) | \CF_t^n]$ and $\CE_t^n(\epsilon) = \max\{|f_t^n(x) - M_t^n(x)| : x \in D_n(\gamma,t,\epsilon)\}$.  For every $\epsilon, \delta > 0$ there exists $n$ sufficiently large such that for every $\CF_t^n$ stopping time $\tau$ we have
\[ \p[ \CE_\tau^n(\epsilon) \geq \delta] \leq \delta.\] 
\end{theorem}
\noindent This, in particular, implies that $f_t^n(x)$ is an approximate martingale.

\subsubsection*{Main Steps}

Theorem \ref{intro::thm::approximate_martingale} should be thought of as a law of large numbers for the conditional mean of the height given the realization of the path up to any stopping time.  Its proof consists of several important steps.  First, Theorem \ref{harm::thm::mean_harmonic} of \cite{M10} implies $M_\tau^n(x)$ is with high probability uniformly close to the discrete harmonic extension of its boundary values from $\partial D_n(\gamma,t,n^{-\epsilon})$ to $D_n(\gamma,t,n^{-\epsilon})$ provided $\epsilon = \epsilon(\CV) > 0$ is sufficiently small.  In particular, $M_\tau^n(x)$ is approximately discrete harmonic mesoscopically close to $\gamma^n[0,\tau]$ relative to the Euclidean metric in $\R^2$.  With respect to the graph metric, in which the distance is given by the number of edges in the shortest path, the distance at which this \emph{a priori} estimate holds from the path is unbounded in $n$.  Using the results of Sections \ref{sec::expectation}, \ref{sec::hic} we will prove that this estimate can be boosted further to get the approximate harmonicity of $M_t^n(x)$ up to finite distances from $\gamma^n[0,\tau]$ in the graph metric:

\begin{theorem}
\label{intro::thm::harmonic_up_to_boundary}
Fix $\ol{\Lambda} > 0$ and suppose that $h^n$ has the law of the GL model on $D_n$ with boundary conditions $\phi$ satisfying $\phi|_{\partial_+^n} \in (0,\ol{\Lambda})$ and $\phi|_{\partial_-^n} \in (-\ol{\Lambda},0)$.  For every $\delta > 0$ there exists $r_0 = r_0(\ol{\Lambda},\delta) > 0$ such that the following is true.  Let $r > r_0$ and $\psi_t^n$ be the function on $D_n(\gamma,t,r n^{-1})$ which satisfies the boundary value problem
\[ (\Delta \psi_t^n)|_{D_n(\gamma,t,rn^{-1})} \equiv 0,\ \ \psi_t^n|_{\partial D_n(\gamma,t,rn^{-1})} \equiv M_t^n\]
where $\Delta$ denotes the discrete Laplacian.
Let $\CD_t^n(r) = \max\{|\psi_t^n(x) - M_t^n(x)| : x \in D_n(\gamma,t,r n^{-1})\}$.  For every $\CF_t^n$ stopping time $\tau$, we have $\E[ \CD_\tau^n(r)] \leq  \delta.$
\end{theorem}

This reduces the proof of Theorem \ref{intro::thm::approximate_martingale} to showing that the boundary values of $M_t^n$ very close to $\gamma^n[0,t]$ averaged according to harmonic measure are approximately constant.  Specifically, the latter task requires two estimates:
\begin{enumerate}
\item Correlation decay of the boundary values of $M_t^n$ at points which are far away from each other.
\item The law of $M_t^n$ at a point on $\gamma^n$ sampled from harmonic measure has a scaling limit as $n \to \infty$.
\end{enumerate}
Step (1) is a consequence of Proposition \ref{hic::prop::hic}, which we will not restate here, and step (2) comes from the following theorem:

\begin{figure}
\includegraphics[width=90mm]{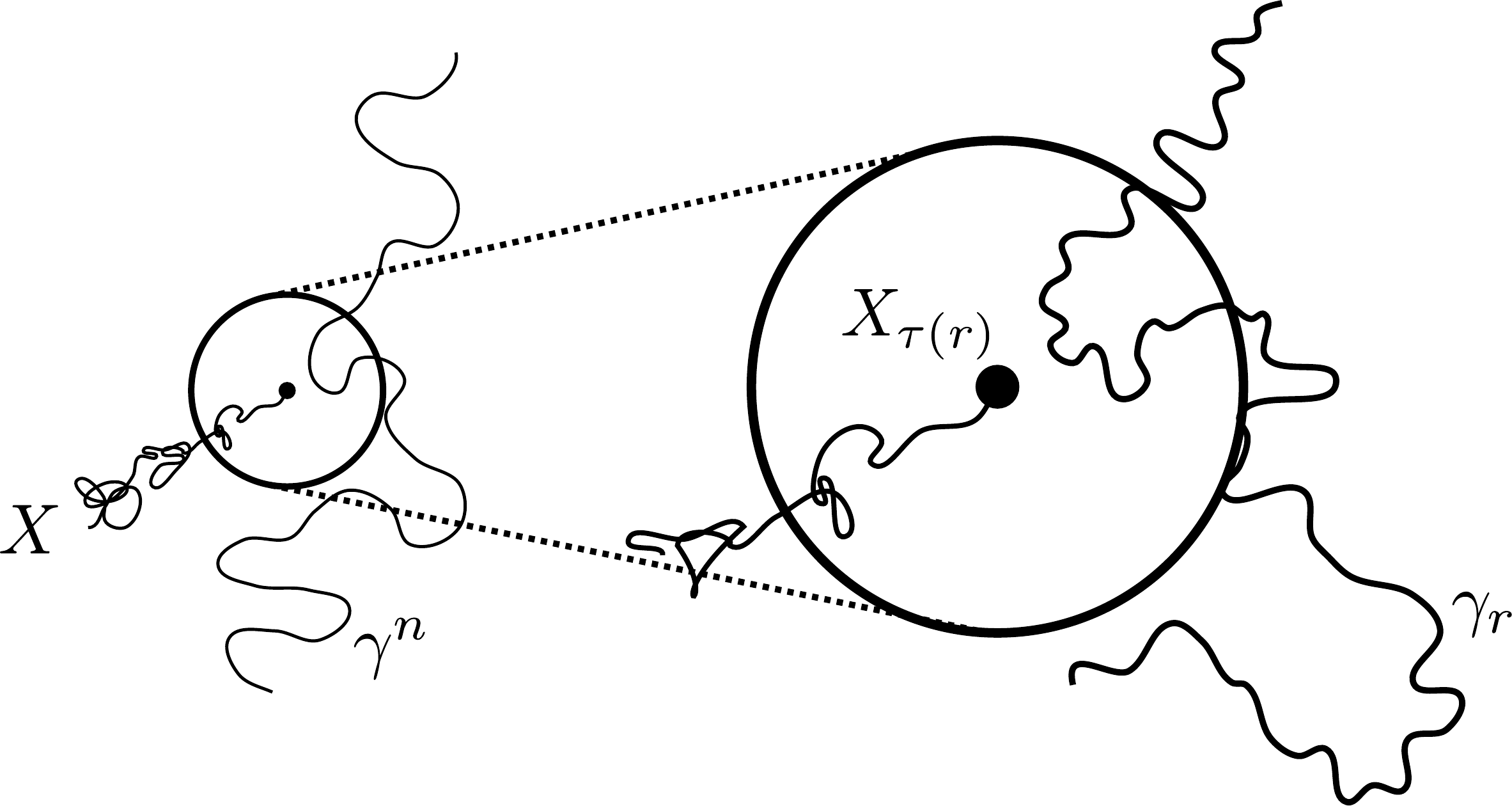}
\caption{We show that for each $r \geq 0$, the interface $\gamma^n$ has a scaling limit when viewed from the perspective of $X_{\tau(r)}$, where $X$ is a simple random walk and $\tau(r)$ is the first time $X$ comes within distance $rn^{-1}$ of $\gamma^n$.  This expands on the Schramm-Sheffield approach, where it is only necessary to construct the scaling limit for $r = 0$.}
\end{figure}

\begin{theorem}
\label{intro::thm::scaling_limit}
For each $r \geq 0$ there exists a unique measure $\nu_r$ on bi-infinite simple paths in $(\Z^2)^*$ which come exactly within distance $r$ of $0$ such that the following is true.  Suppose that $\tau$ is an $\CF_t^n$ stopping time, $X$ is a simple random walk on $\tfrac{1}{n} \Z^2$ initialized in $D_n(\gamma,\tau,\epsilon)$ independent of $h^n$, and $\tau(r)$ is the first time that $X$ gets within distance $rn^{-1}$ of $\partial (D_n \setminus \gamma^n[0,\tau])$.  Let $\gamma_+^n$ denote the positive side of $\gamma_n$ and $\dist(\cdot,A)$ denote the distance in the internal metric of $D_n \setminus \gamma^n[0,\tau]$ to $A$.  Conditional on both
\begin{enumerate}
 \item $\dist(X_{\tau(r)}, \gamma_+^n) = r n^{-1}$ and
 \item $\dist(X_{\tau(r)},\partial D_n \cup \{\gamma^n(\tau)\}) \geq S n^{-1}$,
\end{enumerate}
let $\nu_{n,r,R,S} $ be the probability on simple paths in $(\Z^2)^*$ induced by the law of $B(0,R) \cap n(\gamma^n[0,\tau] - X_{\tau(r)})$.  For every $\delta, r,R> 0$, there exists $S_0$ such that $S \geq S_0$ implies
\[ \| \nu_{n,r,R,S} - \nu_r|_{B(0,R)} \|_{TV} \leq \delta\]
for all $n$ large enough, where $\nu_r|_{B(0,R)}$ denotes the law of $\gamma \cap B(0,R)$ for $\gamma \sim \nu_r$.
\end{theorem}

Theorem \ref{intro::thm::scaling_limit} is a mesoscopic version of \cite[Theorem 3.21]{SS09} applicable for the GL model.  Its proof is based on the idea that the geometry of $\gamma^n$ is spatially mixing and has two main steps which, roughly, are:

\begin{enumerate}
\item The geometry of zero height interfaces of $h^n$ near a point $x_0$ is approximately independent of the geometry of $\gamma^n$ away from $x_0$ (see Section \ref{sec::ni}),
\item With high probability, $\gamma^n$ will hook-up with a large zero-height interface passing through a point $x_0$ conditional upon $\gamma^n$ passing near $x_0$.
\end{enumerate}

Step (1) is model specific and requires a challenging argument in the general GL setting.  On the other hand, we are able to reuse many of the high level ideas behind step (2) from \cite{SS09} in our setting thanks to their generality.

We now explain how to prove Theorem \ref{intro::thm::approximate_martingale} from Theorems \ref{intro::thm::harmonic_up_to_boundary} and \ref{intro::thm::scaling_limit}.  Let $\lambda_r \in (0,\infty)$ be the constant given by the following procedure.
\begin{enumerate}
  \item Sample $\gamma_r \sim \nu_r$, let $V_+(\gamma_r)$ be the sites adjacent to $\gamma_r$ which are in the same connected component of $\Z^2 \setminus \gamma_r$ as $0$, and $V_-(\gamma_r)$ the set of all other sites adjacent to $\gamma_r$.
  \item Conditional on $\gamma_r$, we let $h_r$ have the law of the GL model on $\Z^2$ conditional on $\{ h_r(x) > 0 : x \in V_+(\gamma_r)\}$ and $\{ h_r(x) < 0 : x \in V_-(\gamma_r)\}$. 
   \item  Set $\lambda_r = \E[h_r(0)]$.
\end{enumerate}
For each $\delta > 0$, we can choose $r$ sufficiently large so that $M_\tau^n(x)$ is with high probability uniformly close to the harmonic function in $D_n(\gamma,\tau,r n^{-1})$ with boundary values $\lambda_r$ (resp. $-\lambda_r$) on the left (resp. right) side of $\gamma^n$.  Theorem \ref{intro::thm::approximate_martingale} follows by showing $\lambda = \lim_{r \to \infty} \lambda_r$ exists and $\lambda \in (0,\infty)$.

Deducing convergence to $\SLE$ in the Caratheodory topology from an estimate such as Theorem \ref{intro::thm::approximate_martingale} follows a procedure which by now is standard, see \cite{LSW04, SS05, SS09}.  The model specific arguments of \cite{SS09} used to promote the convergence to the uniform topology work verbatim in our setting, however, are unnecessary thanks to the time symmetry of our problem and recent results of Sheffield and Sun \cite{SS10}.

We remark that the existence and positivity of the limit of $(\lambda_r)$ is one of the crucial points of the proof.  Indeed, it follows from the work of Kenyon \cite{K01, K00} that the mean of the height function of the double dimer model also converges to the harmonic extension of its boundary values.  Using the same method of proof employed here to deduce the convergence of the chordal interfaces of the double dimer model seems to break down \cite{SP10}, though.  The technical difficulty in that setting is the estimate of harmonicity of the mean from \cite{K00} requires the boundary to satisfy certain geometric conditions which need not hold for the zero-height interfaces.  Thus, just as in our case, one does not have an estimate of harmonicity of the mean which holds all of the way up to the interfaces.  In particular, it appears to be very difficult to show that the mean height remains uniformly positive (resp. negative) on the positive (resp. negative) sides of the interfaces.

\subsection{Outline}  The rest of the article is structured as follows.  In Section \ref{sec::setup_notation}, we will fix some notation which will be used repeatedly throughout.  The purpose of Section \ref{sec::cond_dynam} is to develop the theory of dynamic coupling for the GL model under the presence of conditioning in addition to collecting some useful results on stochastic domination.  In Section \ref{sec::expectation} we will prove a few technical estimates which allow us to control the moments of the conditioned field near the interface.  The main result of Section \ref{sec::hic} is that the law of the field near a particular point $x_0$ on the interface does not depend strongly on the exact geometry of the interface far away from $x_0$.  This will allow us to deduce that the mean height is strictly negative near the negative side of the interface and vice versa on the positive side, which in turn implies $\lambda_r$ is uniformly positive in $r$.  We will also prove Theorem \ref{intro::thm::harmonic_up_to_boundary} and deduce from it that $\lambda_r$ is uniformly bounded from both $0$ and $\infty$ in $r$.  In Section \ref{sec::ni}, we will show that the geometry of the interface near a point $x_0$ is approximately independent from its precise geometry far away from $x_0$.  This is the key part of Theorem \ref{intro::thm::scaling_limit}.  We will explain the proofs of Theorems \ref{intro::thm::scaling_limit} and \ref{intro::thm::approximate_martingale} in Section \ref{sec::bvi}.

This article is the second in a series of two.  The first is a prerequisite for this one and we will cite it heavily throughout.

\section{Setup and Notation}
\label{sec::setup_notation}

Throughout the rest of this article, we will frequently make use of the following two assumptions:
\begin{enumerate}
 \item[($\partial$)] \label{assump::boundary}  Suppose that $D \subseteq \Z^2$ with $\diam(D) = R$.  Assume that $\ol{\Lambda} > 0$ and $\psi \in \B_{\ol{\Lambda}}(D) \equiv \{ \phi \colon \partial D \to \R : \| \phi \|_\infty \leq \ol{\Lambda}\}$.
 \item[($C$)] \label{assump::conditioning} Let $V,V_+, V_-$ be non-empty disjoint subsets of $D$ and let $U = \partial D \cup V_- \cup V_+ \cup V$.  Suppose that for every $x \in V_+$ there exists $y \in V_- \cup V \cup \partial D$ with $|x-y| \leq 2$ and vice-versa.  Finally, assume that $a,b \colon D \to \R$ satisfy
\begin{enumerate}
 \item[(1)] $a(x) = -\infty$, $b(x) = \infty$ for $x \notin U$,
 \item[(2)] $a(x) \geq - \ol{\Lambda}$, $b(x) = \infty$ for $x \in V_+$,
 \item[(3)] $a(x) = -\infty$, $b(x)  \leq \ol{\Lambda}$ for $x \in V_-$, and
 \item[(4)] $a(x) \geq - \ol{\Lambda}$, $b(x) \leq \ol{\Lambda}$ for $x \in V$.
\end{enumerate}
\end{enumerate}

We will also occasionally make the assumption
\begin{enumerate}
\item[($\pm$)] The conditions of $(C)$ hold in the special case that $V = \emptyset$, $a \equiv 0$ in $V_+$ and $b \equiv 0$ in $V_-$.
\end{enumerate}

We will also make use of the following notation.  For $D \subseteq \Z^2$ bounded and $\psi \colon \partial D \to \R$ a given boundary condition, we let $\p_D^\psi$ denote the law of the GL model on $D$ with boundary condition $\psi$.  Explicitly, this is the measure on functions $h \colon D \to \R$ with density
\begin{equation}
\label{gl::eqn::density}
  \frac{1}{\CZ} \exp\left( - \sum_{b \in D^*} \CV( \nabla (h \vee \psi) (b)) \right)
\end{equation}
with respect to Lebesgue measure on $\R^{|D|}$, where
\[ h \vee \psi(x) = \begin{cases} h(x) \text{ if } x \in D,\\ \psi(x) \text{ if } x \in \partial D.\end{cases}\]
If $g \colon D \to \R$, then $\Q_D^{\psi,g}$ is the law of $(h-g)$ where $h$ is distributed according to $\p_D^\psi$.  The expectations under $\p_D^\psi$ and $\Q_D^{\psi,g}$ will be $\E^\psi$ and $\E_\Q^{\psi,g}$, respectively, and we will add an extra subscript if we wish to emphasize the domain.  We will omit the superscript $\psi$ if the boundary conditions are clear from the context.  Often we will be taking expectations over complicated couplings of multiple instances of the GL model, in which case we will typically just write $\E$ since the explicit construction of the coupling will be clear from the context.  

We will use $h$ to refer to a generic instance of the GL model and $h_t$ its Langevin dynamics, where the domain and boundary conditions will be clear from the context.  If we wish to emphasize the boundary condition, we will write $h^\psi$ and $h_t^\psi$ for $h,h_t$, respectively, and to emphasize $D$ we will write $h^D$ and $h_t^D$.  Finally, if we wish to emphasize both then we will write $h^{\psi,D}$ and $h_t^{\psi,D}$.  We will often condition on events of the form $\CK = \cap_{x \in D} \{ a(x) \leq h(x) \leq b(x)\}$ where $a,b$ arise as in $(C)$.  Notationally such conditioning will be expressed in two different ways.  The first possibility is that we will indicate in advance that an instance of the GL model $h$ will always be conditioned on $\CK$ and then make no further indication of it, in which case $h_t$ refers to the conditioned dynamics.  If either we need to emphasize the conditioning or $h$ refers to an unconditioned model, we will write $h|\CK$ for the conditioned model, $(h|\CK)_t$ for its dynamics, $h^\psi|\CK$ to emphasize the boundary condition, and $h^D|\CK$ to emphasize the domain.

The proofs in this article will involve many complicated estimates involving numerous constants.  In order to keep the arguments succinct, we will make rather frequent usage of $O$-notation.  Specifically, we say that $f = O(g)$ if there exists constants $c_1,c_2 > 0$ such that $|f(x)| \leq c_1 + c|g(x)|$ for all $x$.  If we write $f = O_\alpha(g)$ for a parameter or possibly family of parameters $\alpha$, then $c_1, c_2$ depend only on $\alpha$.  Finally, if $X$ and $Y$ are random variables, then $X = O(Y)$ means that $|X| \leq c_1 + c_2|Y|$ for \emph{non-random} $c_1, c_2$.
\section{Conditioned Dynamics}
\label{sec::cond_dynam}

Suppose that $D \subseteq \Z^2$ with $\diam(D) < \infty$, $\psi \colon \partial D \to \R$, and $a,b \colon D \to [-\infty,\infty]$ satisfy $a\leq b$.  The Langevin dynamics associated with $h \sim \p_D^\psi[ \cdot | \CK]$ where $\CK = \cap_{x \in D} \{ a(x) \leq h(x) \leq b(x)\}$ are described by the SDS
\begin{align}
\label{gl::eqn::cond_dynam}
 dh_t(x) = \sum_{b \ni x} \CV'(\nabla (h_t \vee \psi)(b))dt& + d[\ell_t^a - \ell_t^b](x) + \sqrt{2}dW_t(x),\\
  &x \in D, t \in \R \notag.
\end{align}
Here, $W$ is a family of independent, standard, two-sided Brownian motions and the processes $\ell^a,\ell^b$ are of bounded variation, non-decreasing, and non-zero only when $h_t(x) = a(x)$ or $h_t(x) = b(x)$, respectively.  If $a(x) = -\infty$, then $\ell^a(x) \equiv 0$ and if $b(x) = \infty$, then $\ell^b(x) \equiv 0$.  In particular, if $a \equiv -\infty$ and $b \equiv \infty$, then we just recover the Langevin dynamics of $\p_D^\psi$; see \eqref{gl::eqn::dynam} of \cite{M10}.

\subsection{Brascamp-Lieb and FKG inequalities}

In subsection \ref{subsec::hs_representation} of \cite{M10} we collected a few of the basic properties of the HS representation for the GL model without conditioning.  The HS representation is actually applicable in much more generality.  We will summarize that which is developed in Remark 2.3 of \cite{DGI00} relevant for our purposes.  Suppose that $\CU_x$ is a family of $C^2$ functions indexed by $x \in D$ satisfying $0 \leq \CU_x'' \leq \alpha.$  The law of the GL model with potential $\CV$ and self-potentials $\CU_x$ is given by the density
\begin{align*}
   \frac{1}{\CZ_{\CV,\CU}} \exp\left( -\sum_{b \in D^*} \CV( \nabla (h \vee \psi)(b)) - \sum_{x \in D} \CU_x(h(x)) \right)
\end{align*}
with respect to Lebesgue measure.  The associated Langevin dynamics are described by the SDS:
\begin{align*}
   d h_t^{\CU}(x) = \left[ \sum_{b \ni x} \CV'(\nabla h_t^{\CU} \vee \psi(b)) + \CU_x'(h_t^{\CU}(x))\right]dt + \sqrt{2} dW_t(x). 
\end{align*}
Letting $X_t^\CU$ be the random walk with time-dependent jump rates $\CV''(\nabla h_t^\CU(b))$, the covariance is given by:
\begin{align}    
\label{gl::eqn::hs_self_potential}
     &\cov(h^\CU(x), h^\CU(y))\\
  =&  \E_x^\CU \left[ \int_0^\tau \exp\left( -\int_0^s \CU_{X_u^\CU}''(h_u^\CU(X_u^\CU)) du \right) \one_{\{ X_s^\CU = y\}}  ds \right], \notag
\end{align}
where the subscript $x$ indicates that $X_0^\CU = x$.

Recall that the DGFF $h^*$ on $D$ is the random field with density as in \eqref{gl::eqn::density} in the special case $\CV(x) = \tfrac{1}{2} x^2$.  From \eqref{gl::eqn::hs_self_potential}, we immediately obtain the following comparison inequality which bounds from above centered moments of linear functionals of the \emph{conditioned} GL model by the corresponding moments of the \emph{unconditioned} DGFF.  Specifically, for $\nu,\mu \in \R^{|D|}$, letting
\[ \langle \mu, \nu \rangle = \sum_{x \in D} \mu_x \nu_x,\] we have:
\begin{lemma}[Brascamp-Lieb inequalities]
\label{bl::lem::bl_inequalities}
 Suppose that $h^*$ is a zero-boundary DGFF on $D$ and $h \sim \p_D^\psi[\cdot|\CK]$. There exists $C > 0$ depending only on $a_\CV,A_\CV$ such that the following inequalities hold:
\begin{align}
 & \var(\langle \nu, h \rangle) \leq C  \var( \langle \nu, h^* \rangle ) \label{gl::eqn::bl_var}, \\
 & \E[\exp(\langle \nu, h \rangle - \E[\langle \nu,h \rangle])] \leq \E[\exp( C\langle \nu, h^* \rangle )] \label{gl::eqn::bl_exp}
\end{align}
for all $\nu \in \R^{|D|}$.
\end{lemma}
\begin{proof}
For each $-\infty \leq \alpha < \beta \leq \infty$, fix a $C^\infty(\R)$ function $f_{\alpha,\beta}$ such that $f_{\alpha,\beta}|_{[\alpha,\beta]} \equiv 0$, $f_{\alpha,\beta}|_{[\alpha,\beta]^c} > 0$, and $0 \leq f_{\alpha,\beta}''(x) \leq 1$ for all $x \in \R$.  Let $\CU_x^n = n f_{a(x),b(x)}$.  If $h_n$ has the law of the GL model with self-potentials $\CU_x^n$ it follows from \eqref{gl::eqn::hs_self_potential} that
\[ \var(\langle \nu, h_n\rangle) \leq C \var( \langle \nu, h^* \rangle)\]
for some $C > 0$ depending only on $a_\CV,A_\CV$.
As $n \to \infty$, $h_n \stackrel{d}{\to} h$, which proves \eqref{gl::eqn::bl_var}.  One proves \eqref{gl::eqn::bl_exp} using a similar method; see also Corollary 2.7 from \cite{DGI00}.
\end{proof}

More generally, if $F, G \colon \R^{|D|} \to \R$ are smooth, then \eqref{gl::eqn::hs_self_potential} becomes
\begin{align}
     \label{gl::eqn::hs_self_potential_local_cov}
    &\cov(F(h),G(h)) = \\
   \E_x\bigg[ \partial F(X_0^\CU&, h_0^\CU) \int_0^\tau \exp\bigg( - \int_0^s \CU_{X_u^\CU}''(h_u^\CU(X_u^\CU)) du \bigg) \partial G(X_s^\CU, h_s^\CU) ds \bigg] \notag
 \end{align}
 where $\partial F(x,h) = \frac{\partial F}{\partial h(x)}(h)$.  This leads to a simple proof of the FKG inequality, which gives that monotonic functionals of the field are non-negatively correlated:
\begin{lemma}[FKG inequality]
\label{gl::lem::fkg}
 Suppose that $F,G \colon \R^{|D|} \to \R$ are smooth monotonic functionals, i.e. if $\varphi_1,\varphi_2 \in \R^{|D|}$ with $\varphi_1(x) \leq \varphi_2(x)$ for every $x \in D$ then $F(\varphi_1) \leq F(\varphi_2)$ and $G(\varphi_1) \leq G(\varphi_2)$.  For $h \sim \p_D^\psi[\cdot|\CK]$, we have
\[ \E[F(h) G(h)] \geq \E[F(h)] \E[G(h)].\]
\begin{proof}
This can be deduced from \eqref{gl::eqn::hs_self_potential_local_cov} using the same method as the previous lemma to deal with the conditioning; see also Remark 2.4 of \cite{DGI00}.
\end{proof}
\end{lemma}

\subsection{Dynamic Coupling}

The method of dynamic coupling, introduced in \cite{FS97} and which played a critical role in \cite{M10}, also generalizes in the presence of conditioning.  Specifically, suppose that $h_t^{\psi}, h_t^{\wt{\psi}}$ both solve \eqref{gl::eqn::cond_dynam} with the same Brownian motions but possibly different boundary conditions $\psi, \wt{\psi}$.  Then $\ol{h}_t(x) = h_t^\psi(x) - h_t^{\wt{\psi}}(x)$ solves the SDE
\begin{equation}
\label{gl::eqn::cond_dynam_diff_long} d\ol{h}_t(x) = \sum_{b \ni x} [\CV'(\nabla h_t^{\psi}(b)) - \CV'(\nabla h_t^{\wt{\psi}}(b))]dt + d(\ol{\ell}_t^a - \ol{\ell}_t^b)(x)
\end{equation}
where $\ol{\ell}^a = \ell^{a,\psi} - \ell^{a,\wt{\psi}}$ and $\ol{\ell}^b = \ell^{b,\psi} - \ell^{b,\wt{\psi}}$.  Letting
\begin{equation}
\label{gl::eqn::c_l_def}
 c_t(b) = \int_0^1 \CV''( \nabla h_t^{\wt{\psi}}(b) + s \nabla \ol{h}_t(b)) ds \text{ and }
   \CL_t f(x) = \sum_{b \ni x} c_t(b) \nabla f(b),
\end{equation}
we can rewrite \eqref{gl::eqn::cond_dynam_diff_long} more concisely as
\begin{equation}
\label{gl::eqn::cond_dynam_diff}
d \ol{h}_t(x) = \CL_t \ol{h}_t(x) dt + d(\ol{\ell}_t^a - \ol{\ell}_t^b)(x).
\end{equation}
By a small computational miracle, the following energy inequality holds in the setting of conditioning:

\begin{lemma}[Energy Inequality]
\label{dynam::lem::ee}
Suppose that $(h_t^\psi,h_t^{\wt{\psi}})$ satisfy \eqref{gl::eqn::cond_dynam} with the same driving Brownian motions and $\ol{h} = h^\psi - h^{\wt{\psi}}$.  There exists $C > 0$ depending only on $\CV$ such that for every $T > S$ we have
\begin{align}
\label{gl::eqn::ee}
&\sum_{x \in D} |\ol{h}_T(x)|^2 + \int_S^T \sum_{b \in D^*} |\nabla \ol{h}_t(b)|^2 dt \notag\\
\leq& C \left(\sum_{x \in D} |\ol{h}_S(x)|^2 + \int_S^T \sum_{b \in \partial D^*} |\ol{\psi}(x_b)||\nabla \ol{h}_t(b)|dt\right).
\end{align}
\end{lemma}
\begin{proof}
This is a generalization of Lemma 2.3 of \cite{FS97}.
From \eqref{gl::eqn::cond_dynam_diff} we have
\[ d (\ol{h}_t(x))^2
   = 2 \ol{h}_t(x) \CL_t \ol{h}_t(x) dt + 2\ol{h}_t(x) d[\ol{\ell}_t^a - \ol{\ell}_t^b](x).
\]
We are now going to prove
\[  d (\ol{h}_t(x))^2
   \leq 2 \ol{h}_t(x) \CL_t \ol{h}_t(x) dt,
\]
from which the result follows by summing by parts and then integrating from $S$ to $T$.
Suppose $a(x) > -\infty$ and $t$ is such that $h_t^\psi(x) \geq h_t^{\wt{\psi}}(x) = a(x)$.  If $h_t^\psi(x) = h_t^{\wt{\psi}}(x)$, then obviously $\ol{h}_t(x) d \ol{\ell}_t^a(x) = 0$.  If $h_t^\psi(x) > h_t^{\wt{\psi}}(x)$ then $d\ell_t^{a,\psi}(x) = 0$ while $d \ell_t^{a,\wt{\psi}}(x) > 0$.  Thus $\ol{h}_t(x) > 0$ and $d \ol{\ell}_t^a(x) < 0$, so that $\ol{h}_t(x) d\ol{\ell}_t^a(x) < 0$.  Therefore $\ol{h}_t(x) d\ol{\ell}_t^a(x) \leq 0$.  We can play exactly the same game to prove $-\ol{h}_t(x) d \ol{\ell}_t^b(x) \leq 0$ if $b(x) < \infty$, which proves our claim.
\end{proof}

Suppose that $(h_\infty^\psi,h_\infty^{\wt{\psi}})$ is a subsequential limit of $(h_t^\psi,h_t^{\wt{\psi}})$ as $t \to \infty$.  By dividing both sides of \eqref{gl::eqn::ee} by $T$ and sending $T \to \infty$ we see that $\ol{h}_\infty$ satisfies
\begin{equation}
\label{gl::eqn::ee_limit}
\sum_{b \in D^*} \E |\nabla \ol{h}_\infty(b)|^2 \leq C\sum_{b \in \partial D^*} \E |\ol{\psi}(x_b)||\nabla \ol{h}_\infty(b)| 
\end{equation}

\begin{lemma}$\ $
\label{gl::lem::ergodic}
\begin{enumerate}
 \item The SDS \eqref{gl::eqn::cond_dynam} is ergodic.
 \item \label{gl::lem::ergodic::stationary} More generally, any finite collection $h^1,\ldots,h^n$ satisfying the SDS \eqref{gl::eqn::cond_dynam} each with the same conditioning and driven by the same family of Brownian motions is ergodic.
 \item If $(h^1,\ldots,h^n)$ is distributed according to the unique stationary distribution from part \eqref{gl::lem::ergodic::stationary}, then $\ol{h}^{ij} = h^i - h^j$ satisfies \eqref{gl::eqn::ee_limit}
\end{enumerate}
\end{lemma}
\begin{proof}
Lemma \ref{gl::lem::ergodic} of \cite{M10} contains the same statement but for the unconditioned dynamics.  The proof, however, relies only on the energy inequality hence is also valid here.
\end{proof}

\noindent We shall refer to the coupling $(h^1,\ldots,h^n)$ provided by part \eqref{gl::lem::ergodic::stationary} of the previous lemma as the \emph{stationary coupling} of the laws of $h^1,\ldots,h^n$.

We now need an analog of Lemma \ref{gl::lem::grad_error} from \cite{M10}.  In the setting of that article, this followed by combining the Caccioppoli inequality with the Brascamp-Lieb inequalities.  While we do have the latter even in the presence of conditioning, we do not have the former.  Luckily, we are able to deduce the same result using only the energy inequality and an iterative technique.  For $E \subseteq D$ we let $E(s) = \{ x \in E : \dist(x, \partial E) \geq s\}$.

\begin{lemma}
\label{gl::lem::grad_error}
Suppose $D \subseteq \Z^2$ with $R = \diam(D) < \infty$, let $\psi,\wt{\psi} \colon \partial D \to \R$, and let $E \subseteq D$ with $r = \diam(E)$.  Assume that $(h_t^\psi,h_t^{\wt{\psi}})$ is a stationary coupling of two solutions of the SDS \eqref{gl::eqn::dynam} with the same conditioning.  Let 
\[ M = \max_{x \in E} \big[ \E[(h^\psi)^2(x)] + \E[(h^{\wt{\psi}})^2(x)] \big].\]
For every $\epsilon > 0$ there exists a constant $k = k(\epsilon)$ and $kr^{1-\epsilon} \leq r_\epsilon \leq (k+1) r^{1-\epsilon}$ such that
\begin{equation}
\label{gl::eqn::dirichlet_bound}
\sum_{b \in E^*(r_\epsilon)} \E |\nabla \ol{h}(b)|^2 = O(r^{3\epsilon} M).
\end{equation}
\end{lemma}
\begin{proof}
Equation \eqref{gl::eqn::ee_limit} used in conjunction with the Cauchy-Schwarz inequality implies that for any subdomain $E_1 \subseteq D$,
\[ \sum_{b \in E_1^*} \E |\nabla \ol{h}(b)|^2 = O(|\partial E_1| M).\]
Since
\[ \sum_{s=1}^{r^{1-\epsilon}} |\partial E(s)| \leq |E| = O(r^2) \text{ and } 
   \sum_{s=1}^{r^{1-\epsilon}} \sum_{b \in \partial E^*(s)} \E | \nabla \ol{h}(b)|^2 = O(r^{2} M),\]
it follows there exists $0 \leq r_1 \leq r^{1-\epsilon}$ with $|\partial E(r_1)| = O(r^{1+\epsilon})$ such that
\[ \sum_{b \in \partial E^*(r_1)} \E |\nabla \ol{h}(b)|^2 = O(r^{1+\epsilon} M).\]
Inserting these bounds back into \eqref{gl::eqn::ee_limit} and applying the Cauchy-Schwarz inequality, we see that
\begin{align*}
 \sum_{b \in E^*(r_1)} \E| \nabla \ol{h}(b)|^2 \leq&
   C\left(|\partial E^*(r_1)| \max_{b \in \partial E^*(r_1)} \E |\ol{h}(x_b)|^2 \right)^{1/2}\left( \sum_{b \in \partial E^*(r_1)} \E |\nabla \ol{h}(b)|^2\right)^{1/2}\\
   =& O(\sqrt{ r^{1+\epsilon} M}) O(\sqrt{r^{1+\epsilon} M}) = 
 O(r^{1+\epsilon} M).
\end{align*}
By the same averaging argument, this in turn implies there exists $r^{1-\epsilon} \leq r_2 \leq 2 r^{1-\epsilon}$ such that
\[ \sum_{b \in \partial E^*(r_2)} \E |\nabla \ol{h}(b)|^2 = O(r^{2\epsilon} M)\]
and $|\partial E^*(r_2)| = O(r^{1+\epsilon})$.  Combining \eqref{gl::eqn::ee_limit} with the Cauchy-Schwarz inequality again yields
\begin{align*}
 \sum_{b \in E^*(r_2)} \E |\nabla \ol{h}(b)|^2 
&\leq C \left(|\partial E^*(r_2)| \max_{b \in \partial E^*(r_2)} \E |\ol{h}(x_b)|^2\right)^{1/2} \left(\sum_{b \in \partial E^*(r_2)} \E |\nabla \ol{h}(b)|^2\right)^{1/2}\\
&= O(\sqrt{r^{1+\epsilon} M}) O(\sqrt{r^{2 \epsilon} M}) = O(r^{1/2+3/2\epsilon} M).
\end{align*}
Iterating this $k$ times yields the existence of $(k-1) r^{1-\epsilon} \leq r_k \leq k r^{1-\epsilon}$ such that $|\partial E^*(r_k)| = O(r^{1+\epsilon})$ and
\begin{align*}
  \sum_{b \in E^*(r_k)} \E |\nabla \ol{h}(b)|^2 
&\leq C \left(|\partial E^*(r_k)| \max_{b \in \partial E^*(r_k)} \E |\ol{h}(x_b)|^2\right)^{1/2} \left(\sum_{b \in \partial E^*(r_k)} \E |\nabla \ol{h}(b)|^2\right)^{1/2}\\
&= O(r^{2^{-k} + \alpha_k \epsilon} M),
\end{align*}
where $\alpha_k = \sum_{j=0}^k 2^{-j} \leq 2$.  Taking $k$ large enough gives \eqref{gl::eqn::dirichlet_bound}.
\end{proof}

Lemma \ref{gl::lem::grad_error} will be used in conjunction with Lemma \ref{te::lem::expectation}, which provides bounds for $M$.

\subsection{The Random Walk Representation and Stochastic Domination}
\label{subsec::rw_difference}

The energy method of the previous subsection allowed us to deduce macroscopic regularity and ergodicity properties of the dynamic coupling.  In this subsection, we will develop the random-walk representation of $\ol{h}_t(x)$, which allows for pointwise estimates.

Fix $T > 0$ and let $X_t^T$ be the random walk in $D$ with time-dependent generator $t \mapsto \CL_{T-t}$ with $\CL_t$ as in \eqref{gl::eqn::c_l_def}.  Note that $c_t(b)$ makes sense for $t < 0$ hence $\CL_{T-t}$ for $t > T$ since \eqref{gl::eqn::cond_dynam} is defined for all $t \in \R$.  Let $\tau = \inf\{t \geq 0 : X_t^T \notin D\}$.

\begin{remark}
\label{gl::rem::stoc_rep}
In the special case $a(x) = -\infty$ and $b(x) = \infty$, so that the fields are unconditioned, the stationary coupling $(h_t^\psi, h_t^{\wt{\psi}})$ satisfies
\begin{equation}
\label{gl::eqn::rand_walk_diff}
 \ol{h}_T(x) = \E_x[\ol{h}_{T-\tau}(X_{\tau}^T)],
\end{equation}
where the expectation is taken only over the randomness of $X^T$.  Consequently, if $\ol{h}|\partial D \geq 0$ then $\ol{h}_T \geq 0$.  In other words, the stationary coupling $(h_t^\psi,h_t^{\wt{\psi}})$ satisfies $h_t^\psi \geq h_t^{\wt{\psi}}$ if the inequality is satisfied uniformly on the boundary.
\end{remark}

The purpose of the following lemma is to establish the same result in the presence of conditioning.

\begin{lemma}
\label{gl::lem::stoch_dom}
If $D \subseteq \Z^d$ is bounded and $\psi,\wt{\psi} \colon \partial D \to \R$ satisfy $\psi(x) \geq \wt{\psi}(x)$ for every $x \in \partial D$, then the stationary coupling $(h_t^\psi,h_t^{\wt{\psi}})$ of $\p_D^\psi[\cdot|\CK], \p_D^{\wt{\psi}}[\cdot|\CK]$ satisfies $h_t^\psi(x) \geq h_t^{\wt{\psi}}(x)$ for every $x \in D$.
\end{lemma}
\begin{proof}
The proof is similar to that of \cite[Lemma 2.4]{DN07}, which gives a stochastic domination result in a slightly different context.  For $\alpha \in \R$, we let $\alpha^- = \min(\alpha,0)$.  Note that
\begin{align*}
     d((\ol{h}_t)^-(x))^2
=& 2 (\ol{h}_t)^-(x) \CL_t \ol{h}_t(x) dt + 2(\ol{h}_t)^-(x) d[\ol{\ell}_t^a - \ol{\ell}_t^b](x)\\
 \leq& 2(\ol{h}_t)^-(x) \CL_t \ol{h}_t(x) dt.
\end{align*}
The last inequality used that $(\ol{h}_t)^-(x) d[\ol{\ell}_t^a - \ol{\ell}_t^b](x) \leq 0$, as in the proof of the energy inequality.  Thus,
\begin{align}
   &  d\left(\sum_{x \in D} ((\ol{h}_t)^-(x))^2 \right)
\leq 2\sum_{x \in D} (\ol{h}_t)^-(x) \CL_t \ol{h}_t(x) dt \notag\\
  =& -2\sum_{b \in D^*} c_t(b) \nabla (\ol{h}_t)^-(b) \nabla \ol{h}_t(b) dt, \label{gl::eqn::square_neg}
\end{align}
where in the last step we used summation by parts and that $(\ol{h}_t)^-|_{\partial D} \equiv 0$.  Now using $(\alpha^- - \beta^-)(\alpha-\beta) \geq (\alpha^- - \beta^-)^2$, we see that the previous expression is bounded from above by
\begin{equation}
\label{gl::eqn::square_neg2}
 -2 \sum_{b \in D^*} a_\CV [\nabla (\ol{h}_t)^-(b)]^2 dt.
\end{equation}
This implies that $S = \lim_{t \to \infty} \sum_{x \in D} ((\ol{h}_t)^-(x))^2$ exists and is constant.  By the Poincar\'e inequality,
\[ \liminf_{t \to \infty} \sum_{b \in D^*} [\nabla (\ol{h}_t)^-(b)]^2 dt \geq c_D S,\]
for $c_D > 0$ depending only on $D$.  Combining this with \eqref{gl::eqn::square_neg}, \eqref{gl::eqn::square_neg2} clearly implies $S = 0$.
\end{proof}

\begin{remark}
In the setting of Remark \ref{gl::rem::stoc_rep}, combining \eqref{gl::eqn::rand_walk_diff} with Jensen's inequality yields in the unconditioned case that
\[ \ol{h}_T^2(x) \leq \E_x [\ol{h}_{T-\tau}^2(X_{\tau}^T)],\]
where the expectation is just over the randomness in $X^T$.
\end{remark}
The same result also holds in the conditioned case, though we have to work ever so slightly harder to prove it.

\begin{lemma}
\label{gl::lem::square_bound}
Assume that we have the same setup as Lemma \ref{gl::lem::stoch_dom}.  We have,
\begin{equation}
\label{gl::eqn::rand_walk_diff_square}
\ol{h}_T^2(x) \leq \E_x[\ol{h}_{T-\tau}^2(X_{\tau}^T)].
\end{equation}
\end{lemma}
\begin{proof}
As in the proof of Lemma \ref{dynam::lem::ee}, 
\[ d \ol{h}_{T-t}^2(x) = 2 \ol{h}_{T-t}(x) d (\ol{h}_{T-t}(x)) \geq -2\ol{h}_{T-t}(x) \CL_{T-t} \ol{h}_{T-t}(x) dt.\]
Thus as
\[ \CL_{T-t} \ol{h}_{T-t}^2(x) = 2 \ol{h}_{T-t}(x) (\CL_{T-t} \ol{h}_{T-t})(x) + \sum_{b \ni x} c_{T-t}(b) (\nabla \ol{h}_{T-t}(b))^2\]
and, in particular,
\[ \CL_{T-t} \ol{h}_{T-t}^2(x) \geq 2 \ol{h}_{T-t}(x) (\CL_{T-t} \ol{h}_{T-t})(x)\]
we consequently have
\begin{equation}
\label{gl::eqn::h_square_bound}
 d \ol{h}_{T-t}^2(x)  + \CL_{T-t} \ol{h}_{T-t}^2(x) dt \geq 0.
\end{equation}
If $g \colon [0,T] \times \Z^2 \to \R$ is $C^1$ in its first variable with $\| \partial_s g\|_\infty < \infty$, then with
\[ M_t(g) = g(t,X_{t}^T) - g(0,x) - \int_0^t (\partial_s g)(s,X_{s}^T) + \CL_{T-s} g(s,X_{s}^T) ds,\ \ t \in [0,T],\]
we see that $M_{t \wedge \tau}(g)$ is a bounded martingale with respect to the filtration $(\CF_t)$, $\CF_t = \sigma( X_s^T : s \leq t)$.  Let $(\tau_k)$ be the jump times of $X^T$.  Then $M_t(g)$ can also be expressed as
\begin{align*}
    M_t(g) =& g(t,X_{t}^T) - g(0,x) - \sum_{k} [g(\tau_{k+1} \wedge t,X_{\tau_k \wedge t}^T) - g(\tau_k \wedge t,X_{\tau_k \wedge t}^T)]\\
       &- \int_0^t \CL_{T-s} g(s,X_{s}^T) ds.
\end{align*}
This representation allows us to make sense of $g \mapsto M_t(g)$ for $g$ which are not necessarily differentiable in time.  Letting $N(T) = \sup\{ k : \tau_k \leq T\}$, we have
\[ \|M_{\cdot \wedge \tau}(g)\|_\infty \leq C(T+1+N(T \wedge \tau))\| g \|_\infty\]
for some $C > 0$, where the supremum on the left hand side is taken over $[0,T]$ and on the right over $[0,T] \times D$.  Observe that $N(T \wedge \tau)$ has finite moments of all orders uniformly bounded in $T$ since the jump rates of $X^T$ are bounded from below and $D$ is bounded.  Consequently, taking a sequence $(g_n)$ with $g_n \colon [0,T] \times \Z^2 \to \R$ which is $C^1$ in the first variable such that $g_n(\cdot,x) \to \ol{h}_{T-\cdot}^2(x)$ uniformly in $t$, implies $M_{t \wedge \tau}(\ol{h}_{T-\cdot}^2)$ is also a $(\CF_t)$ martingale.  Note that
\[ M_t(\ol{h}_{T-\cdot}^2) = \ol{h}_{T-t}^2(X_t^T) - \ol{h}_T^2(x) - \int_0^t d \ol{h}_{T-s}^2(X_s^T) - \int_0^t \CL_{T-s} \ol{h}_{T-s}^2(X_{s}^T) ds.\]
Combining this with \eqref{gl::eqn::h_square_bound} implies
\[ \ol{h}_{T}^2(x) \leq  \ol{h}_{T- \tau}^2(X_{\tau}^T) - M_{\tau}.\]
Taking expectations of both sides, using the uniform integrability of the martingale $M_{t \wedge \tau}$, and invoking the optional stopping theorem proves the lemma.
\end{proof}

\section{Moment Estimates}
\label{sec::expectation}

It will be rather important for us to have control on the exponential moments of $h$ conditional on $\CK  = \cap_{x \in D} \{ a(x) \leq h(x) \leq b(x)\}$ near $U$.  Such an estimate does not follow from the exponential Brascamp-Lieb inequality since this only bounds the \emph{centered exponential moment} in terms of the corresponding moment for the \emph{unconditioned DGFF}, the latter of which is of polynomial order in $R = \diam(D)$ in the bulk.  It will also be important for us to know that $h(x) - a(x)$ for $x \in V_+ \cup V$ and $b(x) - h(x)$ for $x \in V_+ \cup V$ are uniformly positive in expectation conditional on $\CK$.

\begin{lemma}
\label{te::lem::expectation}
Assume $(\partial)$, $(C)$,  and fix $\eta \in (0,1/2)$.  There exists constants $c_1 = c_1(\eta)$ and $c_2 = c_2(\ol{\Lambda}, \eta)$ such that the following holds.  If $v \in D$, $r > (\log R)^{c_1}$ is such that $B(v,r^{1+3\eta} \wedge R) \cap U$ contains a connected subgraph $U_0$ of $U$ with $U_0 \cap \partial B(v,r^{1+3\eta} \wedge R) \neq \emptyset$ and $\dist(v,U_0) \leq r$, then
\begin{equation}
\label{te::eqn::exponential_moment}
\E[ \exp(|h(v)|) | \CK] \leq c_2 r^{c_2}.
\end{equation}
\end{lemma}

The reason for the hypothesis that there is a large, connected subgraph $U_0$ of $U$ near $v$ is to ensure that symmetric random walk with bounded rates initialized at $v$ is much more likely to hit $U_0$ before exiting a ball of logarithmic size around $v$.  Note in particular that this hypothesis trivially holds when $U$ consists of a path in $D$ connected to and along with $\partial D$.

The idea of the proof is to use repeatedly the stochastic domination results of the previous section along with an iterative argument to reduce the problem to a GL model on a domain whose size is polynomial in $r$.  Specifically, we without loss of generality assume $v \in D_W \equiv D \setminus W$ where $W = U \setminus V_+$.  By stochastic domination, it suffices to control $\E[\exp(h^{D_W}(v))|\CK^{D_W}]$ where $h^{D_W}$ has the law of the GL model on $D_W$ with the same boundary conditions as $h$ on $\partial D$, constant boundary conditions $\ol{\Lambda}$ on $W$, and $\CK^{D_W} = \cap_{x \in V_+} \{a(x) \leq h^{D_W}(x)\}$.  By hypothesis, there exists $u \in W$ with $|u-v| \leq r+2$, hence the exponential Brascamp-Lieb inequality applied to $h^{D_W}|\CK^{D_W}$ implies that the centered, exponential moments of $(h^{D_W}|\CK^{D_W})(v)$ are polynomial in $r$.  This reduces the problem to estimating $\E[h^{D_W}(v) | \CK^{D_W}]$.  The idea now is to prove an \emph{a priori} estimate of $\E[h^{D_W}(v)|\CK^{D_W}]$ using the FKG inequality, then use the method of dynamic coupling repeatedly to construct a comparison between $\E[h^{D_W}(v)|\CK^{D_W}]$ and the expected height of a GL model on a ball with diameter which is polynomial in $r$.  This completes the proof since our \emph{a priori} estimate implies the latter is $O_{\ol{\Lambda}}(\log r)$.

\begin{proof}
We begin with the observation
\[ \E[ \exp(|h(v)|) | \CK] \leq \E[ \exp(h(v)) | \CK] + \E[ \exp(-h(v)) | \CK].\]
Let $W, D_W, h^{D_W}, \CK^{D_W}$ be as in the paragraph after the statement of the lemma.  By Lemma \ref{gl::lem::stoch_dom}, there exists a coupling of $h| \CK$, $h^{D_W} | \CK^{D_W}$ such that $h^{D_W} | \CK^{D_W} \geq h | \CK$, hence to bound $\E[\exp(h(v)) | \CK]$ it suffices to bound $\E[ \exp(h^{D_W}(v))| \CK^{D_W}]$.
By the exponential Brascamp-Lieb inequality (Lemma \ref{bl::lem::bl_inequalities}), we have
\begin{align*}
       &  \E[\exp( h^{D_W}(v)) | \CK^{D_W}]\\
\leq& \exp(\E[ h^{D_W}(v) \big| \CK^{D_W}])\E[\exp(h^{D_W}(v) - \E[h^{D_W}(v)|\CK^{D_W}]) | \CK^{D_W}]\\
\leq& \exp(\E[h^{D_W}(v) \big| \CK^{D_W}]) \E[\exp( C (h^{D_W})^*(v))]
\end{align*}
where $(h^{D_W})^*$ has the law of a zero-boundary DGFF on $D_W$ and $C = C(\CV) > 0$ is a constant depending only on $\CV$.  Since $\dist(v,V_+) \leq r$, there exists $w \in \partial D_W$ such that $|v - w| \leq r+2$ by (C), hence $\var((h^{D_W})^*(v)) = O(\log r)$.  The reason for this is that a random walk initialized at $v$ has probability $\Omega((\log r)^{-1})$ of hitting $w$ hence $W$ before visiting $v$ again after each successive visit.  This, in turn, implies $\E[\exp( C (h^{D_W})^*(v))] \leq C' r^{C'}$ for some $C' > 0$.  Consequently, to prove the lemma we just need to bound $\E[| h^{D_W}(v)| \big| \CK^{D_W}]$.  We will break the proof up into three main steps.  The first is to get an \emph{a priori} estimate on the behavior of the maximum, the second is to use a coupling argument to improve the estimate by comparison to a model on a smaller domain, and the third is to show how this coupling argument may be iterated repeatedly in order to get the final bound.

{\it Step 1.}  Let $A = \cup_{y \in D_W} \{h^{D_W}(y) \geq \alpha (\log R)\}.$
The goal of this step is to prove that $\p[ A | \CK^{D_W}] = O_{\ol{\Lambda}}(R^{-100})$
provided $\alpha = \alpha(\ol{\Lambda})$ is chosen sufficiently large.

Let $h^{D_W,\gamma}$ have the law of the GL model on $D_W$ with constant boundary conditions $\gamma C^{-1} (\log R)$ where $\gamma = \gamma(\ol{\Lambda})$ is to be chosen later.  Let $\CK^{D_W,\gamma} = \cap_{x \in D_W} \{ a(x) \leq h^{D_W,\gamma}(x)\}$.  Finally, let $A^\gamma$ be the event analogous to $A$ but with $h^{D_W}$ replaced with $h^{D_W,\gamma}$.  It suffices to show that $\p[ A^\gamma | \CK^{D_W,\gamma}] = O_{\ol{\Lambda}}(R^{-100})$
since by Lemma \ref{gl::lem::stoch_dom} we can couple the laws of $h^{D_W,\gamma}|\CK^{D_W,\gamma}$ and $h^{D_W}|\CK^{D_W}$ such that $h^{D_W,\gamma}|\CK^{D_W,\gamma} \geq h^{D_W}|\CK^{D_W}$ almost surely.  We will first prove that we can pick $\alpha = \alpha(\ol{\Lambda})$ large enough so our claim holds without conditioning:
\begin{equation}
\label{te::eqn::abound} \p[A^\gamma] \leq O_{\ol{\Lambda}}(R^{-100}),
\end{equation}
then show that $\p[\CK^{D_W,\gamma}] = 1-o(1)$ for $\gamma$ large enough.
By the exponential Brascamp-Lieb and Chebychev inequalities, for some $C > 0$ we have
\begin{align*}
  \p[h^{D_W,\gamma}(y) > \beta C^{-1} (\log R)] 
\leq& \exp( -\beta (\log R)) \E[\exp(Ch^{D_W,\gamma}(y))]\\
 \leq& \exp((O_{\ol{\Lambda}}(1)-\beta)(\log R)).
\end{align*}
Here, we are using that $\var(h^{D_W,\gamma}(y)) = O(\log R)$ and $\E[h^{D_W,\gamma}(y)] = O_{\ol{\Lambda}}(\log R)$.  The latter can be seen, for example, using Lemma \ref{gl::lem::hs_mean_cov} of \cite{M10}, the HS representation of the mean. Choosing $\beta = \beta(\gamma) > 0$ large enough along with a union bound now gives \eqref{te::eqn::abound}.

With $h^{D_W,0}$ the zero-boundary GL model on $D_W$, a similar argument with the Brascamp-Lieb and Chebychev inequalities yields
\begin{align*}
   \p[ h^{D_W,0}(y) \leq \gamma C^{-1} (\log R) - \ol{\Lambda}] = 1-O_{\ol{\Lambda}}(R^{-5})
\end{align*}
provided we choose $\gamma = \gamma(\ol{\Lambda})$ large enough.
Note that for $y \in D_W$, the symmetry of the law of $h^{D_W,0}$ about zero gives us
\begin{align*}
    &\p[h^{D_W,\gamma}(y) > a(y)]
\geq \p[ h^{D_W,0}(y) \geq \ol{\Lambda} - \gamma C^{-1} (\log R)]\\
 =& \p[ h^{D_W,0}(y) \leq \gamma C^{-1} (\log R) - \ol{\Lambda}]
  \geq 1 - O_{\ol{\Lambda}}(R^{-5}).
\end{align*}
Invoking the FKG inequality yields
\begin{align*}
  \p[ \CK^{D_W,\gamma}]
&\geq \prod_{y \in V_+} \p[ h^{D_W,\gamma}(y) \geq a(y)]
 \geq (1 - O_{\ol{\Lambda}}(R^{-5}))^{R^2}
 = 1 - O_{\ol{\Lambda}}(R^{-1}).
\end{align*}
Therefore $\p[ A^\gamma | \CK^{D_W,\gamma}] = O_{\ol{\Lambda}}(R^{-100})$, as desired.\newline

{\it Step 2.} We next claim that 
\[ |\E[h^{D_W}(v) | \CK^{D_W}]| = O_{\ol{\Lambda}}(\log r + \log \log R).\]
If $r \geq R^{1/3}$, then this is immediate from the previous part, so assume that $r < R^{1/3}$.
By the definition of $A$,
\begin{align*}
   &\E[\max_{y \in D_W} \big| h^{D_W}(y)|^p \big| \CK^{D_W}] 
= O_{\ol{\Lambda}}((\log R)^p) + \sum_{y \in D_W} \E[|h^{D_W}(y)|^p \one_{A}| \CK^{D_W}].
\end{align*}
Using that $\var[ h^{D_W}(y) | \CK^{D_W}] = O(\log R)$, the Brascamp-Lieb and Cauchy-Schwarz inequalities yield
\begin{align*} 
       &\E[| h^{D_W}(y)|^p \one_{A} \big|\CK^{D_W}]\\
\leq &\bigg( \E[ (h^{D_W}(y) - \E[|h^{D_W}(y)| \big| \CK^{D_W}])^{2p} | \CK^{D_W}] + (\E[|h^{D_W}(y)| \big| \CK^{D_W}])^{2p} \bigg)^{1/2} O_{\ol{\Lambda}}(R^{-50})\\
\leq &O_{\ol{\Lambda}}(R^{-20}) + \E[\max_{y \in D_W}|h^{D_W}(y)|^p \big| \CK^{D_W}] O_{\ol{\Lambda}}(R^{-20}).
\end{align*}
Inserting this into the previous equation and rearranging leads to the bound
\begin{align}
\label{te::eqn::max_bound}
\E[\max_{y \in D_W} |h^{D_W}(y)|^p \big|\CK^{D_W}] 
 = O_{\ol{\Lambda}}((\log R)^p).
\end{align}
Let $\delta > 1$; we will determine its precise value shortly.  Let $B_\delta^W = B(v,r^\delta) \cap D_W$, $\zeta = h^{D_W} |_{\partial B_\delta^W}$, and let $h^{\zeta}$ have the law of the GL model on $B_\delta^W$ with boundary condition $\zeta$ and with conditioning $a(x) \leq h^{\zeta}(x) \leq b(x)$.  Let $h^{\zeta,0}$ have the law of the GL model on $B_\delta^W$ with the same boundary conditions as $h^{\zeta}$ on $(\partial B_\delta^W) \cap W$ and with zero boundary conditions on $(\partial B_\delta^W) \setminus W$.  Finally, let $(h_t^{\zeta}, h_t^{\zeta,0})$ be the stationary coupling of the corresponding dynamic models.  With $\ol{h}_t = h_t^{\zeta} - h_t^{\zeta,0}$, by Lemma \ref{gl::lem::square_bound} we have $\ol{h}_0^2(z) \leq \E_z[\ol{h}_{-\tau}^2(X_{\tau})]$,
where the expectation is taken only over the randomness of the Markov process $X = X^0$ as in subsection \ref{subsec::rw_difference} initialized at $z$ and $\tau$ its time of first exit from $B_\delta^W$.  It follows from \cite{M10}, Lemma \ref{symm_rw::lem::beurling} that if $z \in B(v,r^{1+\eta \delta})$, then the probability that $X_t$ makes it to the outer boundary of $\partial B(v,r^{\delta})$ before hitting $(\partial B_\delta^W) \cap W$ is $O(r^{-\rho_{\rm B}(\delta(1-\eta) - 1)})$, some $\rho_{\rm B} > 0$ depending only on $\CV$.  Combining this with \eqref{te::eqn::max_bound} implies $\E[|\ol{h}_0(z)|^2] = O_{\ol{\Lambda}}(r^{-\rho_{\rm B}(\delta(1-\eta) - 1)}(\log R)^2).$
Hence taking 
\[ \gamma_0 =  \frac{4}{ \rho_{\rm B} (1-\eta)} \text{ and } \delta = \frac{1}{1-\eta} + \gamma_0 \frac{\log \log R}{\log r},\]
we get that
\begin{equation}
\label{te::eqn::step2_bound}
 \big( \E[|\ol{h}_0(z)|^2] \big)^{1/2} = O_{\ol{\Lambda}}((\log R)^{-1}).
 \end{equation}
Note that with this choice of $\delta$ we have that $r^{\delta} \leq r^{1+3\eta}$ provided we take $c_1 = c_1(\eta)$ large enough.  This implies our claim as Step 1 gives
\[ \E[(h^{\zeta,0})^2(v)] = O_{\ol{\Lambda}}(\log r^\delta) = O_{\ol{\Lambda}}(\log r + \log \log R).\]

{\it Step 3.}  In the previous step, we took our initial estimate of $O_{\ol{\Lambda}}(\log R)$ and improved it to $O_{\ol{\Lambda}}(\log r + \log \log R)$ using a coupling argument to reduce the problem to one on a domain of size $r^{\delta} = r^{1/(1-\eta)} (\log R)^{\gamma_0}$.  Assume that $(\log R)^{\gamma_0} \geq r^{1/(1-\eta)}$, for otherwise we are already done.  Then $r^{\delta} \leq (\log R)^{\gamma}$ for $\gamma = 2\gamma_0$.  That is, the new domain produced by one application of Step 2 has diameter which is poly-log in the diameter of the initial domain.  Suppose that $n_0$ is the smallest positive integer such that $\log^{(n_0)}(R) < 100 r$, where $\log^{(n_0)}$ indicates the $\log$ function applied $n_0$ times.  It is not difficult to see that if we run the argument of Step 2 successively $n_0$-times we are left with a domain with size which is polynomial in $r$.  Equation \eqref{te::eqn::step2_bound} implies that the sum of the $L^2$ error that we accrue from iterating this procedure is bounded from above by
\[ O_{\ol{\Lambda}} \left( \sum_{m=1}^{n_0} \frac{1}{\exp^{(m)}(c_0)} \right) \leq O_{\ol{\Lambda}} \left( \sum_{m=1}^\infty \frac{1}{\exp^{(m)}(c_0)} \right)  < \infty,\]
where $\exp^{(m)}$ denotes the exponential function applied $m$ times and $c_0 = c_0(\eta, \ol{\Lambda}) > 0$ is some fixed constant.
Since the final domain is polynomial in $r$, the desired result follows by another application of Step 1.
\end{proof}

We are now going to show that the conditional expectation of the height along the interface is uniformly larger than $a(v)$ for $v \in V_+$ and less than $B(v)$ for $v \in V_-$.

\begin{lemma}
\label{te::lem::expectation_pos}
Assume $(\partial), (C)$, and that $r > (\log R)^{2 c_1}$ with $c_1 = c_1(1/4)$ as in Lemma \ref{te::lem::expectation}.  Suppose $v \in U$ is such that the connected component of $U \cap B(v,r)$ containing $v$ has non-empty intersection with $\partial B(v,r)$.  Then
\begin{align}
 \label{te::eqn::abs_upper_lower_bound}
  \E[h(v) - a(v) |\CK] &\geq \tfrac{1}{c_3} \text{ for } v \in V_+ \cup V,\\
  \E[b(v) - h(v)|\CK] &\geq \tfrac{1}{c_3} \text{ for } v \in V_- \cup V
\end{align}
for $c_3 > 0$ a universal constant.
\end{lemma}
\begin{proof}
Fix $v \in V_+$ which satisfies the hypotheses of the lemma and let $v_1,\ldots,v_m$ be the neighbors of $v$ in $D$.  Let $M > 0$ be some fixed positive constant.  By the explicit form of the law of $h$ conditional on $h(v_1),\ldots,h(v_m), \CK$, we obviously have that
\[ \E[h(v) - a(v)\big| |h(v_1)|,\ldots,|h(v_m)| \leq M, \CK] \geq c(M) > 0.\]
This yields \eqref{te::eqn::abs_upper_lower_bound} since \eqref{te::eqn::exponential_moment} implies
\[ \p[|h(v_1)|,\ldots,|h(v_m)| \leq M | \CK] \geq \epsilon_1 > 0.\]
\end{proof}

We are now going to prove that the mean height of the field at a point $v$ in $U$ remains uniformly bounded conditional on the boundary data of the field in a large ball around $v$ provided that it is of at most logarithmic height.  The proof follows by reusing the coupling and stochastic domination procedure from the proof of Lemma \ref{te::lem::expectation}.

\begin{lemma}
\label{te::lem::local_bound}
Assume ($\partial$) and $(C)$.  For each $\alpha > 0$ there exists $p_0 > 0$ and $c_4 = c_4(\ol{\Lambda},\alpha)$ such that the following holds.  For each $r \geq 0$ and $v \in U$ such that the connected component of $U \cap B(v,r)$ containing $v$ has non-empty intersection with $\partial B(v,r)$, we have
\begin{equation}
\label{te::eqn::boundary_conditioned_mean}
  \E[ |h(v)| \big| \CK, h|_{\partial B(v,\wt{r})}] \one_{A^c} \leq c_4
\end{equation}
for every $(\log r)^{p_0} \leq \wt{r} \leq r$, and $A = \{ \max_{x \in B(v,r)} |h(x)| > \alpha \log r\}$.
\end{lemma}
\begin{proof}
Without loss of generality, it suffices to consider $v \in V_+$ since the argument is symmetric for $v \in V_-$ and is trivial if $v \in V$.  We re-apply the idea of Step 2 from the proof of Lemma \ref{te::lem::expectation}.  With $W, D_W$ as in the proof of Lemma \ref{te::lem::expectation}, let $\zeta = h|_{\partial B_W}$ where $B_W = B(x,\wt{r}) \cap D_W$ and let $h^{\zeta}$ have the law of the GL model on $B_W$ with boundary condition $\zeta$ conditional on $\cap_{x \in B_W} \{a(x) \leq h^{\zeta}(x) \leq b(x)\}$.  Let $\partial_1$ denote the part of $\partial B_W$ which does not intersect $W$ and $\partial_2$ the part which is contained in $W$.  Assume that $h^{\zeta,0}$ has the law of the GL model on $B_W$ with $h^{\zeta,0}|_{\partial_1} \equiv 0$, $h^{\zeta,0}|_{\partial_2} \equiv \zeta$, and the same conditioning as $h^{\zeta}$.  With $(h_t^{\zeta}, h_t^{\zeta,0})$ the stationary coupling of $h^\zeta,h^{\zeta,0}$, Lemma \ref{gl::lem::square_bound} implies $\ol{h}_0^2(v) \leq \E_v[\ol{h}_{-\tau}^2(X_{\tau})]$ where $\tau$ is the first exit time of $X = X^0$ from $B_W$.  Using that $\ol{h} = h^\zeta - h^{\zeta,0} \equiv 0$ on $\partial_2$, Lemma \ref{symm_rw::lem::beurling} of \cite{M10} thus implies $\ol{h}_0^2(v) \leq O_{\ol{\Lambda}}( \wt{r}^{-\rho_B} \max_{x \in \partial_1} |\zeta(x)|^2)$.  Therefore
\begin{equation}
\label{te::eqn::part2_bound}
 \E[ \ol{h}_T^2(v) | \zeta] \one_{A^c} = O_{\alpha}( \wt{r}^{-\rho_{\rm B}} (\log r)^2).
\end{equation}
Assume now $p_0 > 2 / \rho_{\rm B}$ so that $\wt{r}^{-\rho_{\rm B}} (\log r)^2 = O(1)$.  Then the right hand side of \eqref{te::eqn::part2_bound} is $O_{\alpha}(1)$.  Therefore it suffices to show that $\E[ |h^{\zeta,0}(v)| \big| \zeta] \one_{A^c} = O_{\ol{\Lambda}}(1)$.  Since $h^{\zeta,0}(v) \geq -\ol{\Lambda}$, we actually just need to prove $\E[ h^{\zeta,0}(v) \big| \zeta] \one_{A^c} = O_{\ol{\Lambda}}(1)$.  Let $h^{\ol{\Lambda}}$ have the law of the GL model on $B_W$ with $h^{\ol{\Lambda}} |_{\partial_1} \equiv 0$, $h^{\ol{\Lambda}} |_{\partial_2} \equiv \ol{\Lambda}$, and the same conditioning as $h^\zeta$.  As $\zeta|_{\partial_2} \leq \ol{\Lambda}$, Lemma \ref{gl::lem::stoch_dom} implies that the stationary coupling $(h_t^{\zeta,0}, h_t^{\ol{\Lambda}})$ satisfies $h_t^{\ol{\Lambda}} \geq h_t^{\zeta,0}$ almost surely.  As $h_t^{\ol{\Lambda}}$ satisfies the hypotheses of Lemma \ref{te::lem::expectation}, we consequently have $\E[ h^{\ol{\Lambda}}(v)] = O_{\ol{\Lambda}}(1)$, hence $\E[ h^{\zeta,0}(v) | \zeta] \one_{A^c} = O_{\ol{\Lambda}}(1)$.
\end{proof}

\section{Coupling Near the Interface}
\label{sec::hic}

Throughout this section, we shall assume $(\partial)$ and $(C)$.  Suppose that $h_t$ solves \eqref{gl::eqn::cond_dynam} initialized at stationarity and conditioned so that $a(x) \leq h_t(x) \leq b(x)$ for every $x \in D$.  Our first goal will be to show that the law of $h_t(z)$ near some $x_0 \in U$ does not depend too strongly on the precise geometry of $U$ nor $D$ far away from $x_0$ (Proposition \ref{hic::prop::hic}).  The next objective is to boost the estimate of approximate harmonicity of the mean given by Theorem \ref{harm::thm::coupling} of \cite{M10} very close to $\partial D$ and $U$.  This will, in particular, prove Theorem \ref{intro::thm::harmonic_up_to_boundary}.  We end the section by combining Proposition \ref{hic::prop::hic} with Lemma \ref{hic::lem::harmonic_boundary} to show that, under the additional hypothesis of $(\pm)$, the mean height remains uniformly negative close to $V_-$ and uniformly positive near $V_+$.  We remark that the latter is one of the crucial points of the proof.

\subsection{Continuity of the Law Near $U$}

We now work towards establishing Proposition \ref{hic::prop::hic}.  Before we proceed, it will be helpful to give an overview of the proof.  We will first argue (Lemma \ref{hic::lem::min}) that along $U$ there are many points $y$ where $h_t(y)$ is very close to either $a(y)$ or $b(y)$.  The reason that one should expect this to be true is that, for any fixed $y$, this holds with positive probability and, using the Markovian structure, we are able to argue a certain amount of approximate independence between different $y$.  Then we will fix another instance $\wt{h}_t$ of the GL model, though on possibly a different domain $\wt{D}$ and region of conditioning $\wt{U}$ which agrees with $U$ near $x_0$, and take the stationary coupling of $h_t$ and $\wt{h}_t$.  By the energy inequality, we can find large, connected, non-random subsets of deterministic bonds $b$ in $U$ near $x$ where $\E | \nabla \ol{h}_t(b)|$ is small, with $\ol{h} = h - \wt{h}$ as usual.  This implies $\ol{h}_t$ is nearly constant throughout each region.  We will then combine this with Lemma \ref{hic::lem::min} to argue that this constant must be very close to zero, for otherwise either $h$ or $\wt{h}$ would violate the constraint (C).  The result then follows by recoupling $h,\wt{h}$ near $x_0$ with boundary values fixed in the ``good'' regions and then applying the random walk representation.

\begin{lemma}
\label{hic::lem::min}
Fix $\delta > 0$, $r > 0$, and $n = [r^\delta]$.  Suppose that $x_1,\ldots,x_n \in V_+ \cup V$ are distinct.  Assume that $|x_i-x_j| \geq 2r_0 \equiv 2(\log r)^{p_0}$, where $p_0$ is as in Lemma \ref{te::lem::local_bound}, and that for each $i$, the connected component $U_i$ of $U$ containing $x_i$ satisfies $U_i \cap \partial B(x_i,r_1) \neq \emptyset$ for $r_1 = (\log R)^{2c_1}$, $c_1$ as in Lemma \ref{te::lem::expectation}.  For each $\epsilon > 0$, we have
\begin{equation}
\label{hic::eqn::min}
 \p[ \cap_{k=1}^n\{|h(x_k) - a(x_k)| \geq n^{\epsilon-1}\}]  = O_{\ol{\Lambda}, \epsilon}(r^{-50}).
\end{equation}
Similarly, if $x_1,\ldots,x_n$ are distinct in $V_- \cup V$, then
\begin{equation}
\label{hic::eqn::max}
 \p[ \cap_{k=1}^n\{|h(x_k) - b(x_k)| \geq n^{\epsilon-1}\}] = O_{\ol{\Lambda},\epsilon}(r^{-50}).
\end{equation}
\end{lemma}
\begin{proof}
Without loss of generality, if suffices to prove \eqref{hic::eqn::min}.  To this end, for each $k$, let $x_{k1},\ldots,x_{km_k}$, $m_k \leq 4$, be the neighbors of $x_k$ in $D$.  By Lemma \ref{te::lem::expectation}, we know that if $\alpha > 0$ is large enough, then the event 
\[ A = \cup_{k} \{ \max_{x \in B(x_k,r)} |h(x)| > \alpha (\log r)\}\]
satisfies $\p[A] = O_{\ol{\Lambda}}(r^{-50})$ and, by Lemma \ref{te::lem::local_bound},
\[ \E\big[ |h(x_{ki})|\ \big|\ h|_{\partial B(x_k,r_0)} \big] \one_{A^c} = O_{\ol{\Lambda}}(1).\]
Combining this with Markov's inequality implies the existence of $M = M(\ol{\Lambda}) > 0$ sufficiently large such that for each $k$ we have 
\begin{equation}
\label{te::eqn::neighbor_bound}
\p[ E_k\  \big|\  h|_{\partial B(x_k,r_0)} ] \one_{A^c} \geq \frac{1}{2} \one_{A^c} \text{ where }
 E_k = \cap_{\ell=1}^{m_k} \{ |h(x_{k\ell})| \leq M\}.
\end{equation}
From the explicit form of the density of the law of $h(x_k)$ conditional on $h(x_{k1}),\ldots,h(x_{km_k})$ with respect to Lebesgue measure, it is clear that
\begin{equation}
\label{hic::eqn::prob_lower_bound}
\p[ h(x_k) - a(x_k) \leq \beta \big| E_k] \geq a_1 \beta
\end{equation}
for some $a_1 = a_1(M) > 0$ and all $\beta \in [0,\beta_0]$ for some $0 < \beta_0 = \beta_0(M)$.  Let
\[ B_M = \{ 1 \leq k \leq n : |h(x_{k1})| \leq M, \ldots, |h(x_{km_k})| \leq M\}.\]
It is immediate from \eqref{te::eqn::neighbor_bound} that there exists a random variable $Z$ which, conditional on $A^c$, is binomial with parameters $(n,\tfrac{1}{2})$ such that $|B_M| \one_{A^c} \geq Z \one_{A^c}$.  Consequently,
\begin{align*}
      \p[ |B_M| \geq \tfrac{1}{4} n \big| A^c]
&\geq \p[ Z \geq \tfrac{1}{4} n \big| A^c]
  \geq 1-O(e^{-a_2 n}),
\end{align*}
some $a_2 > 0$.  The lemma now follows as by \eqref{hic::eqn::prob_lower_bound},
\begin{align*}
  \p[ \cap_{k=1}^n \{h(x_k) - a(x_k) \geq n^{\epsilon-1}\} \big| |B_M| \geq \tfrac{1}{4} n]
&\leq (1-a_1 n^{\epsilon-1})^{n/4}.
\end{align*}
\end{proof}

We assume that $\wt{D} \subseteq \Z^2$ is another bounded domain with distinguished subsets of vertices $\wt{V}_-,\wt{V}_+, \wt{V}$ and with functions $\wt{a} \leq \wt{b}$ satisfying the hypotheses of $(C)$.  Let $\wt{h}_t$ solve \eqref{gl::eqn::cond_dynam} with stationary initial conditions, conditioned to satisfy $\wt{a}(x) \leq \wt{h}(x) \leq \wt{b}(x)$ for all $x \in \wt{D}$, and boundary condition satisfying $(\partial)$.  Further, we suppose there exists $x_0 \in D \cap \wt{D}$ and $r \geq 5(\log R)^{2c_1}$, $c_1 > 0$ as in Lemma \ref{te::lem::expectation}, such that 
\begin{enumerate}
 \item \label{hic::assump::ball_contained} $B(x_0,2r) \subseteq D \cap \wt{D}$,
 \item \label{hic::assump::agree_geom} $B(x_0,2r) \cap U = B(x_0,2r) \cap \wt{U}$,
 \item \label{hic::assump::agree_cond} $a = \wt{a}, b = \wt{b}$ in $B(x_0,2r)$, and
 \item \label{hic::assump::curve} the connected component of $U_0 \equiv U \cap B(x_0,2r)$ containing $x_0$ has non-empty intersection with $\partial B(x_0,2r)$. 
\end{enumerate}

\begin{figure}
     \centering
          \includegraphics[width=.6\textwidth]{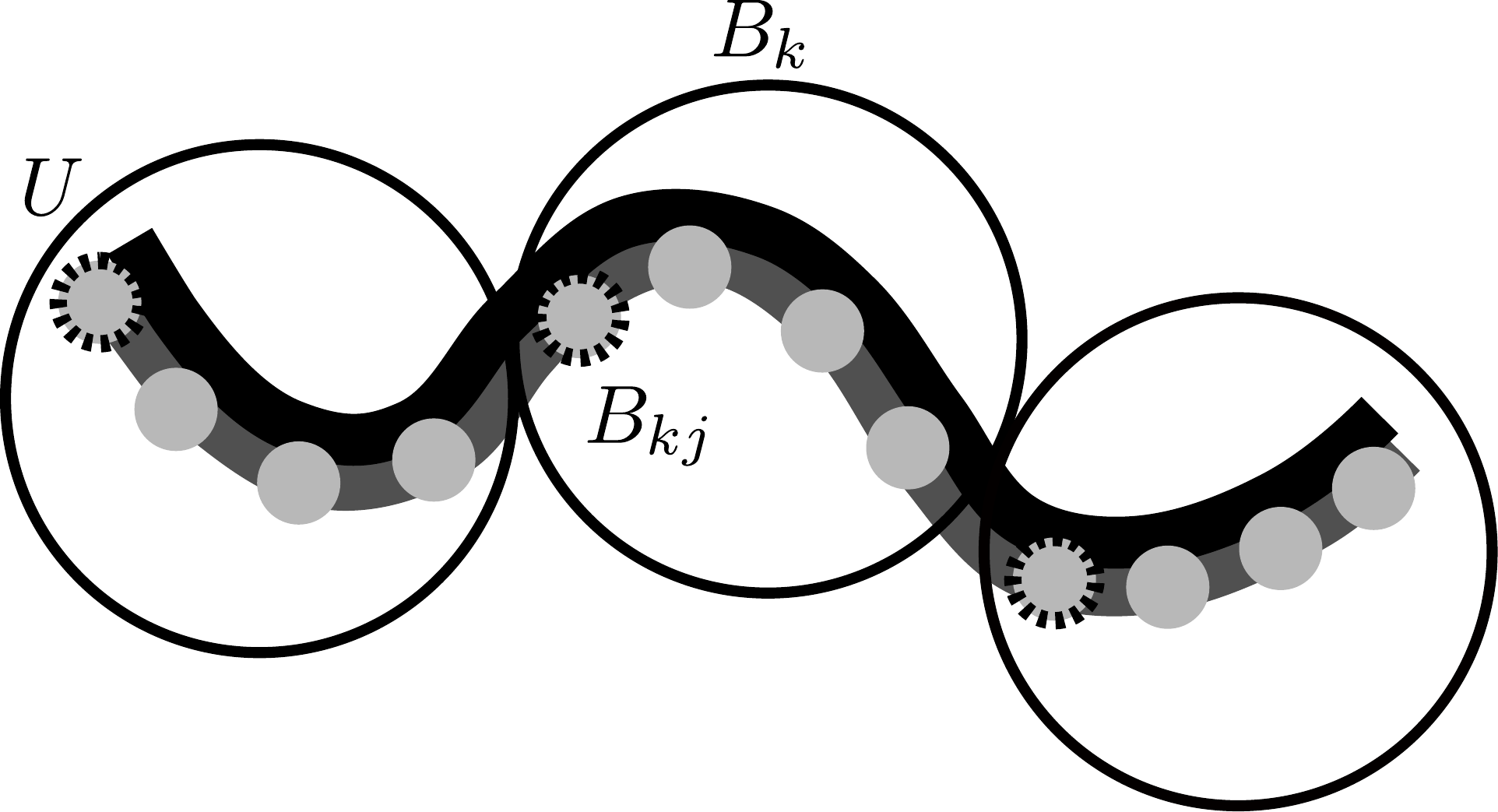}
          \caption{The setup for constructing the net $Y$ in the first step of the proof of Proposition \ref{hic::prop::hic}.  The large circles indicate the balls $B_k$ associated with the initial $r^{99/100}$ net $(x_n : n \leq N_r)$ and the smaller disks are the balls $B_{kj}$ of the corresponding nets of the $B_k$.  The collection of disks with dashed boundary indicates one the groups $\cup_k B_{kj}$ used to construct $Y$.}
\end{figure}

\begin{proposition}
\label{hic::prop::hic}
For every $\epsilon > 0$, there exists $\delta > 0$ independent of $r,h,\wt{h}$ such that there is a coupling of the laws of $h$, $\wt{h}$ satisfying
\[ \E\bigg[ \max_{x \in B(x_0,r^{1-\epsilon})} |h(x) - \wt{h}(x)| \bigg] = O_{\ol{\Lambda}}(r^{-\delta}).\]
\end{proposition}

We remark that the coupling constructed in Proposition \ref{hic::prop::hic} will not be the same as the stationary coupling.

\begin{proof}
{\it Step 1.}  Construction of the net.  First, Lemma \ref{te::lem::expectation} implies 
\[ \max_{z \in B(x_0,2r)} \E[ h_t^2(z) + \wt{h}_t^2(z)] = O_{\ol{\Lambda}}( (\log r)^2).\]
Combining this with Lemma \ref{gl::lem::grad_error} and assumptions \eqref{hic::assump::ball_contained}-\eqref{hic::assump::curve}, we see that the stationary coupling of $(h_t,\wt{h}_t)$ satisfies
\begin{equation}
\label{hic::eqn::grad_bound}
\sum_{b \in B^*(x_0,r)} \E |\nabla \ol{h}(b)|^2 = O_{\ol{\Lambda}}(r^{\epsilon}).
\end{equation}
Let $(x_n : n \leq N_r)$ be an $r^{99/100}$-net of $U_0$ contained in $V_+ \cup V$ and let $n_0 = r^{1/4}$.  Assumption \eqref{hic::assump::curve} implies $U \cap B(x_k,r^{1/2}) \geq r^{1/2}$.  Hence, we can find an $r^{1/4}$-net $(x_{kj})$ of $U \cap B(x_k, r^{1/2})$ of cardinality at least $n_0$.  Let $B_{kj}  = B(x_{kj}, r^{1/20})$ and $U_{kj} = U \cap B_{kj}$.  Trivially, $|U_{kj}| \leq |B_{kj}| \leq 10 r^{1/10}$.  Since the balls $B_{kj}$ are pairwise disjoint, \eqref{hic::eqn::grad_bound} implies 
\[ \E \left[ \sum_{j=1}^{n_0} \left( \sum_{k=1}^{N_r} \sum_{b \in B_{kj}^*} |\nabla \ol{h}(b)|^2 \right)\right] = O_{\ol{\Lambda}}(r^{\epsilon}).\]
Therefore there exists $1 \leq j_0 \leq n_0$ such that
\[ \E \left[ \sum_{k=1}^{N_r} \sum_{b \in B_{kj_0}^*} |\nabla \ol{h}(b)|^2 \right] = O_{\ol{\Lambda}}( r^{\epsilon-1/4}).\]
Noting that $N_r = O(r^2 / r^{198/100}) = O(r^{1/50})$ and $|B_{kj_0}^*| = O(r^{1/10})$, the Cauchy-Schwarz inequality implies
\[ \E \left[ \sum_{k=1}^{N_r} \sum_{b \in B_{kj_0}^*} |\nabla \ol{h}(b)| \right] = \big( O(r^{1/50}) O(r^{1/10}) O_{\ol{\Lambda}}( r^{\epsilon-1/4}) \big)^{1/2} = O_{\ol{\Lambda}}(r^{-1/20}),\]
assuming we have chosen $\epsilon > 0$ small enough.
As each of the sets $B_{kj_0}$ is connected, we consequently have that for each $1 \leq k \leq N_r$ there exists (random) $e_{k}$ with
\begin{equation}
\label{hic::eqn::e_bound}
 \E \left[ \sum_{k=1}^{N_r} M_k \right] = O_{\ol{\Lambda}}(r^{-1/20}) \text{ where } M_k = \max_{y \in B_{kj_0}} |\ol{h}(y) - e_k|.
\end{equation}

We next claim that $\E |e_k| = O_{\ol{\Lambda}}(r^{-1/20})$ uniformly in $k$.  To see this, fix $y \in B_{kj_0} \cap (V_+ \cup V)$.  By rearranging the inequality $e_k - \ol{h}(y) \leq  M_k$ and using $\wt{h}(y) \geq a(y)$, we see that $h(y) - a(y) + M_k \geq e_k$.  By a symmetric argument except starting with the inequality $\ol{h}(y) - e_k \leq M_k$, we also have $\wt{h}(y) - a(y) + M_k \geq -e_k$.  Combining the two inequalities yields  
\[ -\wt{X}_k - M_k \leq e_k \leq X_k + M_k\]
where
\[ \wt{X}_k = \min_{y \in U_{kj_0}} |\wt{h}(y) - a(y)| \text{ and } X_k = \min_{y \in U_{kj_0}} |h(y) - a(y)|.\]
Let $E_k = \{ X_k \geq r^{-1/20}\}$ and $\wt{E}_k = \{ \wt{X}_k \geq r^{-1/20}\}$.  From Lemma \ref{hic::lem::min}, we have both $\p[E_k] = O_{\ol{\Lambda}}(r^{-50})$ and $\p[\wt{E}_k] = O_{\ol{\Lambda}}(r^{-50})$.  Note that
\[ \E[X_k] = O_{\ol{\Lambda}}(r^{-1/20}) + \E[X_k \one_{E_k}] = O_{\ol{\Lambda}}(r^{-1/20}) + \sqrt{ \E[X_k^2] O_{\ol{\Lambda}}(r^{-50})}.\]
By Lemma \ref{te::lem::expectation}, $\E[ X_k^2] = O_{\ol{\Lambda}}( (\log r)^2)$, hence $\E[X_k] = O(r^{-1/20})$.  Similarly, $\E[\wt{X}_k] = O(r^{-1/20})$.  Combining this with \eqref{hic::eqn::e_bound} implies
\[ \sum_{k=1}^{N_r} \E[|e_k|] \leq \sum_{k=1}^{N_r} \E[|M_k|] + \sum_{k=1}^{N_r}  \E \big[X_k + \wt{X}_k \big] = O_{\ol{\Lambda}}(r^{-1/50}).\]
By yet another application of \eqref{hic::eqn::e_bound}, for each $1 \leq k \leq N_r$ we can pick $y_k \in U_{kj_0}$ such that $Y = (y_k : 1 \leq k \leq N_r)$ is an $r^{99/100}$-net of $U_0$ satisfying
\[ \sum_{k=1}^{N_r} \E[|\ol{h}(y_k)|] = O_{\ol{\Lambda}}(r^{-1/50}).\]

{\it Step 2.}  Coupling at the interface.  Let $B_Y = B(x_0,r) \setminus Y$.  Conditional on $\zeta = h|_{\partial B_Y}$, let $h_t^{B_Y}$ be a dynamic version of the GL model on $B_Y$ with $h^{B_Y}|_{\partial B_Y} = \zeta$ and the same conditioning as $h$ off of $Y$.  Define $\wt{h}_t^{B_Y}$ analogously, let $(h_t^{B_Y}, \wt{h}_t^{B_Y})$ be the corresponding stationary coupling, and let $\ol{h}^{B_Y} = h^{B_Y} - \wt{h}^{B_Y}$.  With $X_t = X_t^0$ defined as in subsection \ref{subsec::rw_difference}, Lemma \ref{gl::lem::square_bound} implies that
\begin{equation}
\label{hic::eqn::cai_bound}(\ol{h}_0^{B_Y})^2(x) \leq \E_x [ (\ol{h}_{-  \tau}^{B_Y})^2(X_{\tau})]
\end{equation}
where $\tau = \tau_Y \wedge \tau_r$, 
\[ \tau_Y = \inf\{t > 0 : X_t \in Y\} \text{ and } \tau_r = \inf\{t > 0 : X_t \in \partial B(x_0,r)\},\]
and the expectation is taken only over the randomness of $X$ initialized at $x$.  We claim that there exists $\rho = \rho(\CV,\epsilon) > 0$ such that $\p_x[ \tau_Y \leq \tau_r] \geq 1-O(r^{-\rho})$ for $x \in B(x_0,r^{1-\epsilon})$.  The reason for this is that after hitting the center ring of the annulus $A_k = A(x_0, 2^k r^{1-\epsilon}, 2^{k+1} r^{1-\epsilon})$, $X$ runs a full circle around $A_k$ hence hits $U_0$ before exiting $A_k$ with positive probability (see the proof of Lemma \ref{symm_rw::lem::beurling} of \cite{M10}).  On this event, upon hitting $U_0$, there exists $y \in Y$ with distance at most $r^{99/100}$ of $X$, hence $X$ has positive probability of hitting $y$ before exiting $A_k$.  The claim now follows as there are at least $c(\epsilon) \log r$ chances for this to occur.  Consequently, by \eqref{hic::eqn::cai_bound} we have that
\[ (\ol{h}_0^{B_Y})^2(x) \leq \max_{y \in Y} |\ol{h}(y)| + O(r^{-\rho}) \max_{y \in \partial B(x_0,r)} |\ol{h}(y)| \equiv A_1 + O(r^{-\rho}) A_2.\]
The first part of the lemma implies $\E[A_1] = O_{\ol{\Lambda}}(r^{-1/50})$ and Lemma \ref{te::lem::expectation} implies $\E[A_2] = O_{\ol{\Lambda}}(\log r)^2$.  Taking $\delta = (\rho/2) \wedge \tfrac{1}{50}$ proves the proposition.
\end{proof}

\subsection{Harmonicity of the Mean Near the Boundary}

In view of Proposition \ref{hic::prop::hic}, we now boost the estimate of harmonicity of the mean coming from Theorem \ref{harm::thm::mean_harmonic} of \cite{M10} all of the way up to $\partial D$ and $U$.  This result is \emph{only} applicable for the mean; it does not imply that we can \emph{couple harmonically} up to the boundary.  Recall that $E(r) = \{x \in D : \dist(x, \partial E) \geq r\}$ for $E \subseteq D$.  Let $D_U = D \setminus U$.

\begin{figure}
     \centering
          \includegraphics[width=.7\textwidth]{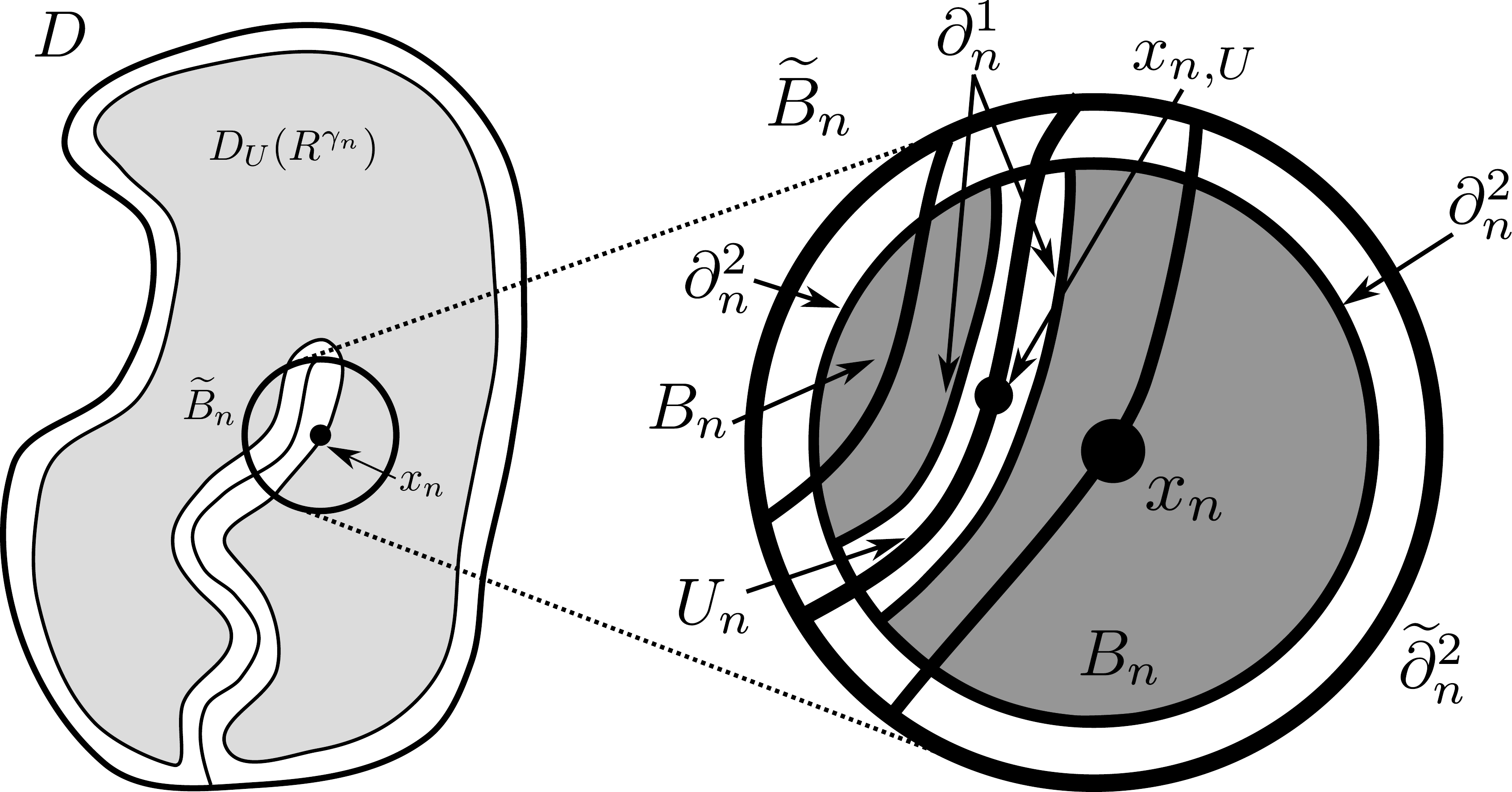}
          \caption{A typical step in the localization procedure used in the proof of Lemma \ref{hic::lem::harmonic_boundary}.  The dark gray region on the right side indicates $B_n$.}
\end{figure}

\begin{lemma}
\label{hic::lem::harmonic_boundary}
Assume that for every connected component $U_0$ of $U$ there exists a connected component $U_1$ of $U$ with $\dist(U_0, U_1) \leq (\diam(U_0))^2$ and $\diam(U_1) \geq (\log R)^{2c_1}$ with $c_1$ as in Lemma \ref{te::lem::expectation}.  There exists $\delta = \delta(\CV) > 0$ such that if $\wh{g}$ is the harmonic extension of $\E[h(x)]$ from $\partial D_U(r)$ to $D_U(r)$, then 
\[ \max_{x \in D_U(r)} |\E[h(x)] - \wh{g}(x)| = O_{\ol{\Lambda}}(r^{-\delta}).\]
\end{lemma}
\begin{proof}
We are going to provide a proof for the lemma under the stronger hypothesis that $U$ is connected (as is the case corresponding to the exploration path $\gamma$ of our main theorem), since moving to the more general case is straightforward though notationally more complicated.  Let $d(x) = |\E[h(x)] - \wh{g}(x)|$.
Fix $\epsilon,\delta > 0$ so that Theorem \ref{harm::thm::mean_harmonic} of \cite{M10} holds and let $\gamma_n = (1-\epsilon)^n$.  Let $x_1$ be a point in $D_U(R^{\gamma_1})$ which maximizes $d|_{D_U(R^{\gamma_1})}$.  For each $n \geq 2$, let $x_n$ be a point in $D_U(R^{\gamma_{n}}) \setminus D_U(R^{\gamma_{n-1}})$ which maximizes $d|_{D_U(R^{\gamma_{n}}) \setminus D_U(R^{\gamma_{n-1}})}$, and let $\Delta_n = d(x_n)$.  We are going to prove that
\begin{equation}
   \label{hic::eqn::delta_bound}
   \Delta_n \leq O_{\ol{\Lambda}}(R^{-c \delta \gamma_{n+1}}) + \Delta_{n+1},
\end{equation}
for some $c > 0$ which depends only on $\CV$.
The constant will be uniform in $n$, so that the result follows by summation.

We will first prove \eqref{hic::eqn::delta_bound} for $n=1$.  Let $\wh{g}_1$ be the harmonic extension of $\E[h(x)]$ from $\partial D_U(R^{\gamma_1})$ to $D_U(R^{\gamma_1})$.  Lemma \ref{te::lem::expectation} implies that with 
\[ A = \{ \max_{x \in \partial D_U} |h(x)| > \alpha (\log R)\},\]
we have $\p[A] = O_{\ol{\Lambda}}(R^{-100})$ provided $\alpha > 0$ is chosen large enough.  Applying Lemma \ref{te::lem::expectation} a second time along with the Cauchy-Schwarz inequality implies
\begin{equation}
\label{hic::eqn::harmonic_good_bc}
 \big|\E[ \E[ h(x) \big| h|_{\partial D_U}] \one_{A^c}] - \E[h(x)] \big| \leq (\E[ h^2(x)] \p[A])^{1/2} =  O_{\ol{\Lambda}}(R^{-10})
 \end{equation}
 for all $x \in D$ since $A$ is $\sigma(h|_{\partial D_U})$-measurable.  Theorem \ref{harm::thm::mean_harmonic} of \cite{M10} is applicable to $h|_{D_U}$ on $D_U$ conditional on $h|_{\partial D_U}$ and $A$, which combined with \eqref{hic::eqn::harmonic_good_bc} implies 
 \[ \Delta_1 \leq O_{\ol{\Lambda}}(R^{-\delta})+ |\wh{g}_1(x_1) - \wh{g}(x_1)|.\]
By the maximum principle for discrete harmonic functions, we know that there exists $\wt{x}_1 \in \partial D_U(R^{\gamma_1}) \subseteq D_U(R^{\gamma_2}) \setminus D_U(R^{\gamma_1})$ such that $|\wh{g}_1(x_1) - \wh{g}(x_1)| \leq |\wh{g}_1(\wt{x}_1) - \wh{g}(\wt{x}_1)|$.  Applying Theorem \ref{harm::thm::mean_harmonic} of \cite{M10} a second time yields 
\[ |\wh{g}_1(\wt{x}_1) - \wh{g}(\wt{x}_1)| \leq O_{\ol{\Lambda}}(R^{-\delta}) + d(\wt{x}_1) \leq O_{\ol{\Lambda}}(R^{-\delta}) + \Delta_2,\]
which gives \eqref{hic::eqn::delta_bound} for $n=1$, as desired.

We are now going to prove \eqref{hic::eqn::delta_bound} for $n \geq 2$.  Let $\wt{\gamma}_n = (1-\epsilon/10)\gamma_{n-1}$, $\gamma_n' = (1-\epsilon/3)\gamma_{n-1}$, $\wt{B}_n = B(x_n,R^{\wt{\gamma}_{n-1}}) \cap D$, and $B_n = B(x_n, R^{\gamma_{n-1}'}) \cap D_U(R^{\gamma_{n+1}})$.  Let $\wt{\partial}_n^1$ be the part of $\partial \wt{B}_n$ which is contained in $\partial D$ and $\wt{\partial}_n^2 = \partial \wt{B}_n \setminus \partial D$.  Let $h_n$ have the law of the GL model on $\wt{B}_n$ with $h_n|_{\wt{\partial}_n^1} \equiv h|_{\wt{\partial}_n^1}$, $h_n|_{\wt{\partial}_n^2} \equiv 0$, and the same conditioning as $h$ otherwise.  By decreasing $\delta > 0$ if necessary, Proposition \ref{hic::prop::hic} implies that we can couple $h,h_n$ such that $\max_{x \in B_n} \E[| h(x) - h_n(x)|] = O(R^{-\wt{\gamma}_n \delta})$.  Let $\wh{g}_n$ be the harmonic extension of $\E[h_n(x)]$ from $\partial B_n$ to $B_n$.  Since $x \in B_n$ implies that $\dist(x, \partial \wt{B}_n \cup U) \geq R^{\gamma_{n+1}}$, Lemma \ref{te::lem::expectation} and Theorem \ref{harm::thm::mean_harmonic} imply that $\max_{x \in B_n}|\E[h_n(x)] - \wh{g}_n(x)| = O_{\ol{\Lambda}}(R^{-\gamma_{n+1} \delta})$, hence
\begin{equation}
\label{hic::eqn::harm_diff_bound_n}
\max_{x \in B_n}|\E[h(x)] - \wh{g}_n(x)| = O_{\ol{\Lambda}}(R^{-\gamma_{n+1} \delta}).
\end{equation}
Therefore
\begin{equation}
\label{hic::eqn::delta_n_prebound}
 \Delta_n \leq O_{\ol{\Lambda}}(R^{-\gamma_{n+1} \delta}) + |\wh{g}(x_n) - \wh{g}_n(x_n)|.
 \end{equation}

We can divide the boundary of $B_{n}$ into the part $\partial_{n}^1$ which intersects $\partial D_U(R^{\gamma_{n+1}})$ and $\partial_{n}^2 = \partial B_{n} \setminus \partial_{n}^1$.  We claim that the harmonic measure of $\partial_n^1$ from $x_n$ in $B_n$ is $1-O(R^{-\rho_{\rm B} \delta \gamma_{n-1}})$ provided we take $\delta < \epsilon/100$.  To see this, let $x_{n,U}$ be a point in $U$ with minimal distance to $x_n$.  Note that $d(x_n,x_{n,U}) \leq R^{\gamma_{n-1}}$.  Since $U$ is connected, there exists a connected subgraph $U_n$ of $U$ contained in $B(x_n, R^{\gamma_{n-1}'}) \cap D$ which itself contains $x_{n,U}$ and has non-empty intersection with $\partial ( B(x_n, R^{\gamma_{n-1}'}) \cap D)$.  Consequently, Lemma \ref{symm_rw::lem::beurling} of \cite{M10} implies that the probability that a random walk started at $x_n$ exits $B(x_n,R^{\gamma_{n-1}'}) \cap D$ before hitting $U_n$ is at most $O((R^{\gamma_{n-1}} / R^{\gamma_{n-1}'})^{\rho_{\rm B}}) = O(R^{-\delta \rho_{\rm B} \gamma_{n-1}})$ since $\delta < \epsilon/100$.  This proves our claim.

Letting $M_n^i = \max_{x \in \partial_{n}^i} |\wh{g}(x) - \wh{g}_n(x)|$, we thus see that
\[ |\wh{g}(x_n) - \wh{g}_n(x_n)| \leq M_n^1 + O(R^{- \rho_{\rm B} \delta \gamma_{n-1}}) M_n^2.\]
Equation \eqref{hic::eqn::harm_diff_bound_n} and the definition of $\Delta_{n+1}$ implies that
\[ M_n^1 \leq \Delta_{n+1} + O_{\ol{\Lambda}}(R^{-\delta \gamma_{n+1}}),\]
hence
\[ \Delta_n \leq \Delta_{n+1} + O(R^{- \rho_{\rm B} \delta \gamma_{n-1}}) M_{n}^2 + O_{\ol{\Lambda}}(R^{-\delta \gamma_{n+1}}).\]
Lemma \ref{te::lem::expectation} implies $M_n^2 = O_{\ol{\Lambda}}(\log R^{\gamma_n})$ hence $O(R^{-\rho_{\rm B} \delta \gamma_{n-1}}) M_n^2 = O_{\ol{\Lambda}}(R^{-\rho_{\rm B}\delta \gamma_{n}})$, which gives exactly \eqref{hic::eqn::delta_bound} and proves the lemma.
\end{proof}

\subsection{Sign of the Mean Near $U$}

We will next show that the mean height is uniformly bounded in $D$, uniformly positive near $V_+$, and uniformly negative near $V_-$.

\begin{figure}
\includegraphics[width=0.5\textwidth]{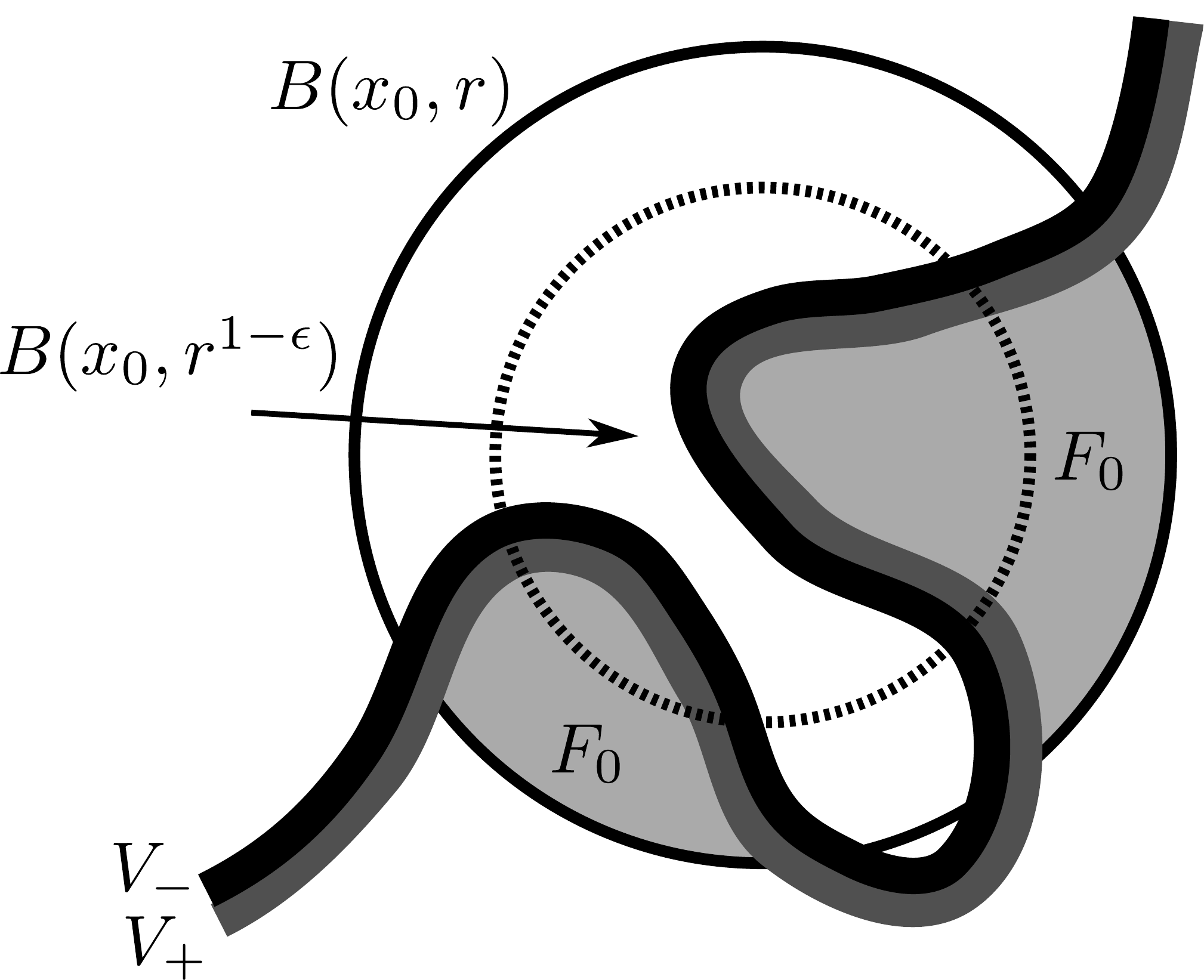}
\caption{The setup for Proposition \ref{hic::prop::exp_bounds}.  The regions shaded black and dark gray correspond to $V_-$ and $V_+$, respectively.  The light gray region is $F_0$ and the subset of $F_0$ surrounded by the disk with dashed boundary arc is $F$.}
\end{figure}

\begin{proposition}
\label{hic::prop::exp_bounds}
We assume $(\pm)$ and that $U$ is connected in addition to $(C)$ and $(\partial)$.  Suppose $r > 0$ and $x_0 \in V_+$ are such that the boundary of every connected component of $B(x_0,r) \setminus U$ does not contain vertices from both $V_-$ and $V_+$.  Let $\CC_+$ be the set of connected components of $B(x_0,r) \setminus U$ whose boundary has non-empty intersection with $V_+$ and let $F_0 = \cup_{C \in \CC_+} C$.  Fix $\epsilon > 0$ and let $F = F_0 \cap B(x_0, r^{1-\epsilon})$.  There exists $\lambda_0 = \lambda_0(\epsilon,\ol{\Lambda}) > 0$ such that
\[ \frac{1}{\lambda_0} \leq \E[h(x)]  \leq \lambda_0 \text{ for all } x \in F.\]
\end{proposition}

The easy part is the upper bound: this follows by using Lemma \ref{hic::lem::harmonic_boundary} and the maximum principle to reduce it to bounding $\E[h(x)]$ for $x$ with $\dist(x,\partial D)$ uniformly bounded, then applying Proposition \ref{hic::prop::hic}.  The lower bound is more challenging.

\begin{lemma}
\label{hic::lem::exp_upper_bound}
Suppose that we have the same assumptions as Lemma \ref{hic::lem::harmonic_boundary}.  There exists $\lambda_0 = \lambda_0(\epsilon, \ol{\Lambda}) > 0$ such that $\E[ h(x)] \leq \lambda_0$ for every $x \in D$.
\end{lemma}
\begin{proof}
Fix $s \geq 1$ sufficiently large that Lemma \ref{hic::lem::harmonic_boundary} applies.  Let $D_U = D \setminus U$.  If $\dist(x,\partial D_U) \leq s$, then Lemma \ref{te::lem::expectation} implies $\E[h(x)] = O_{\ol{\Lambda}}(1)$.  It thus suffices to prove the bound on $D_U(s)$.  Applying Lemma \ref{hic::lem::harmonic_boundary}, we see that if $\wh{g}$ denotes the harmonic extension of $\E[h(x)]$ from $\partial D_U(s)$ to $D_U(s)$, then $|\E[h(x)] - \wh{g}(x)| = O_{\ol{\Lambda}}(1)$ uniformly in $x \in D_U(s)$.  By the maximum principle for discrete harmonic functions, the maximum of $\wh{g}$ in $D_U(s)$ is attained at some point $y_0 \in \partial D_U(s)$.  Consequently,
\[ \E[ h(x)]  \leq O_{\ol{\Lambda}}(1) + |\wh{g}(x)| \leq  O_{\ol{\Lambda}}(1) + |\wh{g}(y_0)| \leq  O_{\ol{\Lambda}}(1) + |\E[h(y_0)]|.\]
 Lemma \ref{te::lem::expectation} implies that the right hand side is $O_{\ol{\Lambda}}(1)$, which proves the lemma.
\end{proof}

The proof of the lower bound will also use Lemma \ref{hic::lem::harmonic_boundary} to reduce the problem to a boundary computation: we will show that $\E[h(x)]$ is uniformly positive very near $V_+$.  This strategy is a bit more difficult to implement in this case, however, since we need to show that this uniform positivity is enough to dominate the error associated with approximating $\E[h(x)]$ by the harmonic extension of its boundary values.  We will deduce this by arguing that along, say, the positive side of the interface, points at which the height is larger than a given threshold are typically not too far from each other.  Then, we will invoke the HS representation of the mean combined with a uniform lower bound on the probability that the HS walk hits any one of these points.

\begin{figure}
     \centering
          \includegraphics[width=.4\textwidth]{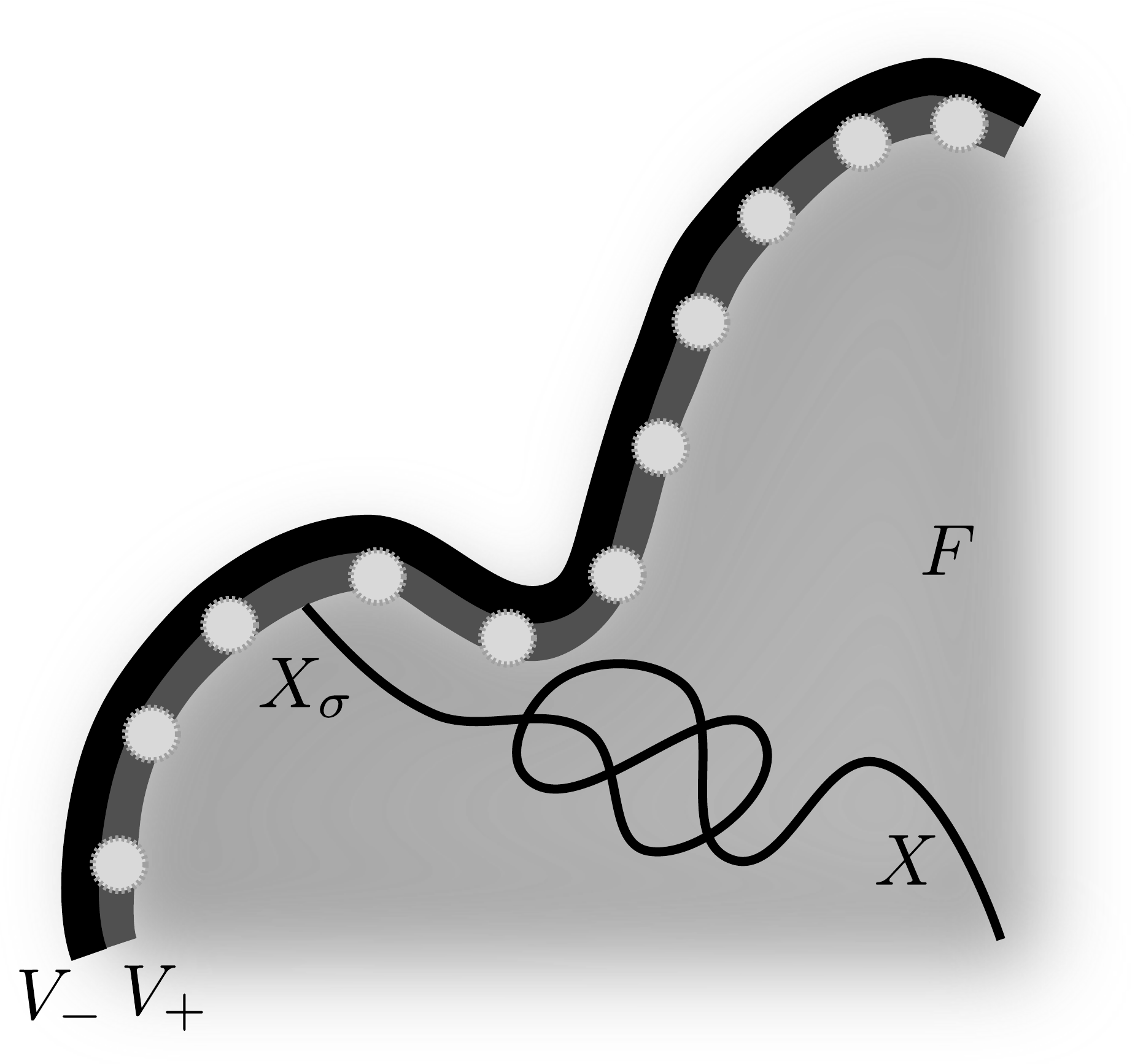}
          \caption{The idea of the proof of Lemma \ref{hic::lem::exp_lower_bound} is to show that with high probability we can find a $\sqrt{\log r}$ net $Y$ of $V_+$, indicated by the light gray disks, such that $h|_{Y} \geq r^{-a}$.  Thus the position of the HS random walk $X_\sigma$ upon first becoming adjacent to $V_+$ is within $\sqrt{\log r}$ jumps of exiting in $Y$.  Since the jump rates of $X$ are bounded, $X$ first enters $V_+$ in $Y$ with probability at least  $\rho^{\sqrt{\log r}}$, some $\rho = \rho(\CV) > 0$.}
\end{figure}

\begin{lemma}
\label{hic::lem::exp_lower_bound}
Suppose we have the same setup as in Proposition \ref{hic::prop::exp_bounds}.  For every $\epsilon > 0$ and $a > 0$ there exists $c(a,\epsilon)$ such that with $F = F_0 \cap B(x_0, r^{1-\epsilon})$ we have
\begin{equation}
\label{hic::eqn::min_decay_lem}
\E[h(x)] \geq c(a,\epsilon) r^{-a} \text{ for all } x \in F.
\end{equation}
\end{lemma}
\begin{proof}
Notationally, it will be easier for us to establish \eqref{hic::eqn::min_decay_lem} with $a$ replaced by $2a$: there exists $c = c(a,\epsilon)$ such that
\begin{equation}
\label{hic::eqn::min_decay}
\min_{y \in F} \E[ h(y)] \geq c(a,\epsilon) r^{-2 a}.
\end{equation}
Let $B = B(x_0,r)$, $V_\pm^B = V_\pm \cap B$.  For $x,y \in B(x_0,r)$ let $d_P(x,y)$ denote the length of the shortest path in $B \setminus (V_+ \cup V_-)$ which connects $x$ to $y$ and set set $d_P(x,y) = \infty$ if there is no such path.  Fix $z_0 \in F_0$ and assume that the connected component $C_0$ of $F_0$ containing $z_0$ has diameter $s > 0$ with respect to $d_P$.  Let $y_1,\ldots,y_n$ be a $\sqrt{\log s}$ net of the subset of $V_+^B$ which is adjacent to $C_0$ with respect to $d_P$.  Fix $M > 0$ and, for each $i$, let $Y_i = \{y_{i1},\ldots,y_{im_i}\}$ be an $M$ net of $B_{d_P}(y_i,\sqrt{\log s}) \cap V_+^B$.  Obviously, $\sqrt{\log s} \leq m_i \leq \log s$ for each $i$.  Fix $a > 0$ arbitrary, let $E_{ij}^a = \{ h(y_{ij}) \geq s^{-a}\}$, and $G_{ij}^a = \cap_{k \neq j} (E_{ik}^a)^c$.  By Proposition \ref{hic::prop::hic}, it follows that if $y \sim y_{ij}$ then $\E[|h(y)| \big| G_{ij}^a] - \E[|h(y)|] = O_{\ol{\Lambda}}(1)$ hence $\E[|h(y)| \big| G_{ij}^a] = O_{\ol{\Lambda}}(1)$.  As in the proof of Lemma \ref{hic::lem::min}, this in turn implies
\begin{equation}
\label{hic::eqn::approx_ind}
 \p[ (E_{ij}^a)^c \big| G_{ij}^a] \leq a_1 s^{-a}
\end{equation}
for some constant $a_1 > 0$.  Indeed, as we are able to bound the mean heights of $h(y)$ for $y \sim y_{ij}$ conditional on $G_{ij}^a$, we can use Markov's inequality to show that $h$ is uniformly bounded at the neighbors of $y_{ij}$ with uniformly positive probability.  Conditioning further on this event, the desired result is clear from the explicit form of the conditional density of $h(y_{ij})$.  With $\wt{G}_{ij}^a$ the intersection of any combination of $E_{ik}^a$ or $(E_{ik}^a)^c$ over $k \neq j$, we see that we can couple together $h|\wt{G}_{ij}^a$ and $h|G_{ij}^a$ such that $h|\wt{G}_{ij}^a \geq h|G_{ij}^a$ by Lemma \ref{gl::lem::stoch_dom}.  By \eqref{hic::eqn::approx_ind} we therefore have
\begin{equation}
 \p[ (E_{ij}^a)^c \big| \wt{G}_{ij}^a] \leq \p[ (E_{ij}^a)^c \big| G_{ij}^a] \leq  a_1 s^{-a}.
\end{equation}
Consequently,
\begin{align}
 \log \p[ \cap_{j=1}^{m_i} (E_{ij}^a)^c]
&= \log \p[ (E_{1j}^a)^c \big| \cap_{j=2}^{m_i} (E_{ij}^a)^c] + \log \p[ \cap_{j=2}^{m_i} (E_{ij}^a)^c] \notag\\
&\leq a_1 - a(\log s) + \log \p[ \cap_{j=2}^{m_i} (E_{ij}^a)^c] \label{hic::eqn::e_bound1}.
\end{align}
Using that $\sqrt{\log s} \leq m_i \leq \log s$ and iterating \eqref{hic::eqn::e_bound1}, we thus see that
\begin{align}
 \log \p[ \cap_{j=1}^{m_i} (E_{ij}^a)^c]
&\leq c_1(\log s) - a(\log s)^{3/2}
  \leq -\frac{a}{2} (\log s)^{3/2} \label{hic::eqn::e_int_bound}
\end{align}
for $s$ sufficiently large.  Therefore
\begin{equation}
 \log \p[ E] \leq -\frac{a}{4}(\log s)^{3/2} \label{hic::eqn::log_e_bound}
\end{equation}
where $E = \cup_i \cap_{j=1}^{m_i} (E_{ij}^a)^c$, again for $s > 0$ sufficiently large.  The reason for this is that the number of elements in the outer union in the definition of $E$ is clearly polynomial in $s$, so \eqref{hic::eqn::log_e_bound} follows from \eqref{hic::eqn::e_int_bound} by a union bound.  Thus 
\[ |\E[h(x) \one_{E}]| \leq (\E[ |h(x)|^2 ])^{1/2} [\p[E]]^{1/2} = O_{\ol{\Lambda}}(s^{-100}),\]
hence to prove the lemma it suffices to show that $|\E[h(x) \one_{E^c}]| \geq c s^{-a}$ for $s$ sufficiently large.  Let $B_U = B \setminus U$ and $\psi = h |_{\partial B_U}$.  By Lemma \ref{gl::lem::hs_mean_cov} of \cite{M10}, we have the HS representation for the conditional mean:
\begin{equation}
\label{hic::eqn::hs_mean}
 \E[h(x)|\psi] = \int_0^1 \E_x^{t\psi}[\psi(X_\tau)] dt,
\end{equation}
where under $\p_x^{t\psi}$, $X$ is the HS random walk on $D$ started at $x$ associated with the GL model on $B_U$ with boundary condition $t\psi$ and $\tau = \inf\{t : X_t \notin B_U\}$.  Our hypotheses imply
\begin{equation}
\label{hic::eqn::hit_vplus_bound}
\p_{z_0}[X_\tau \notin V_+^B] = O(r^{-\epsilon \rho_{\rm B}})
\end{equation}
for $\rho_{\rm B} > 0$ as in Lemma \ref{symm_rw::lem::beurling} of \cite{M10}.  Let $\sigma = \inf\{ t : \dist(X_t, V_+^B) = 1\}$.  On $E$, we know that $X_\sigma$ is at most $\sqrt{\log s}$ jumps from a site $y \in V_+^B$ such that $\psi(y) \geq s^{-a}$.  Therefore the probability that $X$ started at $X_\sigma$ exits at such $y$ is at least $\rho^{\sqrt{\log s}} \geq c_1(a)s^{-a}$, some $\rho > 0$ depending only on $\CV$ and $c_1(a)$ depending only on $a$.  Combining this with \eqref{hic::eqn::hs_mean} and \eqref{hic::eqn::hit_vplus_bound} with $s > 0$, we have
\begin{align*}
          &\E[h(x)|\psi] 
\geq  \frac{c_1(a)}{2} s^{-2a} + O(r^{-\epsilon \rho_{\rm B}}) \| \psi\|_\infty,
\end{align*}
provided we take $r$ sufficiently large.  Lemma \ref{te::lem::expectation} implies $\E[\| \psi\|_\infty] = O_{\ol{\Lambda}}(\log r)$, hence integrating both sides over $\psi$ yields \eqref{hic::eqn::min_decay} as $s = O(r^2)$.
\end{proof}

\begin{proof}[Proof of Proposition \ref{hic::prop::exp_bounds}]
Let $\wt{h}$ have the law of the GL model on $B = B(x_0,r)$ with the same conditioning as $h$ and $\wt{h}|_{\partial B} \equiv 0$.  By Proposition \ref{hic::prop::hic}, we can find a coupling of $\wt{h}$ and $h$ such that $\max_{x \in B(x_0, r^{1-\epsilon})} \E|\wt{h}(x) - h(x)| = \epsilon_1 \equiv O_{\ol{\Lambda}}(r^{-\delta})$.  Let $B_U = B \setminus U$ and let $\wh{g}$ be the harmonic extension of $\E[\wt{h}(x)]$ from $\partial B_U(s)$ to $B_U(s)$.  Lemma \ref{hic::lem::harmonic_boundary} implies that $|\E[\wt{h}(x)] - \wh{g}(x)| = \epsilon_2 \equiv O_{\ol{\Lambda}}(s^{-\delta})$.  For $x \in F$, the harmonic measure of the part of $\partial B_U(s)$ which is not in $B(x_0, r^{1-\epsilon})$ is $\epsilon_3 \equiv O(r^{-\rho_{\rm B}\epsilon})$.  Assume $s > 0$ is chosen sufficiently large so that, with $a > 0$ chosen much smaller than $\delta, \epsilon$, the uniform lower bound of $c(a,\epsilon) s^{-a}$ dominates $\epsilon_1+\epsilon_2+\epsilon_3$.  Putting everything together, increasing $\lambda_0 > 0$ from Lemma \ref{hic::lem::exp_upper_bound} if necessary implies
\begin{align*}
 \E[h(x)] &\geq \epsilon_1 + \E[\wt{h}(x)] \geq \epsilon_1+\epsilon_2 + \wh{g}(x)\\
    &\geq \epsilon_1 + \epsilon_2 + \epsilon_3+ c(a,\epsilon) s^{-a} \geq \frac{1}{\lambda_0}.
 \end{align*}
\end{proof}

\section{Independence of Interfaces}
\label{sec::ni}

We show in this section that the geometry of zero height interfaces near a particular point $x_0$ is approximately independent from the exact geometry of those which are far from $x_0$, that is:

\begin{figure}
     \centering
          \includegraphics[width=.5\textwidth]{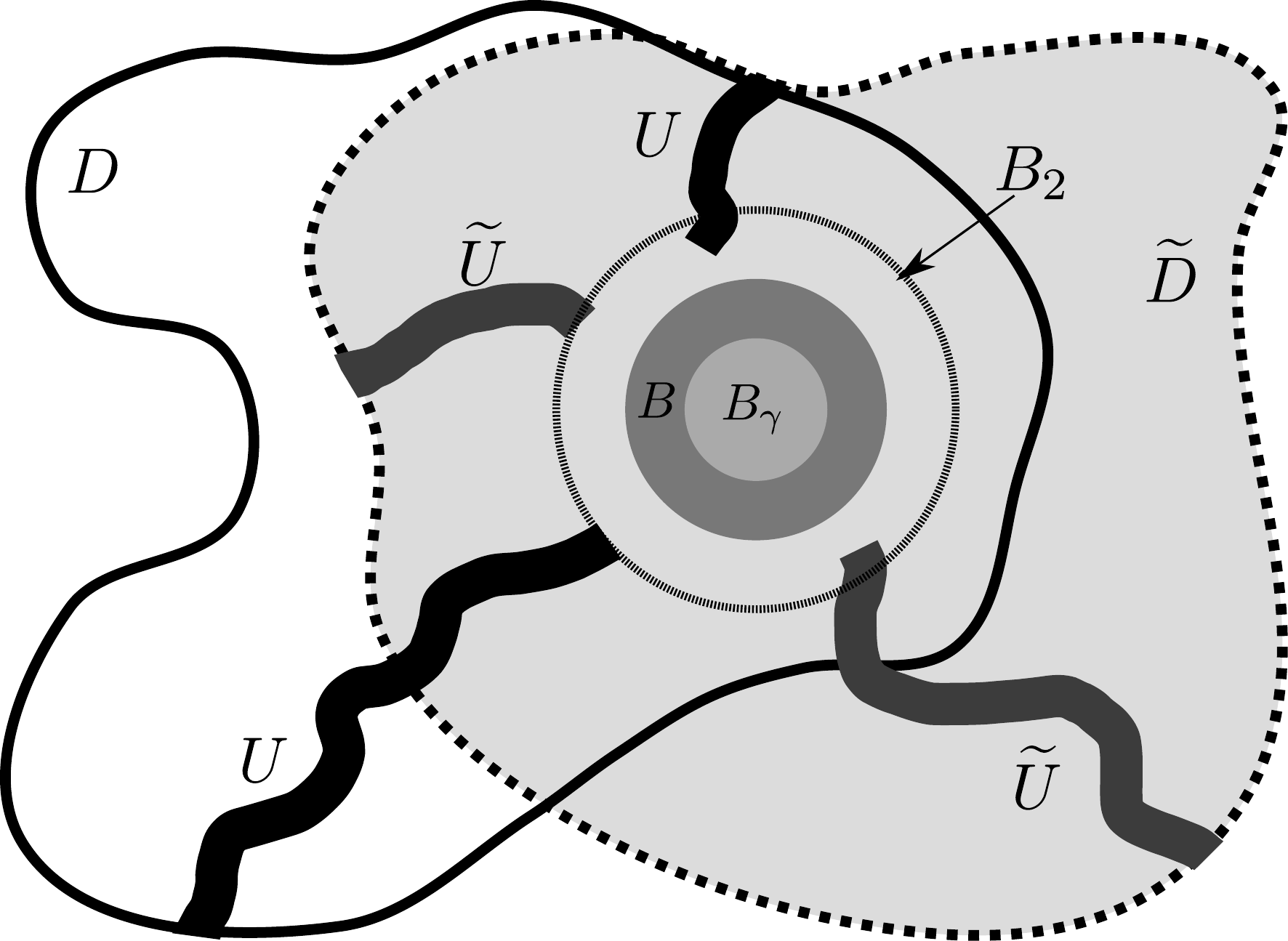}
          \caption{The setup for Proposition \ref{ni::prop::ni}.  We emphasize that $U$ consists of both the black arcs emanating from $\partial D$ along with $\partial D$ itself and likewise for $\wt{U}$.  It is important that $U$, $\wt{U}$ have non-empty intersection with $B_2$ since their presence moderates the fluctuations of the fields in $B_\gamma$.}
\end{figure}

\begin{proposition}[Independence of Interfaces]
\label{ni::prop::ni}
Suppose $D, \wt{D} \subseteq \Z^2$ are bounded, $r > 0$, and $x_0 \in D \cap \wt{D}$.  For each $\alpha > 0$, let $B_\alpha \equiv B(x_0, \alpha r)$, $B \equiv B_1$, and assume $B_3 \subseteq D \cap \wt{D}$.  Suppose $\phi \colon \partial D \to \R, \wt{\phi} \colon \partial \wt{D} \to \R$ satisfy $(\partial)$ and that $U \subseteq D, \wt{U} \subseteq \wt{D}$ correspond to systems of conditioning $(a,b)$, $(\wt{a},\wt{b})$, respectively, both of which satisfy (C), are connected, and intersect $B_2, \partial B_3$ but not $B$.  Let $\CK = \cap_{x \in D} \{ a(x) \leq h(x) \leq b(x)\}$ and $\wt{\CK} = \cap_{x \in \wt{D}}\{\wt{a}(x) \leq \wt{h}(x) \leq \wt{b}(x)\}$.  Fix $0 < \gamma < 1$, suppose $U_\gamma \subseteq B_\gamma$ corresponds to a system of conditioning $(a_\gamma,b_\gamma)$ satisfying (C), is connected, and has non-empty intersection with $B_{\gamma/2}$ and $\partial B_\gamma$, and let $\CK_\gamma = \cap_{x \in B_\gamma}\{ a_\gamma(x) \leq h(x) \leq b_\gamma(x)\}$ and $\wt{\CK}_\gamma = \cap_{x \in B_\gamma}\{ a_\gamma(x) \leq \wt{h}(x) \leq b_\gamma(x)\}$.  There exists $c = c(\gamma, \ol{\Lambda})$ such that
\[ \frac{1}{c} \p_D^\phi[ \CK_\gamma|\CK] \leq \p_{\wt{D}}^{\wt{\phi}}[\wt{\CK}_\gamma | \wt{\CK}] \leq c \p_D^{\phi}[ \CK_\gamma | \CK].\]
\end{proposition}

We will now give an overview of the main steps.  We begin by fixing $0 < \alpha < \alpha'$ small and then couple $h|\CK,\wt{h}|\wt{\CK}$ in $H = B \setminus B_\gamma$ using Theorem \ref{harm::thm::coupling} of \cite{M10} so that $\ol{h} = h|\CK-\wt{h}|\wt{\CK}$ is with high probability harmonic in $H(r^{1-\epsilon})$, $\epsilon > 0$ small.  Recall that $H(r) = \{ x \in H : \dist(x, \partial H) \geq r\}$.  We show in Lemma \ref{ni::lem::bounded_coupling} for $H^\alpha = H(\alpha r)$ that $\E[ \max_{x \in H^\alpha} |\ol{h}(x)|^p] = O_{\alpha,\ol{\Lambda},p}(1)$.  This allows us to conclude for $H^{\alpha'} = H(\alpha' r)$ that $\max_{b \in (H^{\alpha'})^*} |\nabla \ol{h}(b)| \leq C r^{-1}$ with high probability when $\ol{h}$ is harmonic provided $C = C(\alpha,\alpha',\ol{\Lambda}) > 0$ is taken sufficiently large.  Fix $\beta > 0$ so that $\partial B_\beta \subseteq H^{\alpha'}$ and let $(\xi,\wt{\xi}) = (h,\wt{h})|_{\partial B_\beta \times \partial B_\beta}$.  We next study the effect of changing the boundary conditions from $\xi$ to $\wt{\xi}$ on $\partial B_\beta$ has on the probability that $\CK_\gamma$ occurs.  To this end, we let $\varphi \colon B_\beta \to \R$ solve the boundary value problem
\[ \varphi|_{\partial B_\beta} = \ol{\xi},\ \ \varphi|_{B_\gamma} \equiv 0,\ \  (\Delta \varphi)|_{B_\beta \setminus B_\gamma} \equiv 0,\]
where $\ol{\xi} = \xi - \wt{\xi}$, then control the Radon-Nikodym derivative of $\p_{B_\beta}^\xi$ with respect to $\Q_{B_\beta}^{\wt{\xi},\varphi}$ integrated over $\CK_\gamma$, where we recall that $\Q_{B_\beta}^{\wt{\xi},\varphi}$ is the law of $h^{\wt{\xi}} - \varphi$ and $h^{\wt{\xi}} \sim \p_{B_{\beta}}^{\wt{\xi}}$.  Repeated applications of Jensen's inequality (Lemma \ref{ni::lem::prob_ratio_bound}) shows that this quantity is bounded from below by:
\[  \E^{\wt{\xi}} \left[ \exp\left( \sum_{b \in B_\beta^*} \E_{\CK_\gamma}^{\xi,\wt{\xi}}\bigg[ c(b) \nabla \ol{h}^{\xi,\wt{\xi}}(b) \nabla \varphi(b) + O(\CE(b)) \bigg] \right) \one_{\CK_\gamma} \right],\]
where $\E_{\CK_\gamma}^{\xi,\wt{\xi}}$ is the stationary coupling of $\p_{B_\beta}^\xi$ and $\p_{B_\beta}^{\wt{\xi}}[ \cdot | \CK_\gamma]$.  We will then show (Lemma \ref{ni::lem::rn_bound}) that the expectation in the exponential is bounded on $\CA_C = \{ \max_{x \in \partial B_\beta} |\ol{\xi}(x)| \leq C,\ \max_{b \in \partial B_\beta^*} |\nabla \ol{\xi}(b)| \leq C/r\}$, though the estimate deteriorates as we increase $C$.  Integrating the result over $(\xi,\wt{\xi})$ leaves us with an inequality of the form
\[ \p_D^{\phi}[ \wt{\CK}_\gamma | \CK] \geq c_1\E[ \p_{B_\beta}^{\wt{\xi}}[ \wt{\CK}_\gamma] \one_{\CA_C} | \CK, \wt{\CK}].\]
We end the proof (Lemma \ref{ni::lem::density_ratio_bound}) by showing that there exists another event $\CB$, whose probability is uniformly bounded from $0$, such that we have
\[ \p_{B_\beta}^{\wt{\xi}}[\wt{\CK}_\gamma] \one_{\CB} \geq c_2 \p_{\wt{D}}^{\wt{\phi}}[ \wt{\CK}_\gamma | \wt{\CK}] \one_{\CB}.\]
This completes the proof since we can make $\p[\CA_C|\CK, \wt{\CK}]$ as close to $1$ as we like by  choosing $C$ large enough, hence we can ensure that $\p[ \CA_C \cap \CB | \CK, \wt{\CK}]$ is uniformly bounded from zero.

\begin{lemma}[Bounded Coupling]
\label{ni::lem::bounded_coupling}
Assume the hypotheses Proposition \ref{ni::prop::ni} except we replace the restrictions on the geometry of $U, \wt{U}$ with the following.  Suppose that $U \setminus B, \wt{U} \setminus B$ are connected and intersect $B_2, \partial B_3$ and $U, \wt{U}$ do not intersect $H = B \setminus B_\gamma$, and $U \cap B_\gamma, \wt{U} \cap B_\gamma$ are either empty or connected and have non-empty intersection with $B_{\gamma/2}$ and $\partial B_{\gamma}$.  Fix $\epsilon > 0$ so that Theorem \ref{harm::thm::coupling} of \cite{M10} holds.  Consider the coupling of $(h|\CK, \wt{h} | \wt{\CK})$ given by:
\begin{enumerate}
  \item Sampling $(\zeta,\wt{\zeta}) \equiv (h|\CK,\wt{h}|\wt{\CK})|_{\partial H \times \partial H}$ according to any given coupling,
  \item Conditional on $\{\| \zeta\|_\infty + \| \wt{\zeta} \|_\infty \leq (\log r)^2\}$, resample $(h|\CK,\wt{h}|\wt{\CK})$ in $H_\epsilon = H(r^{1-\epsilon})$ according to the coupling of Theorem \ref{harm::thm::coupling} of \cite{M10} \label{ni::item::boundary_good}.
\end{enumerate}
Then $\E[\max_{x \in H^\alpha} |\ol{h}(x)|^p] = O_{\alpha,\ol{\Lambda},p}(1)$ for every $p \geq 1$ and $\alpha > 0$ where $H^\alpha = H(\alpha r)$.
\end{lemma}
\begin{proof}
Our hypotheses on the geometry of $U \setminus B$ imply that Lemma \ref{te::lem::expectation} applies for $h|\CK$ on all of $B$.  This similarly holds for $\wt{h}|\wt{\CK}$ on $B$, so we consequently have $\p[\CE] = O_{\ol{\Lambda}}(r^{-100})$ for $\CE = \{\| \zeta\|_\infty + \| \wt{\zeta}\|_\infty > (\log r)^2\}$.  From Theorem \ref{harm::thm::coupling} of \cite{M10}, we know that on the event $\CE$ the harmonic coupling of $h|\CK ,\wt{h}| \wt{\CK}$ in $H_\epsilon$ is such that with $\wh{g}$ the harmonic extension of $\ol{h} = h|\CK-\wt{h}|\wt{\CK}$ from $\partial H_\epsilon$ to $H_\epsilon$ and $\CH = \{\ol{h} = \wh{g} \text{ in } H_\epsilon\}$ we have $\p[\CH^c | \CE] = O_{\ol{\Lambda}}(r^{-\delta})$, some $\delta > 0.$  It suffices to prove
\[ \E[\max_{x \in H^\alpha} |\ol{h}(x)|^p | \CH,\CE] = O_{\alpha,\ol{\Lambda},p}(1).\]
Indeed, by Lemma \ref{te::lem::expectation} we know that $\E[\max_{x \in H^\alpha} |h(x)|^{2p} | \CK] = O( (\log r)^{2p})$ and likewise for $\wt{h}$.  Hence, by the Cauchy-Schwarz inequality, 
\[ \E[\max_{x \in H^\alpha} |\ol{h}(x)|^p (\one_{\CH^c} + \one_{\CE^c})]\]
is negligible in comparison to the bound we seek to establish.

Let $g,\wt{g}$ be the harmonic extensions of $h,\wt{h}$ from $\partial H_\epsilon$ to $H_\epsilon$.  Then it in turn suffices to show that $\E[ \max_{x \in H^\alpha} |g(x)|^p | \CK] = O_{\alpha, \ol{\Lambda},p}(1)$ and likewise with $\wt{g}$ in place of $g$.  Let $W = V \cup V_-$ and $D_{W} = D \setminus W$.  Let $\partial_1,\partial_2$ be the parts of $\partial D_W$ which do and do not intersect $\partial D$, respectively.  Let $h^{D_W}$ have the law of the GL model on $D_W$ with $h^{D_W}|_{\partial_1} \equiv h|_{\partial_1}$, $h|_{\partial_2} = \ol{\Lambda}$, and the same conditioning as $h|\CK$ otherwise.  Lemma \ref{gl::lem::stoch_dom} implies that we can find a coupling of $h^{D_W}, h|\CK$ such that $h^{D_W} \geq h|\CK$ almost surely.  Hence letting $g^W$ be the harmonic extension of $h^{D_W}$ from $\partial H_\epsilon$ to $H_\epsilon$, we have that $g^W \geq g|\CK$ almost surely.  Of course, we can do exactly the same thing except removing $W' = V \cup V_+$ rather than $W$ and setting the corresponding boundary condition to $-\ol{\Lambda}$.  This leaves us with the lower bound $h|\CK \geq h^{D_{W'}}$ and, with $g^{W'}$ the corresponding harmonic function, we have $g|\CK \geq g^{W'}$.  Thus since $|g|^p|\CK \leq (2^p)(|g^W|^p + |g^{W'}|^p)$, it suffices to show that 
\begin{equation}
\label{ni::eqn::bounded_reduction}
\E[ \max_{x \in H^\alpha} |g^W(x)|^p] = O_{\alpha, \ol{\Lambda},p}(1),
\end{equation} 
and likewise for $g^{W'}$.

Applying the maximum principle to the harmonic function $\E[g^W(x)]$ along with Lemma \ref{hic::lem::exp_upper_bound} implies $\max_{x \in H_\epsilon} |\E[g^W(x)]| = O_{\ol{\Lambda}}(1)$.  Hence to prove \eqref{ni::eqn::bounded_reduction}, we need to prove
\begin{equation}
\label{ni::eqn::bounded_reduction_centered}
\E[ \max_{x \in H^\alpha} |g^W(x) - \E[g^W(x)]|^p] = O_{\alpha,\ol{\Lambda},p}(1).
\end{equation}
Fix $p \geq 1$ and let $g_p^W(x)$ be the harmonic extension of $|g^W(x) - \E[g^W(x)]|^p$ from $\partial H_\epsilon$ to $H_\epsilon$.  For $x \in H_\epsilon$ and $y \in \partial H_\epsilon$, let $p(x,y)$ be the probability that a simple random walk initialized at $x$ first exits $H_\epsilon$ at $y$.  Since
\[ |g^W(x) - \E[g^W(x)]|^p = \left|\sum_{y \in \partial H_\epsilon} p(x,y)(g^W(y) - \E[g^W(y)]) \right|^p,\]
Jensen's inequality implies $|g^W(x) - \E[g^W(x)]|^p \leq g_p^W(x)$.  Fix $y_0 \in H^\alpha$.  By the Harnack inequality, there exists $C_1 = C_1(\alpha)$ such that 
\[ \max_{x \in H^\alpha} |g_p^W(x)| \leq C_1 g_p^W(y_0)\]
since $g_p^W \geq 0$.  Hence we need to bound $\E[ g_p^W(y_0)]$ which, by the maximum principle, is bounded by $\max_{x \in \partial H^\alpha} \E[ |g^W(x) - \E[g^W(x)]|^p]$.  

We can bound this moment using the Brascamp-Lieb inequality (Lemma \ref{bl::lem::bl_inequalities}).  To this end, let $G_{D_W}(x,y)$ be the Green's function for simple random walk on $D_W$.  For $x \in H_\epsilon$, note that
\begin{align*}
    \sum_{y \in \partial H_\epsilon} G_{D_W}(x,y) = O(r).
\end{align*}
The reason for this is that the expected amount of time a random walk started at $x$ spends in $\partial H_\epsilon$ before exiting $B_3$ is $O(r)$ and the expected number of times a random walk reenters $B$ after exiting $B_3$ before hitting $U$ hence $W$ is stochastically dominated by a geometric random variable with parameter $\rho_0 > 0$ by Lemma \ref{symm_rw::lem::beurling} of \cite{M10}.  We also have that $p(x,z) = O_{\alpha}(r^{-1})$ uniformly in $x \in H_{\alpha}$ and $z \in \partial H_\epsilon$.  Hence
\begin{align}
\label{ni::eqn::var_bound}
    \sum_{z_1,z_2 \in \partial H_\epsilon} p(x,z_1) p(x,z_2) G_{D_W}(z_1,z_2) = O_{\alpha}(1).
\end{align}
Combining with the Brascamp-Lieb inequality (Lemma \ref{bl::lem::bl_inequalities}) implies the result since the expression on the left side is exactly the variance of the DGFF on $D_W$.
\end{proof}

\begin{lemma}
\label{ni::lem::prob_ratio_bound}
Suppose $F \subseteq \Z^2$ is bounded, $E \subseteq F$, $\CA \in \CF_E = \sigma(h(x) : x \in E)$, $\psi, \wt{\psi} \colon \partial F \to \R$, and $\varphi \colon F \to \R$ satisfies $\varphi|_E \equiv 0$, $\varphi|_{\partial F} = \psi - \wt{\psi}$.  Then we have that
\begin{align*}
   \p_{F}^\psi[\CA]
\geq \E^{\wt{\psi}} \left[ \exp\left( \sum_{b \in F^*} \E_{\CA}^{\psi,\wt{\psi}}\bigg[ c(b) \nabla \ol{h}(b) \nabla \varphi(b) + O(\CE(b)) \bigg] \right) \one_{\CA} \right]
\end{align*}
where $\E_\CA^{\psi,\wt{\psi}}$ is the expectation under any coupling of $\p_F^{\psi}$ and $\p_F^{\wt{\psi}}[\cdot | \CA]$, 
\[ c(b) = \CV''(\nabla h^{\wt{\psi}}(b)) \text{ and } \CE(b) = |\nabla \ol{h}(b)|^2|\nabla \varphi(b)| + (\nabla \varphi(b))^2.\]
\end{lemma}
\begin{proof}
Let $\CZ, \wt{\CZ}$ be the normalization constants that appear in the densities of $\p_F^\psi, \p_F^{\wt{\psi}}$ with respect to Lebesgue measure.  Recall that $\Q_F^{\wt{\psi},-\varphi}$ denotes the law of $(h^{\wt{\psi}} + \varphi)$ for $h^{\wt{\psi}} \sim \p_F^{\wt{\psi}}$ and $\E_\Q^{\wt{\psi},-\varphi}$ is the corresponding expectation.  Note that the normalization constant of $\Q_F^{\wt{\psi},-\varphi}$ is also $\wt{\CZ}$.  We compute,
\begin{align}
    \p_{F}^\psi[\CA]
  =& \frac{\wt{\CZ}}{\CZ} \E_{\Q}^{\wt{\psi},-\varphi}\left[ \exp\left(\sum_{b \in F^*} [ \CV(\nabla (h - \varphi) \vee \wt{\psi}(b)) - \CV( \nabla h \vee \psi(b))] \right) \one_{\CA} \right] \notag\\
 =&\frac{\wt{\CZ}}{\CZ} \E^{\wt{\psi}}\left[ \exp\left(\sum_{b \in F^*} [ \CV(\nabla h(b)) - \CV( \nabla (h +\varphi)(b))] \right) \one_{\CA} \right] \label{ni::eqn::rn}
\end{align}
Since $\varphi \equiv 0$ on $E$, the part of the summation over $b \in E^*$ is identically zero.  Let $A = F \setminus E$.  By Jensen's inequality, the expression in \eqref{ni::eqn::rn} is bounded from below by
\begin{align}
 \frac{\wt{\CZ}}{\CZ} \E^{\wt{\psi}}\left[ \exp\left(\sum_{b \in A^*} \E^{\wt{\psi}}[ \CV(\nabla h(b)) - \CV( \nabla (h +\varphi)(b)) \big| \CA ] \right) \one_{\CA} \right].
\end{align}
Applying a first order Taylor expansion to $\CV$ about $\nabla h(b)$ and using that $\CV''$ is uniformly bounded, we can rewrite our formula for $\p_F^\psi[\CA]$ as
\begin{align}
 \frac{\wt{\CZ}}{\CZ} \E^{\wt{\psi}}\left[ \exp\left(\sum_{b \in A^*} \bigg[ \E^{\wt{\psi}}[ -\CV'(\nabla h(b)) | \CA] \nabla \varphi(b) + O( (\nabla \varphi(b))^2) \bigg]  \right) \one_{\CA} \right] \label{ni::eqn::rn_lb}.
\end{align}
Applying exactly the same procedure but with $\p_F^{\wt{\psi}}, \Q_F^{\wt{\psi},-\varphi}$ replaced by $\p_F^{\psi}, \Q_F^{\psi,\varphi}$, respectively, and $\CA$ by the whole sample space, we also have
 \begin{align}
  \frac{\wt{\CZ}}{\CZ} \geq \exp\left( \sum_{b \in A^*} \bigg[ \E^\psi[\CV'(\nabla h(b))] \nabla \varphi(b) + O( (\nabla \varphi(b))^2) \bigg] \right). \label{ni::eqn::constant_lb}
 \end{align}
Combining \eqref{ni::eqn::rn_lb} and \eqref{ni::eqn::constant_lb} with \eqref{ni::eqn::rn} yields that $\p_F^\psi[\CA]$ is bounded from below by
\begin{align*}
 \E^{\wt{\psi}} \left[ \exp\left( \sum_{b \in A^*} \bigg[ \big(\E^{\psi}[\CV'(\nabla h(b))] - \E^{\wt{\psi}}[\CV'(\nabla h(b)) | \CA] \big)\nabla \varphi(b) + O((\nabla \varphi(b))^2)\bigg] \right) \one_{\CA}  \right].
\end{align*}
Fixing a coupling $(h^\psi, h^{\wt{\psi}})$ of $\p_F^\psi$, $\p_F^{\wt{\psi}}[ \cdot | \CA]$ and setting $\ol{h} = h^\psi - h^{\wt{\psi}}$, with $\E_\CA^{\psi,\wt{\psi}}$ denoting the corresponding expectation, another application of Taylor's formula implies
\[  (\E^{\psi}[\CV'(\nabla h(b))] - \E^{\wt{\psi}}[\CV'(\nabla h(b)) | \CA]) \nabla \varphi(b) = \E_{\CA}^{\psi,\wt{\psi}}[ c(b) \nabla \ol{h}(b) \nabla \varphi(b) + O(\CE(b))],\]
which, when combined with the previous expression, proves the lemma.
\end{proof}

We say that $F \subseteq \Z^2$ with $\diam(F) < \infty$ is $C$-stochastically regular if
\[ \p_x[ |X_\tau - x| \geq s] \leq \frac{Cs}{\diam(F)}\]
for every $x \in F$ with $\dist(x, \partial F) = 1$ where $X$ is a simple random walk and $\tau$ is its time of first exit from $F$.  We also define the norm
\[ \| \psi \|_{F}^{\nabla} = \max_{x \in \partial F} |\psi(x)| + \diam(F) \left(\max_{\stackrel{x,y \in \partial F}{x \neq y}} \frac{|\psi(x) - \psi(y)|}{|x-y|} \right)\]
on the space of functions $\{ \psi \colon \partial F \to \R\}$.

\begin{lemma}
\label{ni::lem::rn_bound}
Suppose that $F \subseteq \Z^2$ with $r = \diam(F) < \infty$ is $C$-stochastically regular.  Assume that $E \subseteq \Z^2$, $E_\alpha = \cup_{x \in E} B(x,\alpha r)$, $E_{2C^{-1}} \subseteq F$, $A=F \setminus E'$ with $E' = E_{C^{-1}}$ is also $C$-stochastically regular and, for each $k,\delta > 0$, the number of balls of radius $r^{\delta}$ required to cover $A(r^{k\delta}, r^{(k+1)\delta})$ is $O(r^{1-(k-1)\delta})$.  Suppose that $\wt{U}_E \subseteq E$ corresponds to a system of conditioning $(\wt{a}_E,\wt{b}_E)$ satisfying $(C)$ and let $\wt{\CK}_E = \cap_{x \in E} \{ \wt{a}_E(x) \leq h^{\wt{\psi}}(x) \leq \wt{b}_E(x) \}$.  Let $\psi, \wt{\psi} \colon \partial F \to \R$ satisfy $\| \psi - \wt{\psi}\|_{F}^\nabla \leq C$.  Suppose $\varphi \colon F \to \R$ is harmonic off of $E'$, $\varphi|_{E'} \equiv 0$, and $\| \varphi \|_{F}^\nabla \leq C$.  Let $(h^\psi, h^{\wt{\psi}}|\wt{\CK}_E)$ denote the stationary coupling of $\p_F^{\psi}, \p_F^{\wt{\psi}}[\cdot|\wt{\CK}_E]$ and $\E_{\wt{\CK}_E}^{\psi,\wt{\psi}}$ the corresponding expectation.  Using the notation $c(b)$ and $\CE(b)$ from the previous lemma, we have that
\[ \E_{\wt{\CK}_E}^{\psi,\wt{\psi}} \left[ \sum_{b \in F^*} [ c(b) \nabla \ol{h}(b) \nabla \varphi(b) + O(\CE(b)) ]\right]  = O_C(1).\]
\end{lemma}
\begin{proof}
With $\ol{\psi} = \psi - \wt{\psi}$, let 
\[ \alpha_j = \max \{| \ol{\psi}(z_1) - \ol{\psi}(z_2)| : z_1, z_2 \in \partial F, |z_1-z_2| \leq j\}\]
and note that $|\alpha_{j+1} - \alpha_j| \leq C r^{-1}$ since $\|\ol{\psi}\|_F^\nabla \leq C$.  Fix $b = (x,y) \in \partial F^*$.  Letting $X$ be the random walk of Subsection \ref{subsec::rw_difference}, $\tau$ its time of first exit from $A$, $p_j = \p_y[ |X_\tau-x| \geq j]$, $J= r/C$, and $M = \max_{x \in F}\big( |h^\psi(x)| + |(h^{\wt{\psi}} |\wt{\CK}_E)(x)| \big)$ an application of summation by parts implies
\begin{equation}
\label{ni::eqn::sbp}
 |\nabla \ol{h}(b)| \leq \sum_{j=0}^{J-1} (p_{j+1} - p_j) \alpha_j + p_J M \leq \sum_{j=0}^{J-1} |\alpha_{j}-\alpha_{j-1}| p_j + 2p_J M.
\end{equation}
Lemma \ref{symm_rw::lem::beurling} of \cite{M10} implies $p_j = O(j^{-\rho_{\rm B}})$ for $\rho_{\rm B} = \rho_{\rm B}(\CV) \in (0,1)$, hence the right side of \eqref{ni::eqn::sbp} is bounded by
\[  \sum_{0=1}^{J-1} O_C( j^{-\rho_{\rm B}} r^{-1}) + O_C(r^{-\rho_{\rm B}}M ) = O_C(r^{-\rho_{\rm B}}(1+M)).\]
Lemma \ref{te::lem::expectation} implies that $\E_{\wt{\CK}_E}^{\psi,\wt{\psi}}[M^p] = O_{\ol{\Lambda}}( (\log r)^{p})$.  The Nash continuity estimate (Lemma \ref{symm_rw::lem::nash_continuity_bounded} of \cite{M10}) implies that $\nabla h(b) = O_{C}( M r^{-\rho_{\rm NC}})$ uniformly in $b \in \partial (E')^*$.  Consequently, with $\rho = \rho_{\xi_{\rm NC}} \wedge \rho_{\rm B}$, the energy inequality \eqref{gl::eqn::ee_limit} along with Cauchy-Schwarz implies
\begin{equation}
\label{ni::eqn::ee_bound}
 \sum_{b \in A^*} \E_{\wt{\CK}_E}^{\psi,\wt{\psi}}[ |\nabla \ol{h}(b)|^2] 
   \leq c_1 \sum_{b \in \partial A^*} \E_{\wt{\CK}_E}^{\psi,\wt{\psi}}[ | \nabla \ol{h}(b)| |\ol{h}(x_b)|] = O_{C}(r^{1+\epsilon-\rho}).
\end{equation}
The hypotheses of the lemma imply $|\nabla \varphi(b)| = O_{C}(r^{-1})$ uniformly in $b \in F^*$, hence
\[ \sum_{b \in F^*} \E_{\wt{\CK}_E}^{\psi,\wt{\psi}}[ \CE(b)] = O_{C}(1).\]
Thus to prove the lemma we need to control
\[ \E_{\wt{\CK}_E}^{\psi,\wt{\psi}}\left[ \sum_{b \in A^*} c(b) \nabla \ol{h}(b) \nabla \varphi(b) \right].\]
We will first argue that the contribution coming from the terms near $\partial H$ is negligible.  Using \eqref{ni::eqn::ee_bound} along with Cauchy-Schwarz and that $|A^* \setminus A^*(r^{\rho/2})| = O(r^{1+\rho/2})$, we have
\begin{align}
 &\E_{\wt{\CK}_E}^{\psi,\wt{\psi}}\left[ \sum_{b \in A^* \setminus A^*(r^{\rho/2})} |\nabla \ol{h}(b) \nabla \varphi(b)| \right] \notag\\
=& \sqrt{O_C(r^{1+\epsilon-\rho}) O_C(r^{-1+\rho/2})} = O_{C}(r^{\epsilon/2-\rho/4}) \label{ni::eqn::boundary_error}.
\end{align}
We now handle the interior term.  Let $(h_t^\psi,(h^{\wt{\psi}}|\wt{\CK}_E)_t)$ denote the dynamics of the stationary coupling.  Fixing $\delta > 0$, by hypothesis each of the annuli $A(r^{k\delta},r^{(k+1)\delta})$ can be covered by $O(r^{1-(k-1)\delta})$ balls of radius $r^{k\delta}$.  On such a ball $Q$, Theorem \ref{cd::thm::cd} of \cite{M10} implies that
\begin{align}
&\E_{\wt{\CK}_E}^{\psi,\wt{\psi}} \left[ \sum_{b \in Q^*} \CV''(\nabla h^{\wt{\psi}}(b))\nabla \ol{h}(b) \nabla \varphi(b) \right] \notag\\
=& \E_{\wt{\CK}_E}^{\psi,\wt{\psi}} \left[\sum_{b \in Q^*} c_{\CV} \nabla \ol{h}(b) \nabla \varphi(b) \right] + O_{C}(r^{\epsilon+k\delta(1-\rho_{\rm CD})-1}). \label{ni::eqn::interior_error}
\end{align}
Thus summing over a covering of $A(r^{k\delta},r^{(k+1)\delta})$ by such balls yields an error of $O_{C}(r^{\epsilon+\delta-k \delta \rho_{\rm CD}})$.  The exponent is negative for the relevant values of $k$ since boundary term includes those annuli of with $k \delta < \rho/2$.  That is, we may assume $k \delta \geq \rho/2$ and, since we are free to choose $\epsilon, \delta > 0$ as small as we like, we also assume that $\rho > \rho_{\rm CD}^{-1} 10^{10}(\epsilon+\delta)$.  Combining \eqref{ni::eqn::boundary_error} with \eqref{ni::eqn::interior_error} implies that there exists non-random $c_\CV,\ol{\rho} > 0$ depending only on $\CV$ such that
\begin{align*}
  \E_{\wt{\CK}_E}^{\psi,\wt{\psi}}\left[ \sum_{b \in A^*} \CV''(\nabla h^{\wt{\psi}}(b)) \nabla \ol{h}(b) \nabla \varphi(b) \right]
=  \E_{\wt{\CK}_E}^{\psi,\wt{\psi}}\left[ \sum_{b \in A^*} c_\CV \nabla \ol{h}(b) \nabla \varphi(b) \right] + O_C(r^{-\ol{\rho}}).
\end{align*}
Summing by parts and using the harmonicity of $\varphi$, we see that the expectation on the right hand side is bounded from above by
\[ c_{\CV} \E_{\wt{\CK}_E}^{\psi,\wt{\psi}} \left[ \sum_{b \in \partial A^*} |\nabla \varphi(b)|| \ol{h}_0(x_b)| \right] = O_C(1).\]
\end{proof}

\begin{lemma}
\label{ni::lem::density_ratio_bound}
Suppose that we have the same setup as Lemma \ref{ni::lem::bounded_coupling} and fix $\beta \in (\gamma,1)$  Let $f$ and $g$ be the densities of $\xi = (h|\CK)|_{\partial B_\beta}$ and $\wt{\xi} =  (\wt{h} | \wt{\CK})|_{\partial B_\beta}$ with respect to Lebesgue measure on $\R^{|\partial B_\beta|}$, respectively.  There exists $\delta_i = \delta_i(\beta, \gamma, \ol{\Lambda}) > 0$ such that
\[ \p_{D}^{\phi}\left[ \frac{g(\xi)}{f(\xi)} \geq \delta_1 \bigg| \CK \right] \geq \delta_2.\]
\end{lemma}
\begin{proof}
We have the trivial bound
\[ \int_{\R^{|\partial B_\beta|}} \left| \frac{f(z)}{g(z)} - 1 \right| g(z) dz \leq 2.\]
Hence applying Markov's inequality for $\CB^c$ where
\[ \CB = \left\{ \xi : \frac{f(\xi)}{g(\xi)} \leq 100\right\} = 
              \left\{ \xi : \frac{g(\xi)}{f(\xi)} \geq \frac{1}{100}\right\},\]  
we have that $\p_{\wt{D}}^{\wt{\phi}}[\CB| \wt{\CK}] \geq \frac{49}{50}$.  Our goal now is to convert this into a lower bound on $\p_D^{\phi}[\CB | \CK]$.

Let $\beta = (1+\gamma)/2$ and assume that $0 < \alpha < \alpha'$ are chosen sufficiently small so that $\partial B_\beta \subseteq H^{\alpha'}$.  Assume that $h,\wt{h}$ are coupled together in $H_\epsilon$ as in the setup of Lemma \ref{ni::lem::bounded_coupling}.  Letting $(\zeta_{\alpha'},\wt{\zeta}_{\alpha'}) = (h, \wt{h})|_{\partial H^{\alpha'} \times \partial H^{\alpha'}}$, Lemma \ref{ni::lem::bounded_coupling} implies that with $\CA_C = \{ \| \ol{\zeta}_{\alpha'}\|_{H^{\alpha'}}^{\nabla} \leq C\}$, $\ol{\zeta}_{\alpha'} = \zeta_{\alpha'} - \wt{\zeta}_{\alpha'}$,
we can make $\p[\CA_C]$ as close to $1$ as we like by choosing $C,r$ sufficiently large.  Let $\varphi \colon H^{\alpha'} \to \R$ be the solution of the boundary value problem
\[ \varphi|_{\partial H^{\alpha'}} \equiv \ol{\zeta}_{\alpha'},\ \ \varphi|_{\partial B_\beta} \equiv 0,\ \ \Delta \varphi|_{H^{\alpha'} \setminus \partial B_\beta} \equiv 0.\]

By Lemma \ref{harm::lem::entropy_form} of \cite{M10} and with $\h(\cdot|\cdot)$ denoting the relative the entropy, we know that
\begin{align}
   &\h(\p_{H^{\alpha'}}^{\zeta_{\alpha'}} | \Q_{H^{\alpha'}}^{\wt{\zeta}_{\alpha'},\varphi}) + \h(\Q_{H^{\alpha'}}^{\wt{\zeta}_{\alpha'},\varphi}|\p_{H^{\alpha'}}^{\zeta_{\alpha'}}) \notag\\
=& \sum_{b \in (H^{\alpha'})^*} \E^{\zeta_{\alpha'},\wt{\zeta}_{\alpha'}}\big[ c(b) \nabla \ol{h}(b) \nabla \varphi(b) + O(\CE(b)) \big] \label{ni::eqn::ni_ent_bound}
\end{align}
with $c(b), \CE(b)$ as in Lemma \ref{ni::lem::prob_ratio_bound}.  On $\CA_C$, Lemma \ref{ni::lem::rn_bound} implies \eqref{ni::eqn::ni_ent_bound} is of order $O_C(1)$.  By the non-negativity of the relative entropy, this implies $\h(\Q_{H^{\alpha'}}^{\wt{\zeta}_{\alpha'},\varphi}|\p_{{H^{\alpha'}}}^{\zeta}) \one_{\CA_C} =O_C(1) \one_{\CA_C}$,  hence invoking the elementary entropy inequality (see the proof of \cite[Lemma 5.4.21]{DS89})
\[  \p_{H^{\alpha'}}^{\zeta_{\alpha'}}[Q] 
 \geq \exp\left( - \frac{\h(\Q_{H^{\alpha'}}^{\wt{\zeta}_{\alpha'},\varphi}|\p_{H^{\alpha'}}^{\zeta_{\alpha'}}) + e^{-1}}{\Q_{H^{\alpha'}}^{\wt{\zeta}_{\alpha'},\varphi}[Q]} \right) \Q_{H^{\alpha'}}^{\wt{\zeta}_{\alpha'},\varphi}[Q]\]
we have the lower bound
\begin{align}
 \p_{H^{\alpha'}}^{\zeta_{\alpha'}}[\CB] 
\geq \exp\left( - \frac{O_C(1)}{\p_{H^{\alpha'}}^{\wt{\zeta}_{\alpha'}}[\CB]} \right) \p_{H^{\alpha'}}^{\wt{\zeta}_{\alpha'}}[\CB] \one_{\CA_C} \label{ni::eqn::ni_ent_ineq_bound}.
\end{align}
Note that we used $\varphi|_{\partial B_\beta} \equiv 0$ to conclude $\p_{H^{\alpha'}}^{\wt{\zeta}_{\alpha'}}[\CB] = \Q_{H^{\alpha'}}^{\wt{\zeta}_{\alpha'},\varphi}[\CB]$.  

As $\p_{\wt{D}}^{\wt{\phi}}[ \CB | \wt{\CK}] = \E^{\wt{\phi}}[\p_{H^{\alpha'}}^{\wt{\zeta}_{\alpha'}}[\CB] | \wt{\CK}] \geq 49/50$, we have
\[ \p_{\wt{D}}^{\wt{\phi}}[ \{\p_{H^{\alpha'}}^{\wt{\zeta}_{\alpha'}}[\CB] \geq 1/2\} | \wt{\CK}] \geq c_2(\beta,\gamma,\ol{\Lambda}) > 0.\]
Consequently, taking expectations of both sides of \eqref{ni::eqn::ni_ent_ineq_bound} over $(\zeta_{\alpha'},\wt{\zeta}_{\alpha'})$ conditional on $\CK, \wt{\CK}$, we see that $\p_D^\phi[\CB | \CK] \geq c_1(\alpha,\gamma,\ol{\Lambda}) > 0$, as desired.
\end{proof}

\subsection{Proof of Proposition \ref{ni::prop::ni}}
Assume that $h|\CK,\wt{h} | \wt{\CK}$ are coupled together as in the setup of Lemma \ref{ni::lem::bounded_coupling}.  Let $\beta = (1+\gamma)/2$, $A = B_\beta \setminus B_\gamma$, and $(\xi, \wt{\xi}) = (h|\CK,\wt{h}|\wt{\CK})|_{\partial B_\beta \times \partial B_\beta}$. Let $\E_{\wt{\CK}_\gamma}^{\xi,\wt{\xi}}$ denote the expectation under the stationary coupling of $\p_{B_\beta}^{\xi}$ and $\p_{B_\beta}^{\wt{\xi}}[\cdot | \wt{\CK}_\gamma]$, set $\ol{\xi} = \xi - \wt{\xi}$, and let $\CA_C = \{ \| \ol{\xi}\|_{B_\beta}^{\nabla} \leq C\}.$
Let $\varphi \colon B_\beta \to \R$ be the solution of the boundary value problem
\[ \varphi |_{\partial B_\beta} \equiv \ol{\xi},\ \ \varphi|_{B_\gamma} \equiv 0,\ \ (\Delta \varphi)|_{A} \equiv 0.\]
By the definition of $\CA_C$ and the harmonicity of $\varphi$, we have that
\begin{equation}
 \label{ni::eqn::phi_grad_bound}
 \max_{b \in A^*} | \nabla \varphi(b)| \one_{\CA_C} = O_{C}(r^{-1}) \one_{\CA_C}.
\end{equation}

Taking $F = B_\beta$, $E = B_\gamma$, $\CA = \wt{\CK}_\gamma$ in Lemma \ref{ni::lem::prob_ratio_bound} combined with Lemma \ref{ni::lem::rn_bound} implies that
\begin{equation}
\label{ni::eqn::initial_bound}
 \p_{B_\beta}^\xi[\CK_\gamma] \geq  \exp\left( -O_{\gamma,\ol{\Lambda},C}(1) \right) \p_{B_\beta}^{\wt{\xi}}[\wt{\CK}_\gamma] \one_{\CA_C}.
\end{equation}

To finish the proof of Proposition \ref{ni::prop::ni} it suffices to prove the existence of non-random $c > 0$ so that 
\[ \p_{B_\beta}^{\wt{\xi}}[\wt{\CK}_\gamma] \one_{\CA_C} \geq c\p_{\wt{D}}^{\wt{\phi}}[\wt{\CK}_\gamma] \one_{\CA_C}.\]
Let $f$ denote the density of $\wt{\xi} = (\wt{h}|\wt{\CK})|_{\partial B_\beta}$ and $g$ be the density of $(\wt{h}| \wt{\CK}_\gamma \cap \wt{\CK}) |_{\partial B_\beta}$, both with respect to Lebesgue measure on $\R^{|\partial B_\beta|}$.  The Markovian structure of the field implies that the events $\wt{\CK}, \wt{\CK}_\gamma$ are independent conditional on $\wt{\xi}$.  Consequently, by Bayes' rule we have
\[ \frac{g(\wt{\xi})}{f(\wt{\xi})} = \frac{\p_{B_\beta}^{\wt{\xi}}[\wt{\CK}_\gamma]}{\p_{\wt{D}}^{\wt{\phi}}[ \wt{\CK}_\gamma | \wt{\CK}]}, \ \ \ \text{hence}\ \ \  \p_{B_\beta}^{\wt{\xi}}[\wt{\CK}_\gamma] \one_{\CA_C} = \frac{g(\wt{\xi})}{f(\wt{\xi})} \p_{\wt{D}}^{\wt{\phi}}[\wt{\CK}_\gamma | \wt{\CK}] \one_{\CA_C}.\]
Since we can make $\p[\CA_C]$ as close to $1$ as we like by increasing $C, r$, it thus suffices to show that $g(\wt{\xi}) / f(\wt{\xi})$ is uniformly bounded from zero with uniformly positive probability.  This is exactly the statement of Lemma \ref{ni::lem::density_ratio_bound}.  
\qed

\section{Completing the Proof}
\label{sec::bvi}

We will now explain how the estimates of Sections \ref{sec::gl}-\ref{sec::ni} can be put together to prove Theorems \ref{intro::thm::approximate_martingale} and \ref{intro::thm::scaling_limit}.  Both proofs follow from the strategy of  \cite[Subsections 3.5-3.7]{SS09}, so we will only give an overview of how everything fits together in our setting and leave the reader to \cite{SS09} for more details.

\subsection{Scaling Limits}

 \begin{figure}[h]
     \centering
     \subfigure[$Y_1$ and $Y_2$ are barriers and $\wt{\gamma}^n$ is part of the interface in the setup of Theorem \ref{intro::thm::sle_convergence}.  If $\gamma^n$ does not cross $Y_1, Y_2$, then the strands of $\wt{\gamma}^n$ are forced to connected in the red region.]{
         \includegraphics[height=.55\textwidth]{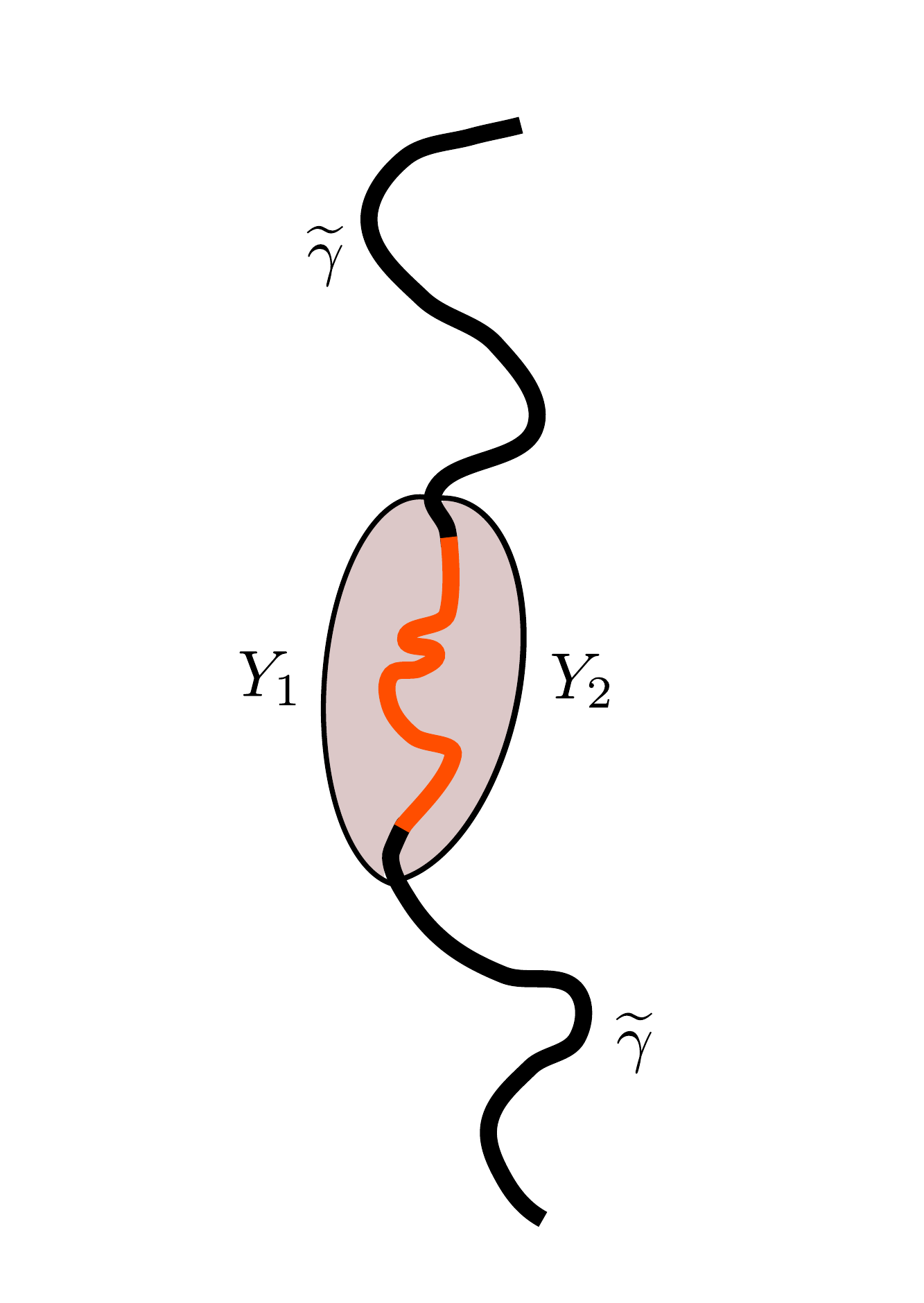}
          \label{bvi::fig::barrier}}
          \hspace{0.05in}
     \subfigure[Using barriers, it is possible to show that internal and external configurations in which the strands are well-separated connect with probability proportional to $(\log R)^{-1}$.]{
         \includegraphics[height=.55\textwidth]{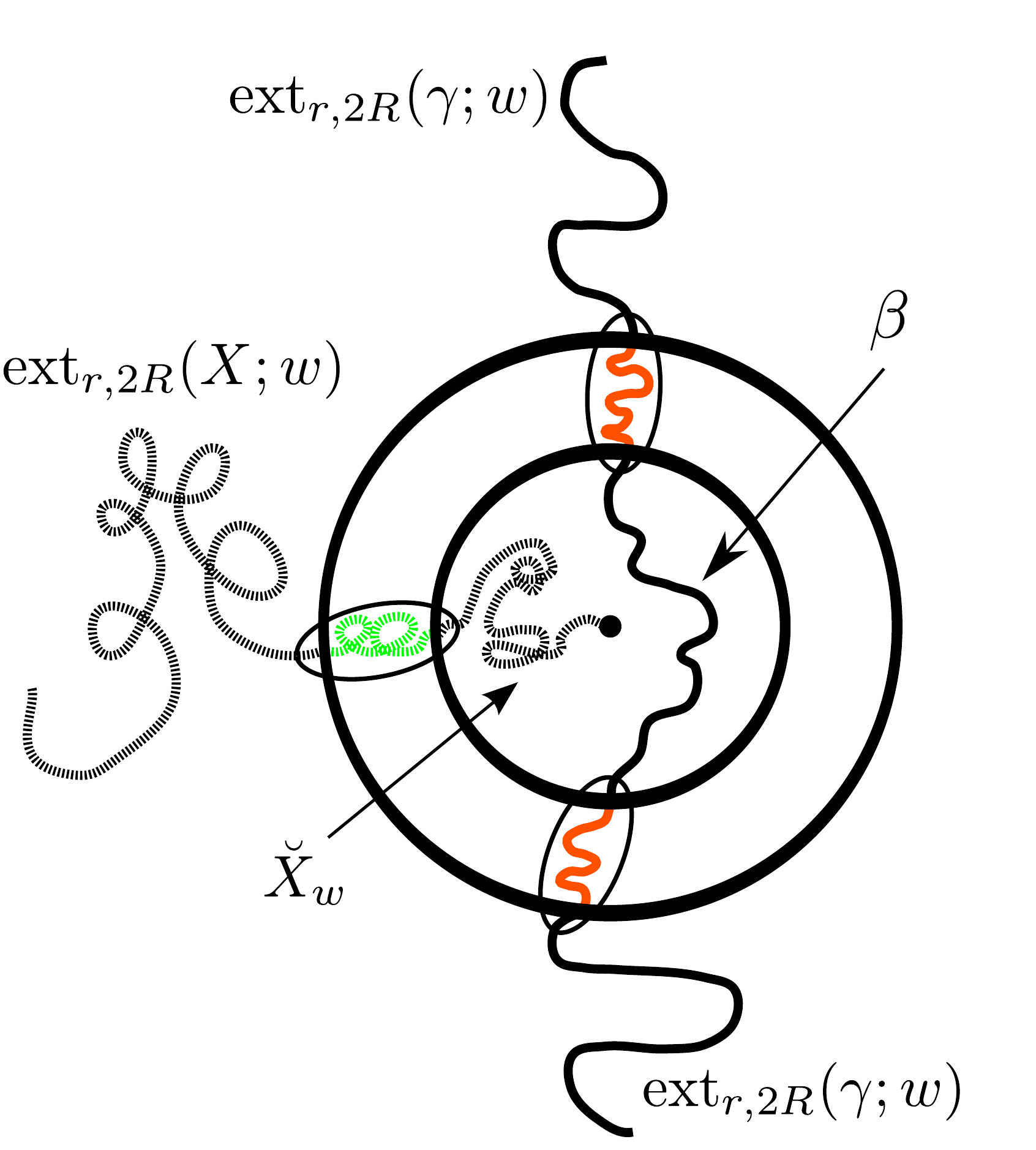} \label{bvi::fig::hookup}}
         \caption{Typical applications of the Barriers Theorem.}
 \end{figure}

The proof of Theorem \ref{intro::thm::scaling_limit} has two main inputs: Proposition \ref{ni::prop::ni} and the notion of a \emph{barrier}, developed in \cite[Subsection 3.4]{SS09}.  Roughly, the latter is a deterministic curve $Y$ through which $\gamma^n$ does not pass with uniformly positive probability (u.p.p.).  Barriers can be used in conjunction with each other to prove that $\gamma^n$ with u.p.p. must pass through certain regions.  A typical usage is illustrated in Figure \ref{bvi::fig::barrier}.  The black line labeled $\wt{\gamma}$ indicates the two strands of $\gamma^n$ from the setup of Theorem \ref{intro::thm::sle_convergence} emanating from $x,y$ and the thin lines $Y_1,Y_2$ indicate barriers.  The Barriers Theorem \cite[Theorem 3.11]{SS09} implies that, conditional on $\wt{\gamma}^n$, with u.p.p. $\gamma^n$ does not pass through $Y_1, Y_2$.  On this event, the two strands of $\wt{\gamma}^n$ are forced to connect since $\gamma^n$ is connected.  The proof of \cite[Theorem 3.11]{SS09} has some dependencies on the specific structure of the DGFF.  The modifications necessary to transfer the result to our setting are deferred to subsection \ref{subsec::nob}.

Barriers can be used in combination with Proposition \ref{ni::prop::ni} to prove Theorem \ref{intro::thm::scaling_limit}.  Assume that we are in the setting of Theorem \ref{intro::thm::scaling_limit}, fix $v_0 \in D_n(\gamma,t,\epsilon)$, $r > 0$, let $X$ be a simple random walk on $\tfrac{1}{n} \Z^2$ initialized at $v_0$ independent of $h^n$, and let $\tau(r)$ be the first time $X$ gets within distance $r n^{-1}$ of $\gamma^n$ with respect to the internal metric of $D_n \setminus \gamma[0,t]$.  Fix $R > r$, which we assume not to vary with and be much smaller than $n$.  We define the \emph{internal configuration} $\Theta_{r,R}$ of $\gamma^n$ and $X$ as seen from $X_{\tau(r)}$ to be the pair $(\internal_{r,R}(\gamma^n), \internal_{r,R}(\breve{X}))$ where $\internal_{r,R}(\gamma^n)$ is the connected component of $\gamma^n \cap B(X_{\tau(r)},Rn^{-1})$ with $\dist(\internal_{r,R}(\gamma^n), X_{\tau(r)}) = rn^{-1}$ re-centered at $X_{\tau(r)}$; ties are broken according to some fixed but unspecified convention.  We remark that by \cite[Lemma 3.17]{SS09} with high probability there is only one such component.  Here, $\internal_{r,R}(\breve{X})$ is the time reversal of $X$ starting at $X_{\tau(r)}$ up until its first exit from $B(X_{\tau(r)},Rn^{-1})$, then translated by $-X_{\tau(r)}$.  The \emph{external configuration} $\Phi_{r,R}$ of $\gamma^n$ and $X$ as seen from $X_{\tau(r)}$ is the pair $(\ext_{r,R}(\gamma^n), \ext_{r,R}(X))$ along with data associated with $D_n$ and the boundary conditions of $h^n$.  Here, $\ext_{r,R}(\gamma^n)$ consists of the two connected components of $\gamma^n \setminus B(X_{\tau(r)},Rn^{-1})$ containing $x_n$ and $y_n$, re-centered at $X_{\tau(r)}$, and $\ext_{r,R}(X)$ is $X$ stopped at its first hitting time of $B(X_{\tau(r)},Rn^{-1})$, re-centered at $X_{\tau(r)}$.

Fix $w \in D_n$ with $\dist(w,\partial D_n)$ much larger than $R n^{-1}$ and let $\CZ_w = \{ X_{\tau(r)} = w\}$.  Let $\zeta = (\beta, \breve{Y}_w)$ be a configuration consisting of an oriented curve $\beta$ in $D_n^*$ coming exactly within distance $rn^{-1}$ to $w$ whose endpoints are contained in $\partial B(w,R n^{-1})$ and $\breve{Y}_w$ a path in $D_n$ connecting $w$ to $\partial B(w,R n^{-1})$.  Let $\CK_\beta$ be the event that $\beta$ is an oriented zero-height interface of $h^n$.  Let $\Phi_{r,R}(w)$ be the external configuration at $w$.  That is, $\Phi_{r,R}(w) = (\ext_{r,R}(\gamma^n;w), \ext_{r,R}(X;w))$, along with the data associated with $D_n$ and the boundary conditions of $h^n$,  where $\ext_{r,R}(\gamma^n;w)$ consists of the two connected components of $\gamma^n$ emanating from $x_n,y_n$ until first hitting $B(w,R n^{-1})$ and $\ext_{r,R}(X;w)$ is the initial segment of $X$ up until it first hits $B(w,R n^{-1})$.  Let $\breve{X}_w$ denote the time reversal of $X$ starting from when it first hits $w$ to its first exit from $B(w,R n^{-1})$.  Using $+w$ to denote translation by $w$, we can write
\begin{align}
  &      \p[ \Theta_{r,R} = \zeta-w | \CZ_w, \Phi_{r, 2R}] \notag\\
=&   \p[ \CK_\beta, \breve{X}_w = \breve{Y}_w | \CZ_w, \Phi_{r, 2R}(w)]
  =  \frac{\p[ \CK_\beta, \breve{X}_w = \breve{Y}_w, \CZ_w | \Phi_{r, 2R}(w)]}{\p[\CZ_w | \Phi_{r, 2R}(w)]} \notag\\
 =& \frac{\p[ \CZ_w | \CK_\beta, \breve{X}_w = \breve{Y}_w, \Phi_{r, 2R}(w)]}{\p[\CZ_w | \Phi_{r, 2R}(w)]} \p[\CK_\beta, \breve{X}_w = \breve{Y}_w | \Phi_{r,2R}(w)] \label{scaling::decomp}
\end{align}
It is an immediate consequence of Proposition \ref{ni::prop::ni} that 
\begin{equation}
\label{scaling::ni}
 \p[ \CK_\beta, \breve{X}_w = \breve{Y}_w | \Phi_{r, 2R}(w)] \asymp q(\zeta)
\end{equation}
for some function $q$ and $a \asymp b$ for $a,b > 0$ means that there exists a universal constant $C > 0$ such that $C^{-1} b \leq a \leq C b$; see \cite[Lemma 3.13]{SS09}.

We are now going to explain how the numerator in \eqref{scaling::decomp} can be estimated when the strands of $\beta, \breve{X}_w$, and $\Phi_{r,2R}(w)$ are \emph{well-separated}.  This means that the three points where $\ext_{r,2R}(\gamma^n;w)$ and $\ext_{r,2R}(X;w)$ enter $B(w,2Rn^{-1})$ are of distance $\epsilon R n^{-1}$, $\epsilon > 0$, from each other and likewise for the exit points of $\beta$ and $\breve{X}_w$ from $B(w, Rn^{-1})$.  Such configurations are said to be of \emph{high-quality}.  The hypothesis that our configurations are of high-quality allows for the application of barriers to show that the strands of $\ext_{r,2R}(\gamma^n;w)$ and $\beta$ connect with each other with u.p.p. and, using standard random walk estimates, $\ext_{r,2R}(X;w)$ hooks up with $\breve{X}_w$ without touching $\gamma^n$ with probability proportional to $(\log R)^{-1}$.  That is, 
\begin{equation}
\label{scaling::hq_hookup}
 \p[\CZ_w | \CK_\beta, \breve{X}_w = \breve{Y}_w, \Phi_{r,2R}(w)] \asymp \frac{1}{\log R}.
\end{equation}

Indeed, the reason for the latter is that the law of $X$ conditional on $\ext_{r,2R}(X;w)$ and $\breve{X}_w = \breve{Y}_w$ is that of the concatenation of a random walk $\wh{X}$ initialized at the first entrance point of $\ext_{r,2R}(X;w)$ to $\partial B(w,2R n^{-1})$ and stopped when it first hits $z_0$, the first exit point of $\breve{X}_w$ from $B(w, R n^{-1})$, along with some number $N$ of random walk excursions $\wh{X}^i$ from $z_0$ back to itself and $\ext_{r,2R}(X;w)$, $\breve{X}_w$.  It is easy to see that with u.p.p., $\wh{X}$ gets within distance $\tfrac{\epsilon}{100} R n^{-1}$ of $z_0$ without hitting the barriers nor $\gamma^n$ and, conditional on this, the probability that $\wh{X}$ hits $z_0$ before hitting $\gamma^n$ is proportional to $(\log R)^{-1}$.  It is also not difficult to see that $N$ is geometric with parameter proportional to $(\log R)^{-1}$ and the probability that a given $\wh{X}^i$ hits $\gamma^n$ before $z_0$ is again proportional to $(\log R)^{-1}$, so the two factors exactly cancel.  See Figure \ref{bvi::fig::hookup} for an illustration of this event and \cite[Lemma 3.14]{SS09} for a precise statement of this result in the case $r=0$.

  \begin{figure}
     \centering
         \includegraphics[height=.4\textwidth]{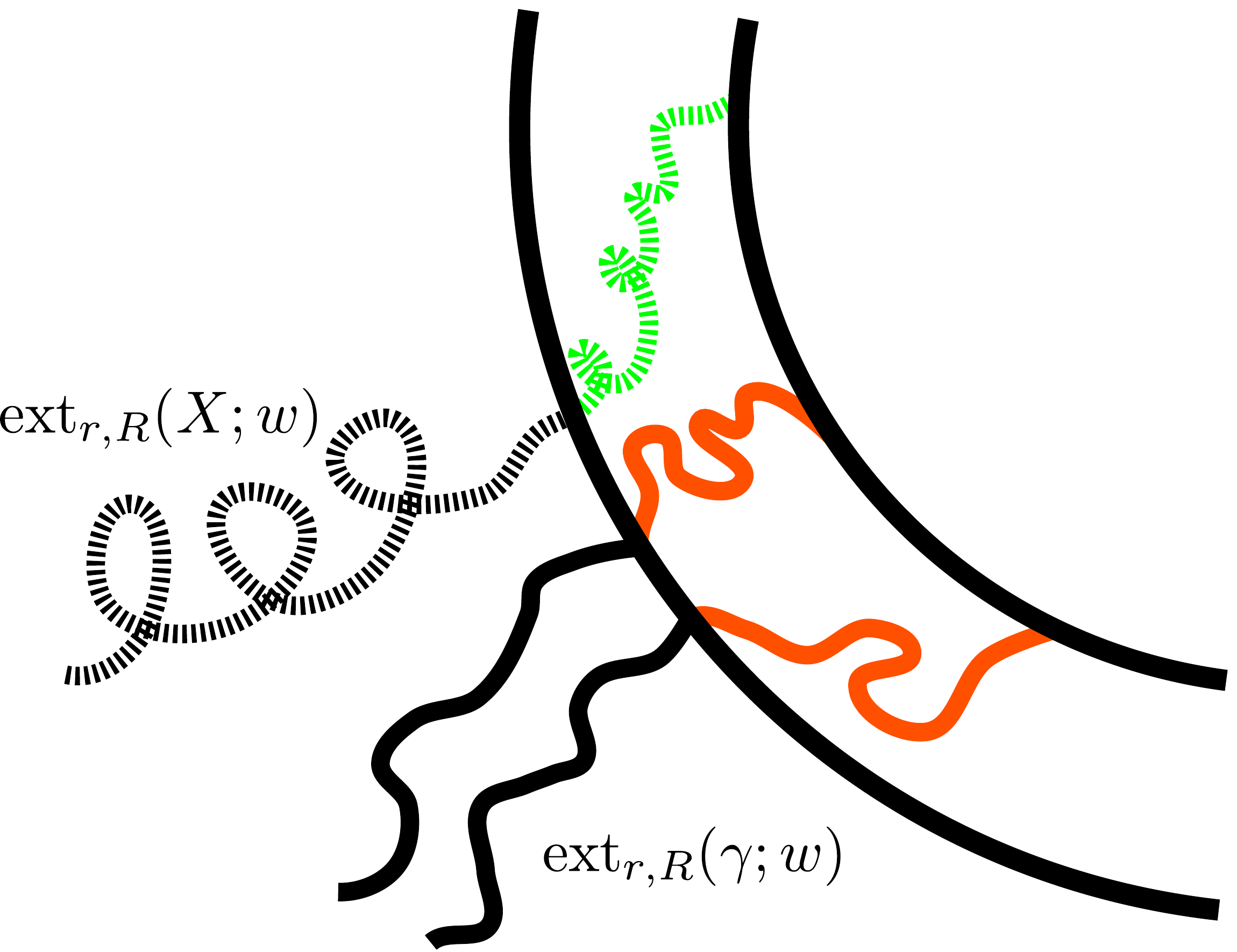}
          \caption{An illustration of the separation lemma, that poorly separated strands with positive probability become well-separated, in the special case of the external configuration.}
          \label{bvi::fig::separation}
 \end{figure}

In order for \eqref{scaling::hq_hookup} to be useful and also to estimate the denominator of \eqref{scaling::decomp}, we need that with u.p.p. high-quality configurations occur.  This is the purpose of the so-called ``separation lemma,'' the second important ingredient in the proof of Theorem \ref{intro::thm::scaling_limit}, which states $\Theta_{r,2R}$ and $\Phi_{r,3R}$ are of high-quality with u.p.p. conditional on $\Theta_{r,R}, \Phi_{r,4R}$ as well as $\CZ_w$, regardless of their quality.  See Figure \ref{bvi::fig::separation} for an illustration and as well as \cite[Lemma 3.15]{SS09} for the precise statement when $r = 0$.  The idea of the proof is to invoke the Barriers Theorem iteratively along with some random walk estimates to show that the strands tend to spread apart.  Exactly the same proof works for $r > 0$.
 
   \begin{figure}[h]
     \centering
     \subfigure[External configurations $\Phi_{r,2R}$ (black) and $\Phi_{r,2R}'$ (blue) and one internal configuration $\Theta_{r,R}$ (red).  The separation lemma combined with \eqref{scaling::decomp} implies that $\p\text{[}\Theta_{r,R} = \zeta | \CZ_w, \Phi_{r,2R}\text{]}$ is comparable to $\p\text{[}\Theta_{r,R} = \zeta | \CZ_w, \Phi_{r,2R}'\text{]}$.]{
         \includegraphics[width=.45\textwidth]{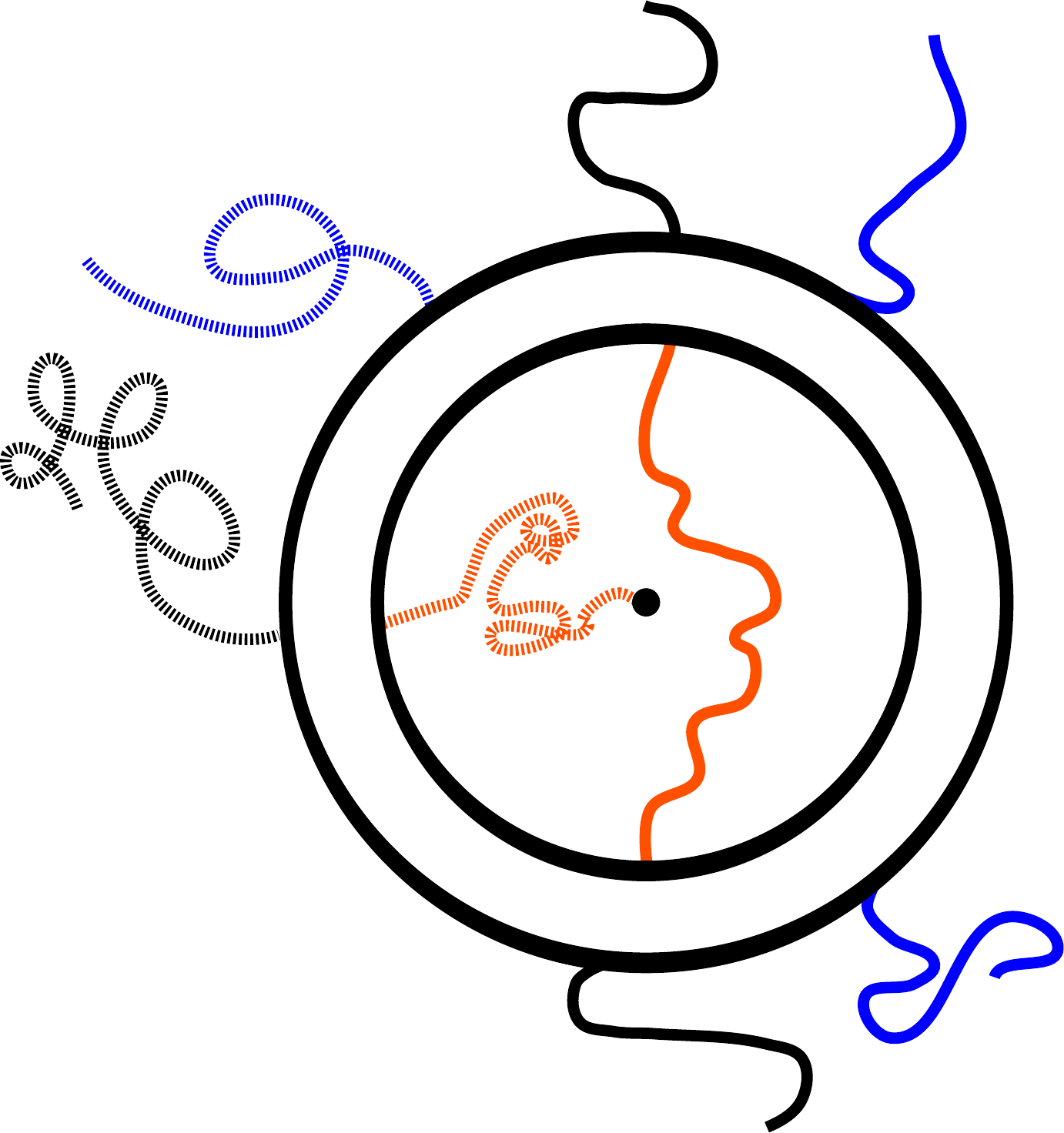}}
          \hspace{0.05in}
     \subfigure[If a step of the coupling argument is successful, then the external configurations in the next step agree in a large annulus, hence future steps in the coupling are more likely to be successful.]{
         \includegraphics[width=.45\textwidth]{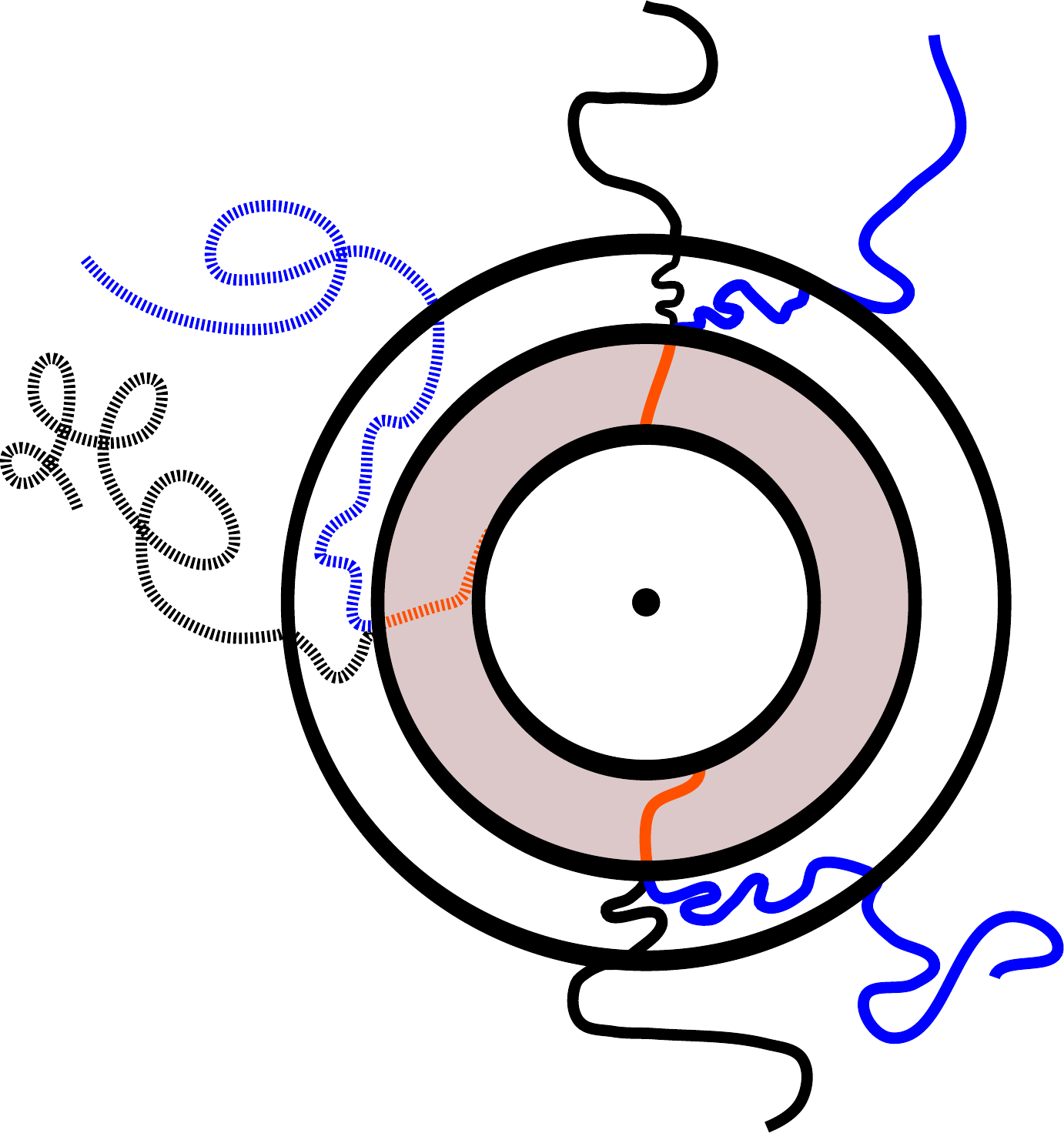}}
          \caption{A typical step in the coupling argument in the proof of Theorem \ref{intro::thm::scaling_limit}.}
          \label{bvi::fig::coupling}
 \end{figure}

Combining the separation lemma with \eqref{scaling::hq_hookup} implies that the estimate analogous to \eqref{scaling::hq_hookup} holds with arbitrary internal and external configurations.  This can also be used to give an estimate of the denominator of \eqref{scaling::decomp}.  Putting everything together thus implies the approximate independence of $\Theta_{r,R}$ from $\Phi_{r,2R}$:
\[ \p[ \Theta_{r,R} = \zeta | \CZ_w, \Phi_{r, 2R}] \asymp \p[ \Theta_{r,R} = \zeta | \CZ_w, \Phi_{r, 2R}']\]
uniformly in $\Phi_{r,2R}, \Phi_{r,2R}'$.  See \cite[Corollary 3.16]{SS09}, where  we again emphasize that Proposition \ref{ni::prop::ni} takes the role of \cite[Proposition 3.7]{SS09}.

Theorem \ref{intro::thm::scaling_limit} can now be proved using the following iterative coupling argument.  Fix $R_1$ very large and let $(R_k : k \geq 1)$ be a sequence decreasing appropriately quickly so that the previous lemmas always apply for the pairs $(r,R_k)$ and $(r,R_{k+1})$.  We shall assume that we are always conditioning on $\CZ_w$ where $w$ satisfies $\dist(w, \partial D) \geq 100 R_1 n^{-1}$.  Suppose we start with two arbitrary external configurations $\Phi_{r,R_1},\Phi_{r,R_1}'$, which we emphasize could come from different domains, boundary conditions, starting points of $X$, or all three.  We couple $\Theta_{r,R_2}|\Phi_{r,R_1}$, the conditional law of $\Theta_{r,R_2}$ given $\Phi_{r,R_1}$, and $\Theta_{r,R_2}'|\Phi_{r,R_1}$ to maximize the probability of success, that is $\Theta_{r,R_2} = \Theta_{r,R_2}'$.  In subsequent steps, we couple $\Theta_{r,R_{\ell+1}}|\Phi_{r,R_\ell}$ and $\Theta_{r,R_{\ell+1}}'|\Phi_{r,R_\ell}'$ to maximize the probability of success.  This probability is always uniformly positive regardless of whether or not previous stages of coupling have succeeded.  Thus we can make the probability that there is at least one successful coupling as close to $1$ as desired by choosing $R_1$ large enough to allow for sufficiently many steps of this procedure.  Conditional on the event that there is at least one success, the probability that the terminal internal configurations agree is also very close to $1$, depending on the rate at which the $R_\ell$ are decreasing, since then the external configurations at some step agree in a large annulus.  Rescaling by $n$, the theorem clearly follows.  A typical step of this procedure is illustrated in Figure \ref{bvi::fig::coupling}.  See Lemmas 3.19, 3.20, as well as Theorem 3.21 of \cite{SS09} for a proof when $r=0$.

Note that $\internal_{r,R}(\gamma^n) \cap B(0,\wt{r} n^{-1})$, $\wt{r} > r$ much smaller than $R$, is not necessarily the same as $\gamma^n \cap B(X_{\tau(r)}, \wt{r} n^{-1}) - X_{\tau(r)}$ since it could be that $\gamma^n$ makes multiple excursions from $B(X_{\tau(r)}, R n^{-1})$ to $B(X_{\tau(r)}, \wt{r} n^{-1})$.  This possibility is ruled out w.h.p. by \cite[Lemma 3.17]{SS09}.

The argument we have described thus far gives Theorem \ref{intro::thm::scaling_limit} in the special case that the $\CF_t^n$-stopping time $\tau$ is when $\gamma^n$ hits $y_n$.  We can repeat the same procedure for general $\tau$ provided we make the following modifications.  We let $\tau(r)$ be the first time that $X$ gets within distance $r$ of $\gamma^n|_{[0,\tau]}$ and change the definitions of $\Theta_{r,R}$ and $\Phi_{r,R}$ accordingly.  Exactly the same coupling procedure goes through provided $w$ is far from both $\partial D_n$ as well as the tip $\gamma^n(\tau)$, we just need to be sure that the result does not depend on $\tau$.  This, however, is not the case since it is not difficult to see that the coupling works even if the initial external configurations $\Phi_{r,R}, \Phi_{r,R'}$ arise from stopping $\gamma^n$ at different times $\tau,\tau'$.

 \subsection{Boundary Values}
 
  \begin{figure}
     \centering
         \includegraphics[width=.40\textwidth]{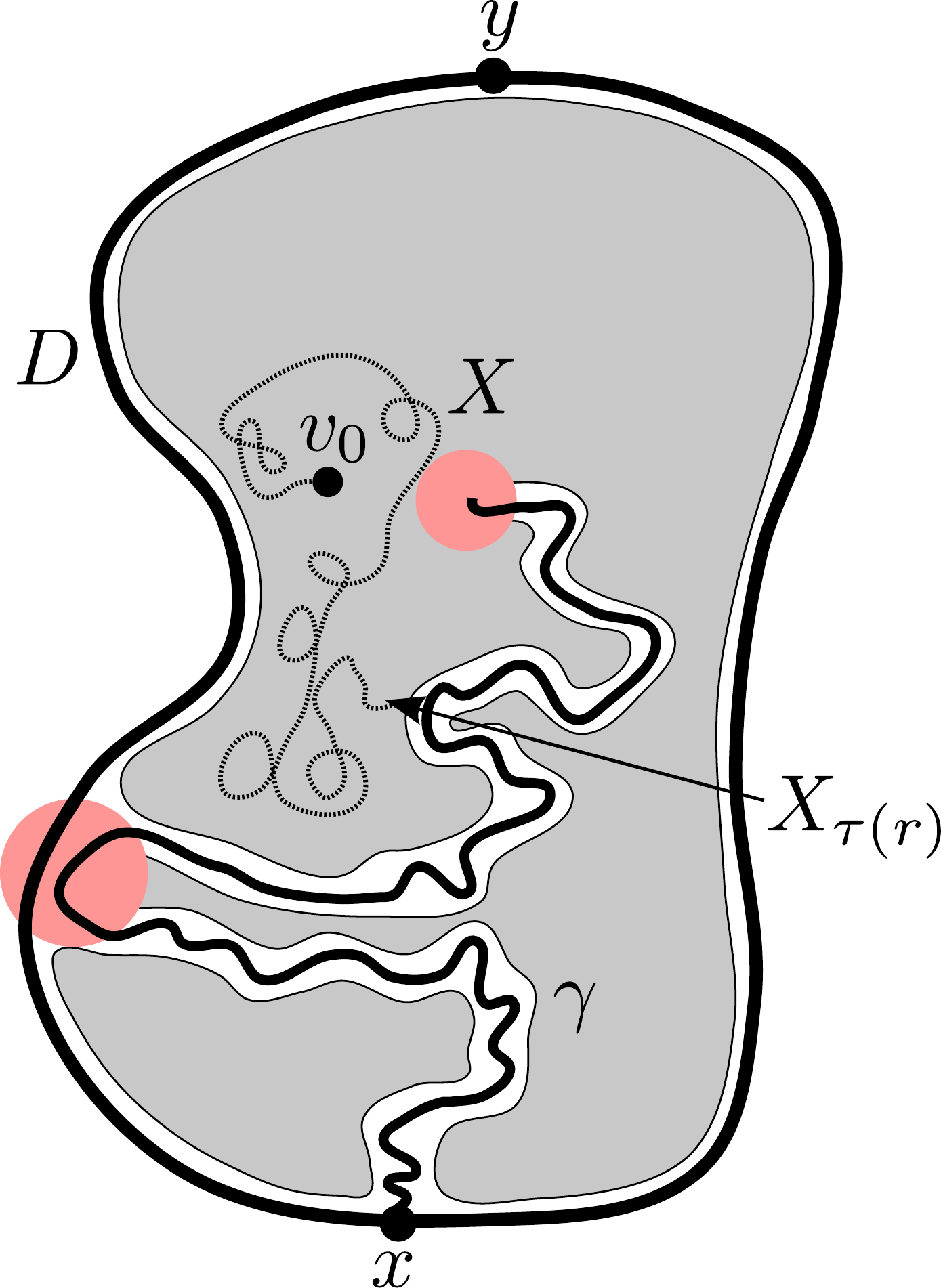}
          \caption{The local geometry of $\gamma^n$ as seen from $X_{\tau(r)}$ looks like $\gamma_r$, the bi-infinite path sampled from $\nu_r$ of Theorem \ref{intro::thm::scaling_limit}, provided $X_{\tau(r)}$ is sufficiently far from the tip of $\gamma^n$ and $\partial D_n$.  Thus it is important in the proof of Theorem \ref{intro::thm::approximate_martingale} to show that such regions have small harmonic measure.}
 \end{figure}

We will now explain how Theorem \ref{intro::thm::approximate_martingale} can be proved by combining Theorems \ref{intro::thm::harmonic_up_to_boundary} and \ref{intro::thm::scaling_limit}.  Fix $r \geq 0$ and let $\gamma_r$ be a bi-infinite path in $(\Z^2)^*$ sampled according to the probability $\nu_r$ of Theorem \ref{intro::thm::scaling_limit}.  Let $V_+(\gamma_r)$ be the set of vertices in $\Z^2$ which are adjacent to $\gamma_r$ and are contained in the same connected component of $\Z^2 \setminus \gamma_r$ as $0$ and $V_-(\gamma_r)$ the set of all other vertices adjacent to $\gamma_r$.  From Proposition \ref{hic::prop::hic}, it is clear that we can construct a random field $h_r \colon \Z^2 \to \R$ that given $\gamma_r$ has the law of the GL model on $\Z^2$ conditioned on the events 
\[ \bigcap_{x \in V_+(\gamma_r)} \{ h_r(x) > 0\}\  \text{ and } \bigcap_{x \in V_-(\gamma_r)} \{ h_r(x) < 0\}.\]
Let 
\begin{equation}
\label{bvi::eqn::lambda}
 \lambda_r = \E[h_r(0)].
\end{equation}

Let $\CF_t^n = \sigma(\gamma^n(s) : s \leq t)$ and let $\tau$ be any $\CF_t^n$-stopping time.  Fix a point $v_0 \in D_n(\gamma,\tau,\epsilon)$, let $X$ be a random walk initialized at $v_0$, and $\tau(r)$ the first time it gets within distance of $r n^{-1}$ of $\partial D_n$ or $\gamma^n[0,\tau]$.  Theorem \ref{intro::thm::scaling_limit} implies that the local geometry of $\gamma^n$ near $X_{\tau(r)}$ looks like $\gamma_r$ provided $X_{\tau(r)}$ lands on the positive side of $\gamma^n$ and is neither close to the tip of $\gamma^n[0,\tau]$ nor $\partial D_n$, which, in view of \cite[Lemmas 3.23 and 3.26]{SS09}, happens with low probability.  Letting $V_+^r(\gamma^n,\tau)$ (resp. $V_-^r(\gamma^n,\tau)$) be the set of vertices in $D_n$ with distance exactly $rn^{-1}$ from the positive (resp. negative) side of $\gamma^n[0,\tau]$,  it thus follows from Proposition \ref{hic::prop::hic} that 
\begin{equation}
\label{bvi::eqn::sample_mean}
 \E\big[ (h^n(X_{\tau(r)}) - \lambda_r) \one_{\{X_{\tau(r)} \in V_+^r(\gamma^n,\tau)\}}\big]
 \end{equation}
is small and similarly when $+$ and $-$ are swapped.  It is then possible to show that
\begin{equation}
\label{bvi::eqn::sample_mean_cond}
 \E\bigg[ \left( \E\big[ (h^n(X_{\tau(r)}) - \lambda_r) \one_{\{X_{\tau(r)} \in V_+^r(\gamma^n,\tau)\}} \big| \gamma^n[0,\tau] \big] \right)^2 \bigg]
 \end{equation}
is small by considering independent copies and arguing that the corresponding random walks are unlikely to hit $\gamma^n[0,\tau]$ close to each other, then invoke the approximate independent of internal and external configurations.  It is important here that the internal configuration of one random walk is contained in the external configuration of the other, but this happens with high probability by \cite[Lemma 3.17]{SS09}.  Theorem \ref{intro::thm::harmonic_up_to_boundary} then implies that, with $\psi_\tau^{n,r}$ the discrete harmonic function in $D_n(\gamma,\tau,rn^{-1})$ with boundary values $\pm \lambda_r$ on $V_\pm^r(\gamma^n,\tau)$ and the same boundary values as $h^n$ otherwise and $M_\tau^n(x) = \E[ h^n(x) | \CF_{\tau}^n]$, we have that
\begin{equation}
\label{bvi::eqn::harmonic_error}
 \E\big[ \max_{x \in D_n(\gamma,\tau,\epsilon)} |M_\tau^n(x) - \psi_\tau^{n,r}(x)|\big] \leq \delta
\end{equation}
for $\delta = \delta(r)$ with $\lim_{r \to \infty} \delta(r) = 0$.

To finish proving Theorem \ref{intro::thm::approximate_martingale}, it is left to show that the sequence $(\lambda_r)$ has a positive and finite limit $\lambda$:
\begin{lemma}
There exists $\lambda = \lambda(\CV) \in (0,\infty)$ such that
\[ \lim_{r \to \infty} \lambda_r = \lambda.\]
\end{lemma}
\begin{proof}
Suppose that $h_r, \gamma_r$ are as in the paragraph just above \eqref{bvi::eqn::lambda}.  By construction, $0$ is separated from $V_-(\gamma_r)$ by $V_+(\gamma_r)$.  Proposition \ref{hic::prop::exp_bounds} thus implies the existence of non-random $\lambda_0 = \lambda_0(\CV) > 0$ such that $\lambda_0^{-1} \leq \E[h_r(0) | \gamma_r] \leq \lambda_0$, hence also $\lambda_0^{-1} \leq \lambda_r \leq \lambda_0$.  Thus we just need to show that $(\lambda_r)$ is Cauchy, which in turn is an immediate consequence of \eqref{bvi::eqn::harmonic_error}.  Indeed, we first fix $r_1,r_2 > 0$ then $n$ very large so that \eqref{bvi::eqn::harmonic_error} holds for both $\psi^{n,r_1}$ and $\psi^{n,r_2}$ simultaneously.  We omit the subscript to indicate the functions corresponding to the entire path $\gamma^n$.  Applying the triangle inequality, we have
\begin{equation}
\label{bvi::eqn::harmonic}
 \E\big[ \max_{x \in D_n(\gamma,\epsilon)} |\psi^{n,r_1}(x) - \psi^{n,r_2}(x)| \big] \leq 2\delta(r_1) \vee \delta(r_2).
\end{equation}
Fix $x \in D_n(\gamma,\epsilon)$ very close to the positive side of $\gamma^n$ and away from $\partial D_n$.  We can always make such a choice so that $|\psi^{n,r_i}(x) - \lambda_{r_i}| \leq 2\delta(r_1) \vee \delta(r_2)$, which implies $|\lambda_{r_1} - \lambda_{r_2}| \leq 4\delta(r_1) \vee \delta(r_2)$.
\end{proof}

\subsection{The Barriers Theorem for the GL Model}
\label{subsec::nob}

We conclude by explaining how to redevelop the relevant parts of \cite[Subsections 3.3-34]{SS09} so that the proof of \cite[Theorem 3.11]{SS09} is applicable in the GL setting.  To keep the exposition compact, we will just indicate the necessary changes without repeating statements and proofs.

We begin with \cite[Lemma 3.8]{SS09}, ``Narrows.''  By the Brascamp-Lieb inequalities (Lemma \ref{bl::lem::bl_inequalities}), the the conditional variance upper bound \cite[Equation (3.25)]{SS09} also holds for the GL model.  The only claim that needs to be reproved is that $b \geq 0$ if $\delta > 0$ is sufficiently small since, while the inequality $\E[ h(u) | \CK] \geq c^{-1}$ for $u \in V_+$ does hold, we do not have the linearity of the mean in its boundary values.  Nevertheless, we get the desired result by invoking Proposition \ref{hic::prop::exp_bounds}.  The rest of the proof is exactly the same.  The ``Domain boundary narrows,'' \cite[Lemma 3.9]{SS09}, goes through with the same modifications.

We now turn to \cite[Lemma 3.10]{SS09}, ``Obstacle.''  One of the ingredients of the proof is equation (3.5) from \cite[Lemma 3.1]{SS09}.  This can be established in the GL setting by writing the second moment as the sum of the variance and the square of the mean, bounding the former using the Brascamp-Lieb inequality and the corresponding bound in \cite{SS09}, then controlling the mean using Lemma \ref{hic::lem::exp_upper_bound}.  Note that while our setting does not satisfy the hypotheses of Lemma \ref{hic::lem::exp_upper_bound}, the proof is in fact very general and still applies here.  The proof in \cite{SS09} breaks down at (3.29) since we do not have the exact harmonicity of the mean.  However, using Lemma \ref{hic::lem::harmonic_boundary} we can replace \cite[Equation (3.29)]{SS09} with 
\[ \E\big[ h(x) | \CK, \CQ, \beta \big] \leq O_{\ol{\Lambda}}(1) + \frac{\|g\|_\infty}{100} - \sum_{u \in U'} p_u g(u).\]
Since we may assume without loss of generality that $\| g\|_\infty$ is larger than any fixed constant, we can replace the above with
\[ \E\big[ h(x) | \CK, \CQ, \beta \big] \leq \frac{\|g\|_\infty}{50} - \sum_{u \in U'} p_u g(u),\]
from which the rest of the proof goes through without any changes.

The remaining part of the proof of \cite[Theorem 3.11]{SS09} that needs modification is the application of \cite[Lemma 3.6]{SS09}, for which we offer the following substitute.  Note that this result is in fact different from \cite[Lemma 3.6]{SS09} since we work with entropies as there is no analog of the Cameron-Martin formula for the GL model.  Recall that $\Q_D^{\psi,g}$ is the law of $h^\psi - g$ for $h^\psi \sim \p_D^\psi$.

\begin{lemma}
\label{nob::lem::distortion}
Let $D \subseteq \Z^2$ be bounded and let $g \colon D \to \R$ satisfy $g = 0$ on $\partial D$.  Then
\[ \h(\p_D^\psi|\Q_D^{\psi,g}) + \h(\Q_D^{\psi,g}|\p_D^\psi) \leq C \sum_{b \in D^*} |\nabla g(b)|^2\]
for $C = C(\CV)$.  In particular, if $A$ is any event then
\begin{align*}
   \exp\left(-\frac{C \sum_{b \in D^*} |\nabla g(b)|^2 + e^{-1}}{\Q^{\psi,g}[A]}\right) \leq
   \frac{\p^\psi[A]}{\Q^{\psi,g}[A]} \leq 
   \exp\left(\frac{C \sum_{b \in D^*} |\nabla g(b)|^2 + e^{-1}}{\p^{\psi}[A]}\right) 
\end{align*}
\end{lemma}
\begin{proof}
The latter claim is an immediate consequence of the first part of the lemma, the non-negativity of the entropy, and the entropy inequality (see the proof of \cite[Lemma 5.4.21]{DS89}) 
\[ \log\left( \frac{\mu(A)}{\nu(A)} \right) \geq - \frac{\h(\mu|\nu) + e^{-1}}{\nu(A)}.\]
Using a proof similar to Lemma \ref{harm::lem::entropy_form} of \cite{M10}, we have
\begin{align*}
  &  \h(\Q_D^{\psi,g}| \p_D^\psi) + \h(\p_D^{\psi}|\Q_D^{\psi,g})\\
=& \sum_{b \in D^*} \E^{\psi} \left(  \int_0^1 \CV'(\nabla (h + s g)(b)) ds - \int_0^1 \CV'(\nabla [(h + (s-1)g) ](b))ds\right) \nabla g(b)\\
=& \sum_{b \in D^*} \left( \E^{\psi} \left(  \int_0^1 \CV'(\nabla h(b)) ds - \int_0^1 \CV'(\nabla h(b))ds\right) \nabla g(b) + O( (\nabla g(b))^2) \right)\\
=& \sum_{b \in D^*} O( (\nabla g(b))^2).
\end{align*}
\end{proof}

Using the same notation as the proof of \cite{SS09} in the paragraph after equation \cite[Equation (3.32)]{SS09}, we get the same bound
\[ \p\bigg[ \p[ \wh{\gamma}_g \cap (\cup Y) = \emptyset | \CK, h_U]  > 1/10 \big| \CK \bigg] > 1/10.\]
Suppose that $h_U$ is chosen so that the inner inequality holds.  Invoking Lemma \ref{nob::lem::distortion} then implies
\begin{align*}
 &\p[ \wh{\gamma} \cap (\cup Y) = \emptyset| \CK, h_U] = 
  \Q^{g}[\wh{\gamma}_g \cap (\cup Y) = \emptyset | \CK, h_U]\\
\geq& \frac{1}{10} \exp\left( - 10 C \sum_{b \in D^*} |\nabla g(b)|^2 - 10e^{-1} \right)
  \geq \rho > 0
\end{align*}
where $\rho$ depends only on $\epsilon, \ol{\Lambda},m$.  This is equivalent to the statement in \cite{SS09} that
\[ O_{\epsilon,\ol{\Lambda},m}(1) \p[ \wh{\gamma} \cap (\cup Y) = \emptyset | \CK, h_U] \geq 1.\]

\section*{Acknowledgements}  I thank Amir Dembo and Scott Sheffield for very helpful comments on an earlier version of this manuscript, which led to many significant improvements in the exposition.

\bibliographystyle{acmtrans-ims.bst}
\bibliography{gl_level_lines}

\end{document}